\documentclass[a4paper,10pt]{amsart}
\usepackage{graphicx, amssymb, amsmath, amscd}
\usepackage[arrow,tips,matrix]{xy}
\usepackage{pst-node, hhline}
\usepackage{subfigure}

\theoremstyle{plain}
\newtheorem{theorem}{Theorem}[section]
\newtheorem{main}{Theorem} 
\newtheorem{proposition}[theorem]{Proposition}
\newtheorem{lemma}[theorem]{Lemma}
\newtheorem{corollary}[theorem]{Corollary}

\theoremstyle{definition}

\newtheorem{definition}[theorem]{Definition}
\newtheorem{notation}[theorem]{Notation}
\newtheorem{example}{Example}

\theoremstyle{remark}
\newtheorem{remark}[theorem]{Remark}

\numberwithin{table}{section} \numberwithin{figure}{section}
\numberwithin{equation}{section} \numberwithin{example}{section}

\DeclareMathOperator{\id}{id} \DeclareMathOperator{\D}{D}
\DeclareMathOperator{\GL}{GL} \DeclareMathOperator{\orthogonal}{O}
\DeclareMathOperator{\SO}{SO} 
\DeclareMathOperator{\res}{res} \DeclareMathOperator{\Iso}{Iso}
 \DeclareMathOperator{\Irr}{Irr}
\DeclareMathOperator{\Vect}{Vect} \DeclareMathOperator{\Rep}{Rep}
\DeclareMathOperator{\ind}{ind} \DeclareMathOperator{\Hom}{Hom}

\DeclareMathOperator{\ext}{ext} 

 \DeclareMathOperator{\pr}{pr}
\DeclareMathOperator{\modulo}{mod}

\newcommand{\complex}[1]{\mathcal{#1}}
\newcommand{\complexK}{\complex{K}}
\newcommand{\complexL}{\complex{L}}

\newcommand{\lineK}{{\bar{\complexK}}}
\newcommand{\lineL}{{\bar{\complexL}}}
\newcommand{\hatL}{{\hat{\complexL}}}

\newcommand{\rot}{\mathrm{rot}}
\newcommand{\tetra}{\mathrm{T}}
\newcommand{\octa}{\mathrm{O}}
\newcommand{\icosa}{\mathrm{I}}
\newcommand{\vect}{\mathrm{vect}}

\newcommand{\field}[1]{\mathbb{#1}}
\newcommand{\C}{\field{C}}
\newcommand{\R}{\field{R}}
\newcommand{\Z}{\field{Z}}
\newcommand{\N}{\field{N}}

\begin{document}

\title[equivariant vector bundles over two-sphere]
{Classification of equivariant vector bundles over two-sphere}

\author{Min Kyu Kim}
\address{Department of Mathematics Education,
Gyeongin National University of Education, San 59-12, Gyesan-dong,
Gyeyang-gu, Incheon, 407-753, Korea}

\email{mkkim@kias.re.kr}



\maketitle

\begin{abstract}
We classify equivariant topological complex vector bundles over
two-sphere under a compact Lie group (not necessarily effective)
action. It is shown that nonequivariant Chern classes and isotropy
representations at (at most) three points are sufficient to classify
equivariant vector bundles except a few cases. To do it, we
calculate homotopy of the set of equivariant clutching maps. In
other papers, we also give classifications over two-torus, real
projective plane, Klein bottle.
\end{abstract}

\section{Introduction} \label{section: introduction}
In topology, nonequivariant complex vector bundles can be classified
just by calculating their Chern classes under a dimension condition
\cite{P}. However, there is no such general result on equivariant
vector bundles. Instead, a few results on extreme cases are known.
Let us mention four of them. Let a compact Lie group $G$ act on a
topological space $X.$
\begin{itemize}
  \item If $G$ is trivial, then $G$-vector bundles over $X$ are just
  nonequivariant vector bundles, and are classified by their
  Chern classes in $H^* (X).$
  \item In \cite[Proposition 1.6.1]{At}, \cite[p. 132]{S},
  $G$-vector bundles over a free $G$-space $X$ are
  in one-to-one correspondence with nonequivariant
  vector bundles over $X/G,$ and are classified by Chern classes
  in $H^* (X/G).$
  \item For a closed subgroup $H$ of $G$ and any point $x$ in $X=G/H,$
  $G$-vector bundles over $X$ are classified by their
  isotropy representations at $x$ which are contained in $\Rep(G_x).$
  For this, see \cite[p. 130]{S}, \cite[Proposition II.3.2]{B}.
  In \cite{CKMS}, $G$-vector bundles
  over $X=S^1$ are classified by their isotropy representations at
  (at most) two points.
\end{itemize}
Since invariants live in $H^* (X),$ $H^* (X/G),$ $\Rep (G_x),$ these
results might be considered to have three different types. In this
paper, we classify equivariant topological complex vector bundles
over two-sphere under a compact Lie group (not necessarily
effective) action. Readers will see those different types at once in
it. By developing ideas and machineries of it, classification over
two-torus is given in \cite{Ki1} which shows that our method is not
restrictive. Classifications over real projective plane and Klein
bottle are also obtained as corollaries in \cite{Ki2}, \cite{Ki3}.
After these, classification of equivariant holomorphic vector
bundles over Riemann sphere under a holomorphic complex reductive
group (not necessarily effective) action will follow which is an
equivariant version of Grothendieck's Theorem for holomorphic vector
bundles on Riemann sphere \cite{G}, and is the motivation for the
paper.

To state main results, we need introduce several notations. It is
well-known that a topological action on $S^2$ by a compact Lie group
is conjugate to a linear action \cite[Theorem 1.2]{Ko}, \cite{CK}.
Let a compact Lie group $G$ act linearly (not necessarily
effectively) on the unit sphere $S^2$ in $\R^3$ through a
representation $\rho : G \rightarrow \orthogonal(3).$ Let $\Vect_G
(S^2)$ be the set of isomorphism classes of topological complex
$G$-vector bundles over $S^2$ with the given $G$-action. For a
bundle $E$ in $\Vect_G (S^2)$ and a point $x$ in $S^2,$ denote by
$E_x$ the isotropy $G_x$-representation on the fiber at $x.$ It is
needed to decompose $\Vect_G (S^2)$ as the sum of more simple
subsemigroups. For this, some terminologies are introduced. Put $H =
\ker \rho,$ i.e. the kernel of the $G$-action on $S^2.$ Let
$\Irr(H)$ be the set of characters of irreducible complex
$H$-representations which has a $G$-action defined as
\begin{equation*}
(g \cdot \chi) (h) = \chi (g^{-1} h g)
\end{equation*}
for $\chi \in \Irr(H),$ $g \in G,$ $h \in H.$ Sometimes, we also use
the notation $\Irr(H)$ to denote the set of isomorphism classes of
irreducible complex $H$-representations themselves. For $\chi \in
\Irr(H),$ an $H$-representation is called $\chi$-isotypical if its
character is a multiple of $\chi.$ We slightly generalize this
concept. For $\chi \in \Irr(H)$ and a compact Lie group $K$
satisfying $H \lhd K < G$ and $K \cdot \chi = \chi,$ a
$K$-representation $W$ is called $\chi$-\textit{isotypical} if
$\res_H^K W$ is $\chi$-isotypical, and denote by $\Vect_K (S^2,
\chi)$ the set
\begin{equation*}
\Big\{ ~ [E] \in \Vect_K (S^2) ~ \big| ~ E_x \text{ is }
\chi\text{-isotypical for each } x \in S^2 ~ \Big\}
\end{equation*}
where $S^2$ delivers the restricted $K$-action. In \cite{CKMS}, the
(\textit{isotypical}) \textit{decomposition} of a $G$-bundle is
defined, and from this a semigroup isomorphism is constructed to
satisfy
\begin{equation*}
\Vect_G (S^2) \cong \bigoplus_{\chi \in \Irr(H)/G} \Vect_{G_\chi}
(S^2, \chi)
\end{equation*}
where $G_\chi$ is the isotropy subgroup of $G$ at $\chi.$ As a
result, our classification is reduced to $\Vect_{G_\chi} (S^2,
\chi)$ for each $\chi \in \Irr(H).$ Details are found in
\cite[Section 2]{CKMS}.

The classification of $\Vect_{G_\chi} (S^2, \chi)$ is highly
dependent on the $G_\chi$-action on the base space $S^2,$ i.e. on
the image $\rho(G_\chi),$ and classification is actually given case
by case according to $\rho(G_\chi).$ So, we need to describe
$\rho(G_\chi)$ in a moderate way. For this, we would list all
possible $\rho(G_\chi)$'s up to conjugacy, and then assign an
equivariant simplicial complex structure of $S^2$ to each finite
$\rho(G_\chi).$ Cases of nonzero-dimensional $\rho(G_\chi)$ are
relatively simple and separately dealt with as special cases. First,
let us define some polyhedra. Let $P_m$ for $m \ge 3$ be the regular
$m$-gon on $xy$-plane in $\R^3$ whose center is the origin and one
of whose vertices is $(1,0,0).$ Then,
\begin{enumerate}
  \item $| \complexK_m |$ is defined as the boundary
  of the convex hull of $P_m,$ $S=(0,0,-1),$
  $N=(0,0,1),$
  \item $| \complexK_\tetra |$ is defined as
  the tetrahedron which is the boundary of the
  convex hull of four points
  $(\frac {1} {3}, \frac {1} {3}, \frac {1} {3}),$
  $(-\frac {1} {3}, -\frac {1} {3}, \frac {1} {3}),$
  $(-\frac {1} {3}, \frac {1} {3}, -\frac {1} {3}),$
  $(\frac {1} {3}, -\frac {1} {3}, -\frac {1} {3}),$
  and which is inscribed to $| \complexK_4 |,$
  \item $| \complexK_\icosa |$ is defined as
  an icosahedron which has the origin as the center.
\end{enumerate}
With these, denote natural simplicial complex structures on $|
\complexK_m |,$ $| \complexK_\tetra |,$ $| \complexK_\icosa |$ by
$\complexK_m,$ $\complexK_\tetra,$ $\complexK_\icosa,$ respectively.
Then, it is well-known that each closed subgroup of $\SO(3)$ is
conjugate to one of the following subgroups \cite[Theorem 11]{R}:
\begin{enumerate}
  \item $\Z_n$ generated by the rotation $a_n$ through the angle
  $2\pi/n$ around $z$-axis,
  \item $\D_n$ generated by $a_n$ and the rotation $b$ through
  the angle $\pi$ around $x$-axis,
  \item the tetrahedral group $\tetra$ which is the rotation group
  of $| \complexK_\tetra |,$
  \item the octahedral group $\octa$ which is the rotation group
  of $| \complexK_4 |,$
  \item the icosahedral group $\icosa$ which is the rotation group
  of $| \complexK_\icosa |,$
  \item $\SO(2)$ which is the set of rotations around $z$-axis,
  \item $\orthogonal(2)$ which is defined as $\langle \SO(2), b \rangle,$
  \item $\SO(3)$ itself.
\end{enumerate}
Note that $\tetra \subset \octa,$ and pick an element $o_0$ of
$\octa \setminus \tetra$ so that $\octa = \langle \tetra, o_0
\rangle.$ And, denote by $Z$ the centralizer $\{ \id, -\id \}$ of
$\orthogonal(3).$ In Section \ref{section: closed subgroups}, it is
shown that each closed subgroup of $\orthogonal(3)$ is conjugate to
an $R$-entry of Table \ref{table: introduction} (there is no
literature on this as far as the author knows). In the table, the
notation $\times$ means internal direct product of two subgroups in
$\orthogonal(3).$ Henceforward, it is assumed that $\rho(G_\chi) =
R$ for some $R.$ Let $\orthogonal(3)$ and their subgroups act
naturally on $\R^3.$ To each finite $R,$ we assign a simplicial
complex $\complexK_R$ of Table \ref{table: introduction} where
$\complexK_\octa$ is defined in the below. Each $|\complexK_R|$ is
invariant under the natural $R$-action on $\R^3$ as shown in Section
\ref{section: simplicial complex}, and is also invariant under the
$G_\chi$-action defined through $\rho.$ So, we assume that
$|\complexK_R|$ delivers the $G_\chi$-action, and that $\complexK_R$
delivers the $G_\chi$-action inherited from it. Henceforward, we
consider $|\complexK_R|$ as the base space instead of the usual
two-sphere $S^2$ when $R$ is finite.

\begin{table}[ht!]
\begin{center}
{\footnotesize
\begin{tabular}{l||c|c|c}
$R$                                             & $\complexK_R$        & $D_R$                                & $d^{-1}$  \\

\hhline{=#=|=|=}

$\D_n,$ $n > 1$                                 & $\complexK_n$        & $[v^0, b(e^0)]$                      & $S$     \\
$\Z_n$                                          & $\complexK_n$        & $|e^0|$                              & $S$     \\

\hline

$\D_n \times Z,$ odd $n$                        & $\complexK_{2n}$     & $[v^0, b(e^0)]$                      & $S$    \\
$\langle a_n, -b \rangle,$ odd $n$              & $\complexK_{2n}$     & $[b(e^0), v^1] \cup [v^1, b(e^1)]$   & $S$    \\
$\Z_n \times Z,$ odd $n$                        & $\complexK_{2n}$     & $|e^0|$                              & $S$     \\

\hline

$\D_n \times Z,$ even $n$                       & $\complexK_n$        & $[v^0, b(e^0)]$                      & $S$   \\
$\langle a_n, -b \rangle,$ even $n$             & $\complexK_n$        & $[v^0, b(e^0)]$                      & $S$   \\
$\langle -a_n, b \rangle,$ odd $n/2,$ $n > 2$   & $\complexK_{n/2}$    & $|e^0|$                              & $S$   \\
$\langle -a_n, b \rangle,$ even $n/2$           & $\complexK_n$        & $[v^0, b(e^0)]$                      & $S$   \\
$\langle -a_n, -b \rangle,$ odd $n/2,$ $n>2$    & $\complexK_n$        & $[b(e^0), v^1] \cup [v^1, b(e^1)]$   & $S$   \\
$\langle -a_n, -b \rangle,$ even $n/2$          & $\complexK_n$        & $[v^0, b(e^0)]$                      & $S$   \\
$\Z_n \times Z,$ even $n,$ $n>2$                & $\complexK_n$        & $|e^0|$                              & $S$   \\
$\langle -a_n \rangle,$ odd $n/2,$ $n>2$        & $\complexK_{n/2}$    & $|e^0|$                              & $S$   \\
$\langle -a_n \rangle,$ even $n/2$              & $\complexK_n$        & $|e^0|$                              & $S$   \\

\hline

$\tetra$                                        & $\complexK_\tetra$   & $[v^0, b(e^0)]$                      & $b(f^{-1})$     \\
$\octa$                                         & $\complexK_\octa$    & $[v^0, b(e^0)]$                      & $b(f^{-1})$     \\
$\icosa$                                        & $\complexK_\icosa$   & $[v^0, b(e^0)]$                      & $b(f^{-1})$     \\

\hline

$\langle \tetra, -o_0 \rangle$                  & $\complexK_\tetra$   & $[v^0, b(e^0)]$                      & $b(f^{-1})$     \\
$\tetra \times Z$                               & $\complexK_\octa$    & $|e^0|$                              & $b(f^{-1})$     \\
$\octa \times Z$                                & $\complexK_\octa$    & $[v^0, b(e^0)]$                      & $b(f^{-1})$     \\
$\icosa \times Z$                               & $\complexK_\icosa$   & $[v^0, b(e^0)]$                      & $b(f^{-1})$     \\

\hline

$\orthogonal(3)$                                &                      & $\{ v^0 \}$                          & $v^0$   \\
$\orthogonal(2) \times Z$                       &                      & $\{ v^0 \}$                          & $S$     \\
$\langle \SO(2), -b \rangle$                    &                      & $\{ v^0 \}$                          & $S$     \\
$\langle \SO(2), -a_2 \rangle$                  &                      & $\{ v^0 \}$                          & $S$     \\
$\SO(3)$                                        &                      & $\{ v^0 \}$                          & $v^0$   \\
$\orthogonal(2)$                                &                      & $\{ v^0 \}$                          & $S$     \\
$\SO(2)$                                        &                      & $\{ v^0 \}$                          & $S$     \\

\end{tabular}}
\caption{\label{table: introduction} $\complexK_R,$ $D_R,$ $d^{-1}$
for closed subgroup $R$}
\end{center}
\end{table}

In dealing with equivariant vector bundles over two-sphere, we need
to consider isotropy representations at a few points (at most three
points) of $|\complexK_R|.$ To specify those points, we introduce
some more notations. When $m \ge 3,$ denote by $v^i$ the vertex
$\exp \big( \frac {2 \pi i \sqrt{-1} ~} m \big) \in \R^2$ of
$\complexK_m,$ and by $e^i$ the edge of $\complexK_m$ connecting
$v^i$ and $v^{i+1}$ for $i \in \Z_m.$ These notations are
illustrated in Figure \ref{figure: simple form K_4}. When we use the
notation $\Z_m$ to denote an index set, it is just the group $\Z /
m\Z$ of integers modulo $m.$ In Section \ref{section: simplicial
complex}, $\complexK_m,$ $v^i,$ $e^i$ for $m=1,$ $2$ are also
defined. We would define similar notations for $\complexK_\tetra,$
$\complexK_\icosa.$ For $\complexK_\tetra,$ $\complexK_\icosa,$ pick
two adjacent faces in each case, and call them $f^{-1}$ and $f^0.$
And, label vertices of $f^{-1}$ as $v^i$ for $i \in \Z_3$ to satisfy
\begin{enumerate}
  \item $v^0,$ $v^1,$ $v^2$ are arranged
  in the clockwise way around $f^{-1},$
  \item $v^0,$ $v^1$ are contained in $f^{-1} \cap f^0.$
\end{enumerate}
For $i \in \Z_3,$ denote by $e^i$ be the edge connecting $v^i$ and
$v^{i+1},$ and by $f^i$ the face which is adjacent to $f^{-1}$ and
contains the edge $e^i.$ We distinguish the superscripts $-1$ and
$2$ only for $f^i,$ i.e. $f^{-1} \ne f^2$ in contrast to $v^{-1} =
v^2,$ $e^{-1} = e^2.$ These notations are illustrated in Figure
\ref{figure: example of notation}.(a). Here, we define one more
simplicial complex denoted by $\complexK_\octa$ which is the same
simplicial complex with $\complexK_4$ but has the same convention of
notations $v^i, e^i, f^{-1}, f^i$ with $\complexK_\tetra,$
$\complexK_\icosa.$ Also, put $|\complexK_\octa|=|\complexK_4|.$
These notations are illustrated in Figure \ref{figure: simple form
TxZ}.(a). With these notations, we explain for $D_R$-entry of Table
\ref{table: introduction}. To each finite $R,$ we assign a path
$D_R$ (called the (closed) \textit{one-dimensional fundamental
domain}) in $|\complexK_R|$ which is listed in the third column of
Table \ref{table: introduction} where $b(\sigma)$ is the barycenter
of $\sigma$ for any simplex $\sigma$ and $[x,y]$ is the shortest
path in $|\complexK|$ for any simplicial complex $\complexK$ and two
points $x,y$ in $|\complexK|.$ And, let $d^0$ and $d^1$ be boundary
points of $D_R$ such that $d^0$ is nearer to $v^0$ than $d^1.$ Here,
we define one more point $d^{-1}$ for each finite $R$ which is
listed in the fourth column in Table \ref{table: introduction}. If
$R$ is one-dimensional, then denote by $D_R$ the one point set $\{
v^0 = (1,0,0) \},$  and let $d^{-1},$ $d^0,$ $d^1$ be equal to $S,$
$v^0,$ $v^0,$ respectively. Similarly, if $R$ is three-dimensional,
then denote by $D_R$ the one point set $\{ v^0 = (1,0,0) \},$ and
let $d^{-1},$ $d^0,$ $d^1$ be all equal to $v^0.$ So far, we have
defined $d^{-1},$ $d^0,$ $d^1$ for each $R.$ Then, $\{ S, N \}$ or
$\{ d^{-1},$ $d^0,$ $d^1 \}$ are wanted points according to $R,$ and
we will consider the restriction $E|_{\{ S, ~ N \}}$ or $E|_{\{
d^{-1}, ~ d^0, ~ d^1 \}}$ for each $E$ in $\Vect_{G_\chi} (S^2,
\chi).$ Define the following semigroup which will be shown to be
equal to the set of all the restrictions:

\begin{definition} \label{definition: A_R}
For $\chi \in \Irr(H),$ assume that $\rho(G_\chi) = R$ for some $R$
of Table \ref{table: introduction}.
\begin{enumerate}
  \item If $R = \Z_n,$ $\langle a_n, -b \rangle,$ $\SO(2),$
  $\langle \SO(2), -b \rangle,$ then let
  $A_{G_\chi} ( S^2, \chi )$ be the semigroup of pairs
  $(W_S, W_N)$ in $\Rep( G_\chi )^2$
  satisfying
  \begin{enumerate}
    \item[i)] $W_S$ is $\chi$-isotypical,
    \item[ii)] $\res_{(G_\chi)_x}^{G_\chi} W_S \cong
              \res_{(G_\chi)_x}^{G_\chi} W_N$
              for $x = d^0, d^1.$
  \end{enumerate}
  And, let $p_{\vect} : \Vect_{G_\chi} (S^2, \chi) \rightarrow
A_{G_\chi} ( S^2, \chi )$ be the semigroup homomorphism defined as
$[E] \mapsto ( E_S, E_N ).$
  \item Otherwise, let
  $A_{G_\chi} ( S^2, \chi )$ be the semigroup of
  triples $( W_{d^{-1}}, W_{d^0}, W_{d^1} )$ in $\Rep \big(
  (G_\chi)_{d^{-1}} \big) \times \Rep \big( (G_\chi)_{d^0} \big)
  \times \Rep \big( (G_\chi)_{d^1} \big)$ satisfying
  \begin{enumerate}
    \item[i)] $W_{d^{-1}}$ is $\chi$-isotypical,
    \item[ii)] $W_{d^1} \cong ~ ^g W_{d^0}$ if there
          exists $g \in G_\chi$
          such that $g d^0 = d^1,$
    \item[iii)] for any two points $x, x^\prime$ of three points $d^{-1},$ $d^0,$ $d^1,$
          \begin{equation*}
          \res_{(G_\chi)_x \cap (G_\chi)_{x^\prime}}^{(G_\chi)_x} W_x \cong
          \res_{(G_\chi)_x \cap (G_\chi)_{x^\prime}}^{(G_\chi)_{x^\prime}}
          W_{x^\prime}.
          \end{equation*}
\end{enumerate}
And, let $p_{\vect} : \Vect_{G_\chi} (S^2, \chi) \rightarrow
A_{G_\chi} ( S^2, \chi )$ be the semigroup homomorphism defined as
$[E] \mapsto ( E_{d^{-1}}, E_{d^0}, E_{d^1} ).$
\end{enumerate}
\end{definition}
See Definition \ref{definition: conjugate representation} for the
superscript $g.$ Well-definedness of $p_{\vect}$ is proved in Lemma
\ref{lemma: elementary lemma on isotropy representation} and Lemma
\ref{lemma: restricted isotropy representation}. Put $I= \{ 0, 1 \}$
and $I^+= \{ -1, 0, 1 \}.$ And, denote a triple $( W_{d^{-1}},
W_{d^0}, W_{d^1} )$ by $(W_{d^i})_{i \in I^+}.$ Especially, if
$\rho(G_\chi) = \SO(3),$ $\orthogonal(3),$ then $A_{G_\chi} ( S^2,
\chi )$ is equal to the set
\begin{equation*}
\Big\{ ~ (W_{d^i})_{i \in I^+} \in \Rep \Big( (G_\chi)_{v_0} \Big)^3
~ \Big| ~ W_{d^i}\text{'s are the same } \chi \text{-isotypical
representation} \Big\},
\end{equation*}
and $p_\vect$ becomes an isomorphism by classification over
homogeneous space.

Now, we can state main results. Let $c_1 : \Vect_{G_\chi} (S^2,
\chi) \rightarrow H^2 (S^2)$ be the map defined as $[E] \mapsto c_1
(E).$ Denote by $l_R$ the number $|G_\chi| / |G_\chi(D_R)|$ where
$G_\chi(D_R)$ is the subgroup of $G_\chi$ preserving $D_R.$

\begin{main} \label{main: by isotropy and chern}
Assume that $\rho(G_\chi) = R$ is equal to one of $\Z_n,$ $\D_n,$
$\tetra,$ $\octa,$ $\icosa.$ Then, $p_\vect$ is surjective, and
\begin{equation*}
p_\vect \times c_1 : \Vect_{G_\chi} (S^2, \chi) \rightarrow
A_{G_\chi} ( S^2, \chi ) \times H^2 (S^2)
\end{equation*}
is injective. More precisely, bundles in
$p_{\vect}^{-1}(\mathbf{W})$ for each $\mathbf{W}$ in $A_{G_\chi} (
S^2, \chi )$ have all different Chern classes, and $c_1 \big(
p_{\vect}^{-1}(\mathbf{W}) \big)$ is equal to $\big\{ ~ \chi( \id )
\big( l_R k + k_0 \big) ~ \big| ~ k \in \Z ~ \big\}$ where $k_0$ is
dependent on $\mathbf{W}.$
\end{main}

\begin{main} \label{main: not by isotropy and chern class}
Assume that $\rho(G_\chi) = R$ is equal to one of $\Z_n \times Z$
with odd $n,$ $\langle -a_n \rangle$ with even $n/2.$ For each
$\mathbf{W}$ in $A_{G_\chi} ( S^2, \chi ),$ its preimage
$p_{\vect}^{-1} ( \mathbf{W} )$ has exactly two elements which have
the same Chern class. Also, $[E \oplus E_1] \ne [E \oplus E_2]$ for
any bundles $E,$ $E_1,$ $E_2$ in $\Vect_{G_\chi} (S^2, \chi)$ such
that $p_\vect([E_1]) = p_\vect([E_2])$ and $[E_1] \ne [E_2].$
\end{main}

\begin{main} \label{main: only by isotropy}
Assume that $\rho(G_\chi) = R$ for some $R$ of Table \ref{table:
introduction} which is not equal to any group appearing in the above
two theorems. Then, $p_{\vect}$ is an isomorphism.
\end{main}

Theorem \ref{main: only by isotropy} says that isotropy
representations at some points classify equivariant vector bundles
for some $R$'s. Comparing to this, Theorem \ref{main: by isotropy
and chern}, \ref{main: not by isotropy and chern class} might be
unsatisfactory to some readers because statements on Chern classes,
especially $k_0$, are not so concrete. To handle this, we show the
following:

\begin{main} \label{main: reduction to line bundle}
Assume that $\rho(G_\chi)=R$ for some $R$ appearing in Theorem
\ref{main: by isotropy and chern}, \ref{main: not by isotropy and
chern class}. Then, $\Vect_{G_\chi} (S^2, \chi)$ is isomorphic to
$\Vect_R (S^2)$ as semigroups, and $\Vect_R (S^2)$ is generated by
line bundles. Also, $A_R (S^2, \id)$ is generated by all the
elements with one-dimensional entries. The number of such elements
is equal to
\begin{equation*}
\left\{
  \begin{array}{ll}
    |R_S| \times |R_N|                              &  \text{if } R = \Z_n,   \\
    |R_{d^{-1}}| \times |R_{d^0}| \times |R_{d^1}|  &  \text{if } R \ne
\Z_n
  \end{array}
\right.
\end{equation*}
where $R_x$ is the isotropy subgroup at $x \in S^2$ and we denote
simply by $\id$ the trivial character of the trivial group.
\end{main}

In Section \ref{section: reduction to effective}, we calculate Chern
classes of line bundles in $\Vect_R (S^2),$ and from this we obtain
$k_0.$ Also in Section \ref{section: reduction to effective}, we
explain for the reason why we prove the isomorphism of Theorem
\ref{main: reduction to line bundle} only for $R$'s appearing in
Theorem \ref{main: by isotropy and chern}, \ref{main: not by
isotropy and chern class}.

This paper is organized as follows. In Section \ref{section: closed
subgroups}, we list all closed subgroups of $\orthogonal(3)$ up to
conjugacy. In Section \ref{section: simplicial complex}, we give an
equivariant simplicial complex structure $\complexK_R$ on $S^2$
according to finite $R,$ and investigate equivariance of $S^2$ by
calculating isotropy subgroups at vertices and barycenters of
$|\complexK_R|.$ Section \ref{section: clutching
construction}$\sim$\ref{section: nonzero-dimensional case} are
divided into three parts. In Section \ref{section: clutching
construction}$\sim$\ref{section: platonic case}, we first deal with
cases when $\complexK_R = \complexK_\tetra,$ $\complexK_\octa,$
$\complexK_\icosa.$ In Section \ref{section: clutching
construction}, we introduce a new simplicial complex denoted by
$\lineK_R$ which is just a disjoint union of faces of $\complexK_R,$
and consider $|\complexK_R|$ as a quotient space of the underlying
space $|\lineK_R|.$ Also, we consider an equivariant vector bundle
over $S^2$ as an equivariant clutching construction of an
equivariant vector bundle over $|\lineK_R|.$ For this, we define
equivariant clutching map. In Section \ref{section: relation}, we
investigate relations among $\Vect_{G_\chi} (S^2, \chi),$
$A_{G_\chi} (S^2, \chi),$ and homotopy of the set of equivariant
clutching maps. From these relations, it is shown that our
classification in most cases is obtained by calculation of the
homotopy. In Section \ref{section: pointwise clutching map}, we
develop our machinery called equivariant pointwise gluing which
glues an equivariant vector bundle over a finite set through a map
called equivariant pointwise clutching map. In calculation of the
homotopy, equivariant pointwise clutching map plays a key role
because an equivariant clutching map can be considered as a
continuous collection of equivariant pointwise clutching maps. Here,
the concept of representation extension enters with which
equivariant pointwise gluing is described in the language of
representation theory. In this way, the homotopy is techniquely
related to $A_{G_\chi} ( S^2, \chi ),$ and calculation of it becomes
reduced to calculation of a relative (nonequivariant) homotopy. So,
we prove a lemma on relative homotopy in Section \ref{section:
lemmas on fundamental group}. In Section \ref{section: equivariant
clutching maps}, we prove technical lemmas needed in dealing with
equivariant clutching maps through equivariant pointwise clutching
maps. In Section \ref{section: platonic case}, we prove main
theorems for cases when $\complexK_R = \complexK_\tetra,$
$\complexK_\octa,$ $\complexK_\icosa.$ In Section \ref{section:
clutching construction cyclic dihedral}$\sim$\ref{section: cyclic
dihedral cases}, we second deal with cases when $\complexK_R =
\complexK_m$ for some $m \in \N.$ In Section \ref{section: clutching
construction cyclic dihedral}, we rewrite what we have done in
Section \ref{section: clutching construction}, \ref{section:
relation}, \ref{section: equivariant clutching maps} to be in
accordance with $\complexK_m.$ In Section \ref{section: cyclic
dihedral cases}, we prove main theorems for the cases when
$\complexK_R = \complexK_m.$ In Section \ref{section:
nonzero-dimensional case}, we third prove Theorem \ref{main: only by
isotropy} for cases when $R$ is one-dimensional. In Section
\ref{section: reduction to effective}, we prove Theorem \ref{main:
reduction to line bundle}. Section \ref{section: representation
extension} is the appendix on representation extension.

\section{Closed subgroups of $\orthogonal (3)$} \label{section: closed subgroups}

In this section, we list all closed subgroups of $\orthogonal(3)$ up
to conjugacy. In this section, the notation $\times$ is internal
direct product of two subgroups in $\orthogonal(3).$ Since $Z$ is
the centralizer of $\orthogonal(3),$ $\orthogonal(3) = \SO(3) \times
Z$ and this gives the exact sequence
\begin{equation*}
0 \rightarrow Z \hookrightarrow \orthogonal(3)
\overset{\pr}{\longrightarrow} \SO(3) \rightarrow 0
\end{equation*}
where $\pr|_{\SO(3)}$ is the identity map. First, we deal with
finite subgroups.

\begin{proposition} \label{proposition: finite subgroup of O(3)}
Let $R$ be a finite subgroup of $\orthogonal(3)$ such that $R
\nsubseteq \SO(3).$
\begin{enumerate}
  \item If $\pr(R)$ is conjugate to $\Z_n,$ then
        \[
        R \text{ is conjugate to }
        \begin{cases}
        \Z_n \times Z & \hbox{if $n$ is odd,} \\
        \Z_n \times Z \hbox{ or } \langle -a_n \rangle
        & \hbox{if $n$ is even.}
        \end{cases}
        \]
  \item If $\pr(R)$ is conjugate to $\D_n,$ then
        \[
        R \text{ is conjugate to }
        \begin{cases}
        \D_n \times Z \hbox{ or } \langle a_n, -b \rangle
        & \hbox{if $n$ is odd,} \\
        \D_n \times Z, ~ \langle a_n, -b \rangle,
         \\
        \langle -a_n, b \rangle, \hbox{ or } \langle -a_n, -b \rangle
        & \hbox{if $n$ is even.}
        \end{cases}
        \]
  \item If $\pr(R)$ is conjugate to $\tetra,$ then $R$ is conjugate to $\tetra \times Z.$
  \item If $\pr(R)$ is conjugate to $\octa,$ then $R$ is conjugate to $\octa \times Z$ or
  $\langle \tetra, -o_0 \rangle.$
  \item If $\pr(R)$ is conjugate to $\icosa,$ then $R$ is conjugate to $\icosa \times Z.$
\end{enumerate}
\end{proposition}

\begin{proof}
We may assume that $\pr(R)$ is equal to one of $\Z_n,$ $\D_n,$
$\tetra,$ $\octa,$ $\icosa.$ Denote by $R_\rot$ the subgroup $R \cap
\SO(3).$ Then, $R_\rot$ is an index two subgroup of $R.$ Denoting
$\pr(R)$ by $K,$ $R_\rot \subset K$ because $\pr|_{\SO(3)}$ is the
identity map. Also, since the preimage $\pr^{-1}(g)$ of $g \in
\SO(3)$ is equal to $\{ g, -g \},$ it is obtained that $R \subset K
\times Z.$ Here are two possibilities. First, if $Z = \ker \pr
\subset R,$ then $|R| = 2 |K|$ so that
\begin{equation}
\label{equation: R=KxZ} R = K \times Z
\end{equation}
from $R \subset K \times Z.$ Second, if $Z \nsubseteq R,$ then
$\pr|_R$ is injective and $|R| = |K|$ so that
\begin{equation}
\label{equation: K=2Rrot} |K| = 2 |R_\rot|
\end{equation}
because $R_\rot$ is an index two subgroup of $R.$ For an element
$g_0 \in K - R_\rot$ and its preimage $\pr^{-1} (g_0) = \{ g_0, -g_0
\},$ $R$ is equal to $\langle R_\rot, g_0 \rangle$ or $\langle
R_\rot, -g_0 \rangle$ because $K$ $=$ $\langle R_\rot, g_0 \rangle$
is the injective image of $R.$ However, if $R = \langle R_\rot, g_0
\rangle,$ then $R \subset \SO(3)$ and this contradicts the
assumption. So, we obtain
\begin{equation}
\label{equation: R=Rrot-g0} R = \langle R_\rot, -g_0 \rangle.
\end{equation}
In the remaining proof, we apply (\ref{equation: R=KxZ}),
(\ref{equation: K=2Rrot}), (\ref{equation: R=Rrot-g0}) to each
possible $K.$

If $K = \Z_n$ with odd $n$ and $Z \subset R,$ then $R = \Z_n \times
Z.$ If $K = \Z_n$ with odd $n$ and $Z \nsubseteq R,$ then $K$ has no
index two subgroup so that there exists no possible $R_\rot$ in $K.$
That is, there exists no such $R.$ Therefore, a proof of (1) for odd
$n$ is obtained.

If $K = \Z_n$ with even $n$ and $Z \subset R,$ then $R = \Z_n \times
Z.$ If $K = \Z_n$ with even $n$ and $Z \nsubseteq R,$ then there is
the unique index two subgroup $\langle a_n^2 \rangle$ of $K$ which
should be equal to $R_\rot.$ Since $a_n \in K - R_\rot,$ $R =
\langle R_\rot, -a_n \rangle = \langle -a_n \rangle$ and this is the
proof of (1) for even $n.$

If $K = \D_n$ with odd $n$ and $Z \subset R,$ then $R = \D_n \times
Z.$ If $K = \D_n$ with odd $n$ and $Z \nsubseteq R,$ then there is
the unique index two subgroup $\Z_n$ of $\D_n$ which should be equal
to $R_\rot.$ Since $b \in K - R_\rot,$ $R = \langle R_\rot, -b
\rangle = \langle a_n, -b \rangle$ and this is a proof of (2) for
odd $n.$

If $K = \D_n$ with even $n$ and $Z \subset R,$ then $R = \D_n \times
Z.$ If $K = \D_n$ with even $n$ and $Z \nsubseteq R,$ then index two
subgroups of $\D_n$ are $\Z_n,$ $\langle a_n^2, b \rangle,$ $\langle
a_n^2, a_n b \rangle$ which are candidates for $R_\rot$ for some
$R.$ If $R_\rot = \Z_n,$ then $R= \langle R_\rot, -b \rangle =
\langle a_n, -b \rangle.$ If $R_\rot = \langle a_n^2, b \rangle,$
then $R = \langle R_\rot, -a_n \rangle = \langle -a_n, b \rangle.$
If $R_\rot = \langle a_n^2, a_n b \rangle,$ then $R = \langle
R_\rot, -a_n \rangle = \langle -a_n, a_n b \rangle = \langle -a_n,
-b \rangle.$ Therefore, a proof of (2) for even $n$ is obtained.

If $K=\tetra$ or $\icosa,$ then it is well-known that $K$ has no
index two subgroup because $\tetra \cong A_4$ and $\icosa \cong
A_5.$ So, $Z \nsubseteq R$ can not happen. Therefore, we obtain a
proof for (3) and (5).

If $K=\octa$ and $Z \subset R,$ then $R = \octa \times Z.$ If
$K=\octa$ and $Z \nsubseteq R,$ then $\tetra$ is the only index two
subgroup of $\octa,$ so $R_\rot = \tetra.$ Since $o_0 \in \octa -
\tetra,$ we have $R= \langle R_\rot, -o_0 \rangle = \langle \tetra,
-o_0 \rangle,$ and this is a proof of (4). \qed
\end{proof}

In Proposition \ref{proposition: finite subgroup of O(3)}, it can be
observed that \\
\begin{tabular}{lrlr}
  (1) & $\D_n$ with $n=1$                      & is conjugate to & $\Z_n$ with $n=2,$  \\
  (2) & $\langle -a_n, b \rangle$ with $n=2$   & is conjugate to & $\langle a_n, -b \rangle$ with $n=2,$  \\
  (3) & $\langle -a_n \rangle$ with $n=2$      & is conjugate to & $\langle a_n, -b \rangle$ with $n=1,$ \\
  (4) & $\langle -a_n, -b \rangle$ with $n=2$  & is conjugate to & $\langle a_n, -b \rangle$ with $n=2,$ \\
  (5) & $\Z_n \times Z$ with $n=2$             & is conjugate to & $\D_n \times Z$ with $n=1.$ \\
\end{tabular}
\\ To avoid these repetition, we have put conditions `$n>1$' or
`$n>2$' in five $R$-entries of Table \ref{table: introduction}.

Second, we deal with closed nonzero-dimensional subgroups of
$\orthogonal(3).$



\begin{proposition}
\label{proposition: infinite subgroup of O(3)} Let $R$ be a
nonzero-dimensional closed subgroup of $\orthogonal(3)$ such that $R
\nsubseteq \SO(3).$
\begin{enumerate}
  \item If $\pr(R)$ is conjugate to $\SO(2),$
  then $R$ is conjugate to the group $\SO(2) \times Z$ $=$ $\langle \SO(2), -a_2 \rangle.$
  \item If $\pr(R)$ is conjugate to $\orthogonal(2),$
  then $R$ is conjugate to $\langle \SO(2), -b \rangle$ or $\orthogonal(2) \times Z.$
  \item If $\pr(R)$ is conjugate to $\SO(3),$ then
        $R$ is conjugate to $\orthogonal(3).$
\end{enumerate}
\end{proposition}

Proof of this is done in a similar way with Proposition
\ref{proposition: finite subgroup of O(3)}. We have explained for
$R$-entry of Table \ref{table: introduction}. Here, we calculate
some isotropy subgroups for later use.

\begin{lemma}
 \label{lemma: isotropy at v^0 for nonzero dimensional}
For each one-dimensional $R$ of Table \ref{table: introduction} and
its natural action on $S^2,$ isotropy subgroups $R_{v^0}$ and
$R_{v^0} \cap R_S$ are calculated as in Table \ref{table: isotropy
of nonzero-dimensional}.
\begin{table}[ht!]
\begin{center}
{\footnotesize
\begin{tabular}{c|c|c}
$R$                             & $R_{v^0}$                  & $R_{v^0} \cap R_S$ \\

\hhline{=|=|=}

$\orthogonal(2) \times Z$       & $\langle b, -a_2 \rangle$  & $\langle -a_2 b \rangle$  \\
$\langle \SO(2), -b \rangle$    & $\langle -a_2 b \rangle$   & $\langle -a_2 b \rangle$  \\
$\langle \SO(2), -a_2 \rangle$  & $\langle -a_2 \rangle$     & $\langle \id \rangle$     \\
$\orthogonal(2)$                & $\langle b \rangle$        & $\langle \id \rangle$     \\
$\SO(2)$                        & $\langle \id \rangle$      & $\langle \id \rangle$     \\

\end{tabular}}
\caption{\label{table: isotropy of nonzero-dimensional} Isotropy
subgroups at $v^0 = (1,0,0)$ for closed one-dimensional subgroups of
$\orthogonal(3)$}
\end{center}
\end{table}

\end{lemma}

\section{Equivariant simplicial complex structures for finite subgroups of $\orthogonal(3)$}
\label{section: simplicial complex}

Let a finite $R$ of Table \ref{table: introduction} act naturally on
$S^2.$ In this section, we explain for $\complexK_R$- and
$D_R$-entries of Table \ref{table: introduction}, and investigate
equivariance of $|\complexK_R|$ by calculating isotropy subgroups at
some points. We do these for $\pr(R) = \Z_n,$ $\D_n,$ and then for
$\pr(R) = \tetra,$ $\octa,$ $\icosa.$

First, we define $\complexK_m$ and its $v^i,$ $e^i$ for $m=1,$ $2$
as promised in Introduction. Denote by $\complexK_1$ the simplicial
complex which is the same with $\complexK_4$ but has the following
notations:
\begin{align*}
v^0           &= (1,0,0), \\
b(e^0)        &= (-1,0,0), \\
[v^0, b(e^0)] &= [(1,0,0), (0,1,0)] \cup [(0,1,0), (-1,0,0)], \\
|e^0|         &= P_4.
\end{align*}
And, denote by $\complexK_2$ the simplicial complex which is the
same with $\complexK_4$ but has the following notations:
\begin{align*}
v^i    &= \exp \Big( ~ \pi i \sqrt{-1} ~ \Big), \\
b(e^i) &= \exp \Big( ~ \frac{2 \pi (2i+1) \sqrt{-1}} {4} ~ \Big), \\
|e^i|  &= [v^i, b(e^i)] \cup [b(e^i), v^{i+1}]
\end{align*}
for $i \in \Z_2.$ Also, let $|\complexK_1|$ and $|\complexK_2|$ be
equal to $|\complexK_4|.$ Here, we remark that if $\complexK_R =
\complexK_1$ and $D_R = |e^0|$ in Table \ref{table: introduction},
then $d^0$ and $d^1$ are defined as $v^0.$ So, we have finished
defining $\complexK_m$ for all natural number $m.$ Similarly, denote
by $P_m$ for $m = 1,2$ the regular polygon $P_4.$ For reader's
convenience, we list all possible $R$'s with $\complexK_R =
\complexK_1$ or $\complexK_2.$

\begin{table}[ht!]
\begin{center}
{\footnotesize
\begin{tabular}{l|c||c}
$R$                                   & $n$ & $\complexK_R$ \\

\hhline{=|=#=}

$\D_n$                                & $2$ & $\complexK_2$  \\
$\Z_n$                                & $1$ & $\complexK_1$  \\
$\Z_n$                                & $2$ & $\complexK_2$  \\

\hline

$\D_n \times Z,$ odd $n$              & $1$ & $\complexK_2$  \\
$\langle a_n, -b \rangle,$ odd $n$    & $1$ & $\complexK_2$  \\
$\Z_n \times Z,$ odd $n$              & $1$ & $\complexK_2$  \\

\hline

$\D_n \times Z,$ even $n$             & $2$ & $\complexK_2$  \\
$\langle a_n, -b \rangle,$ even $n$   & $2$ & $\complexK_2$  \\
\end{tabular}}
\caption{\label{table: group with m_R <=2} Groups with $\complexK_R
= \complexK_1$ or $\complexK_2$}
\end{center}
\end{table}

\begin{lemma}
 \label{lemma: cases of m_R less than 3}
There are eight $R$'s in Table \ref{table: introduction} which
satisfy $\pr(R) = \Z_n,$ $\D_n$ and $\complexK_R = \complexK_1,$
$\complexK_2.$ They are listed in Table \ref{table: group with m_R
<=2}.
\end{lemma}

Now, we explain for $\complexK_R$-entry of Table \ref{table:
introduction} when $\pr(R) = \Z_n,$ $\D_n.$ Denoting by $m_R$ the
number $|R| / |R_{v^0}|,$ we can check that $\complexK_R$ is equal
to $\complexK_{m_R}.$ Since $R \cdot v^0$ is equal to the set of
vertices of $\complexK_{m_R}$ in $xy$-plane, $|\complexK_R|$ and
$\complexK_R$ are $R$-invariant for each $R.$ Here, we calculate
$R_{v^0}.$

\begin{lemma}
 \label{lemma: isotropy for cyclic dihedral}
For each finite $R$ of Table \ref{table: introduction} such that
$\pr(R) = \Z_n,$ $\D_n,$ isotropy subgroups $R_{v^0},$ $R_{v^0} \cap
R_S$ are calculated as in Table \ref{table: isotropy at (1,0,0)}
when $R$ acts naturally on $S^2.$
\end{lemma}

\begin{proof}
By calculation, $R_{v^0} = \langle b \rangle$ for $R = \D_n \times
Z$ with odd $n,$ and $R_{v^0} = \langle -a_n^{n/2}, b \rangle$ for
$R = \D_n \times Z$ with even $n.$ Other $R$'s are subgroups of one
of these two. So, $R_{v^0} = R \cap (\D_n)_{v^0}$ is easily
calculated, and from this the group $R_{v^0} \cap R_S$ is also
obtained. \qed
\end{proof}

\begin{table}[ht!]
\begin{center}
{\footnotesize
\begin{tabular}{l||c|c}
$R$                                            & $R_{v^0}$                         & $R_{v^0} \cap R_S$  \\

\hhline{=#=|=}

$\D_n,$ $n>1$                                  & $\langle b \rangle$               & $\langle \id \rangle$ \\
$\Z_n$                                         & $\langle \id \rangle$             & $\langle \id \rangle$ \\

\hline

$\D_n \times Z,$ odd $n$                       & $\langle b \rangle$               & $\langle \id \rangle$ \\
$\langle a_n, -b \rangle,$ odd $n$             & $\langle \id \rangle$             & $\langle \id \rangle$ \\
$\Z_n \times Z,$ odd $n$                       & $\langle \id \rangle$             & $\langle \id \rangle$ \\

\hline

$\D_n \times Z,$ even $n$                      & $\langle -a_n^{n/2}, b \rangle$   & $\langle -a_n^{n/2}b \rangle$ \\
$\langle a_n, -b \rangle,$ even $n$            & $\langle -a_n^{n/2}b \rangle$     & $\langle -a_n^{n/2}b \rangle$ \\
$\langle -a_n, b \rangle,$ odd $n/2,$ $n>2$    & $\langle -a_n^{n/2}, b \rangle$   & $\langle -a_n^{n/2}b \rangle$ \\
$\langle -a_n, b \rangle,$ even $n/2$          & $\langle b \rangle$               & $\langle \id \rangle$ \\
$\langle -a_n, -b \rangle,$ odd $n/2,$ $n>2$   & $\langle -a_n^{n/2} \rangle$      & $\langle \id \rangle$ \\
$\langle -a_n, -b \rangle,$ even $n/2$         & $\langle -a_n^{n/2}b \rangle$     & $\langle -a_n^{n/2}b \rangle$ \\
$\Z_n \times Z,$ even $n,$ $n>2$               & $\langle -a_n^{n/2} \rangle$      & $\langle \id \rangle$ \\
$\langle -a_n \rangle,$ odd $n/2,$ $n>2$       & $\langle -a_n^{n/2} \rangle$      & $\langle \id \rangle$ \\
$\langle -a_n \rangle,$ even $n/2$             & $\langle \id \rangle$             & $\langle \id \rangle$ \\
\end{tabular}}
\caption{\label{table: isotropy at (1,0,0)} Isotropy subgroup at
$v^0$ for $\pr(R) = \Z_n$ or $\D_n.$}
\end{center}
\end{table}

We calculate other isotropy subgroups.

\begin{lemma}
 \label{lemma: isotropy at b(e^0), x for cyclic dihedral}
Let $R$ be a finite group in Table \ref{table: introduction} such
that $\pr(R) = \Z_n,$ $\D_n.$ Isotropy subgroups $R_{b(e^0)},$
$R_{b(e^0)} \cap R_S,$ $R_x$ are calculated as in Table \ref{table:
isotropy for cyclic dihedral} where $x$ is a point in $P_{m_R} - \{
v^i, b(e^i) | i \in \Z_{m_R} \}.$

\begin{table}[ht!]
\begin{center}
{\footnotesize
\begin{tabular}{l||c|c|c|c}
$R$                                            & $\complexK_R$     & $R_{b(e^0)}$                          & $R_{b(e^0)} \cap R_S$             & $R_x$                \\

\hhline{=#=|=|=|=}

$\D_n,$ $n>1$                                  & $\complexK_n$     & $\langle a_n b \rangle$               & $\langle \id \rangle$              & $\langle \id \rangle$ \\
$\Z_n$                                         & $\complexK_n$     & $\langle \id \rangle$                  & $\langle \id \rangle$              & $\langle \id \rangle$ \\

\hline

$\D_n \times Z,$ odd $n$                       & $\complexK_{2n}$  & $\langle -a_n^{(n+1)/2}b \rangle$     & $\langle -a_n^{(n+1)/2}b \rangle$ & $\langle \id \rangle$ \\
$\langle a_n, -b \rangle,$ odd $n$             & $\complexK_{2n}$  & $\langle -a_n^{(n+1)/2}b \rangle$     & $\langle -a_n^{(n+1)/2}b \rangle$ & $\langle \id \rangle$ \\
$\Z_n \times Z,$ odd $n$                       & $\complexK_{2n}$  & $\langle \id \rangle$                  & $\langle \id \rangle$              & $\langle \id \rangle$ \\

\hline

$\D_n \times Z,$ even $n$                      & $\complexK_n$     & $\langle -a_n^{n/2}, a_n b \rangle$   & $\langle -a_n^{n/2 +1} b \rangle$ & $\langle -a_n^{n/2} \rangle$ \\
$\langle a_n, -b \rangle,$ even $n$            & $\complexK_n$     & $\langle -a_n^{n/2+1}b \rangle $      & $\langle -a_n^{n/2+1} b \rangle $ & $\langle \id \rangle$ \\
$\langle -a_n, b \rangle,$ odd $n/2,$ $n>2$    & $\complexK_{n/2}$ & $\langle -a_n^{n/2}, a_n^2 b \rangle$ & $\langle -a_n^{n/2+2} b \rangle$  & $\langle -a_n^{n/2} \rangle$ \\
$\langle -a_n, b \rangle,$ even $n/2$          & $\complexK_n$     & $\langle -a_n^{n/2+1}b \rangle$       & $\langle -a_n^{n/2+1} b \rangle$  & $\langle \id \rangle$ \\
$\langle -a_n, -b \rangle,$ odd $n/2,$ $n>2$   & $\complexK_n$     & $\langle -a_n^{n/2}, a_n b \rangle$   & $\langle -a_n^{n/2+1} b \rangle$  & $\langle -a_n^{n/2} \rangle$ \\
$\langle -a_n, -b \rangle,$ even $n/2$         & $\complexK_n$     & $\langle a_n b \rangle$               & $\langle \id \rangle$              & $\langle \id \rangle$ \\
$\Z_n \times Z,$ even $n,$ $n>2$               & $\complexK_n$     & $\langle -a_n^{n/2} \rangle$          & $\langle \id \rangle$              & $\langle -a_n^{n/2} \rangle$ \\
$\langle -a_n \rangle,$ odd $n/2,$ $n>2$       & $\complexK_{n/2}$ & $\langle -a_n^{n/2} \rangle$          & $\langle \id \rangle$              & $\langle -a_n^{n/2} \rangle$ \\
$\langle -a_n \rangle,$ even $n/2$             & $\complexK_n$     & $\langle \id \rangle$                  & $\langle \id \rangle$              & $\langle \id \rangle$ \\
\end{tabular}}
\caption{\label{table: isotropy for cyclic dihedral} Isotropy
subgroups for $\pr(R) = \Z_n,$ $\D_n.$}
\end{center}
\end{table}
\end{lemma}

\begin{proof}
Since the $R$-action on $|\complexK_R|$ is simplicial, any element
$g$ in $R_x$ fixes the whole $P_{m_R}.$ Therefore, $R_x = \langle
\id \rangle$ or $\langle -a_n^{n/2} \rangle$ for even $n,$ and it is
easy to calculate $R_x.$

Observing that $R_{b(e^0)}$ is isomorphic to a subgroup of $\Z_2
\times \Z_2$ and that $R_x \subset R_{b(e^0)}$ and $R_{b(e^0)} / R_x
\cong \langle \id \rangle$ or $\Z_2,$ we can calculate $R_{b(e^0)}$
case by case. \qed
\end{proof}

\begin{remark}
In the cases of $\langle a_n, -b \rangle$ with odd $n$ and $\langle
-a_n, -b \rangle$ with odd $n/2,$ it is observed that $| R_{v^0} | >
| R_{b(e^0)} |,$ so $R$ does not act transitively on $b(e^i)$'s.
And, in the case of $\langle -a_n, -b \rangle$ with odd $n/2,$ we
additionally calculate $R_{b(e^1)} = \langle -a_n^{n/2}, a_n^3 b
\rangle.$ \qed
\end{remark}

Here, we explain for $D_R.$

\begin{lemma}
 \label{lemma: D_R for cyclic dihedral}
For each finite $R$ of Table \ref{table: introduction} such that
$\pr(R) = \Z_n,$ $\D_n,$ the $R$-orbit of the $D_R$-entry in Table
\ref{table: introduction} covers $P_{m_R},$ and $D_R$ is a minimal
path satisfying such a property. So, any interior point $x$ of $D_R$
is not moved to other point in $D_R$ by $R.$
\end{lemma}

\begin{proof}
First, observe that $|R_{v^0}| = |R_{b(e^0)}|$ if and only if $R$
acts transitively on $b(e^i)$'s because $R$ acts transitively on
$v^i$'s and $|\complexK_R|$ has the same number of $v^i$'s and
$b(e^i)$'s. And, $R$ acts transitively on $b(e^i)$'s if and only if
the $R$-orbit of $|e^0|$ cover $P_{m_R}.$ By Table \ref{table:
isotropy at (1,0,0)} and \ref{table: isotropy for cyclic dihedral},
$|R_{v^0}| = |R_{b(e^0)}|$ except two cases $\langle a_n, -b
\rangle$ with odd $n$ and $\langle -a_n, -b \rangle$ with odd $n/2.$
If $|R_{v^0}| = |R_{b(e^0)}|$ and $R_{b(e^0)}/R_x \cong \Z_2,$ then
$[v^0, b(e^0)]$ can be moved to $[b(e^0), v^1]$ by $R$ so that the
$R$-orbit of $[v^0, b(e^0)]$ covers $P_{m_R}.$ In the other side, if
$|R_{v^0}| = |R_{b(e^0)}|$ and $R_{b(e^0)} = R_x,$ then the
$R$-orbit of $[v^0, b(e^0)]$ does not cover $P_{m_R}.$ If $|R_{v^0}|
< |R_{b(e^0)}|,$ then the $R$-orbit of $|e^0|$ does not cover
$P_{m_R}$ because $R$ does not act transitively on $b(e^i)$'s, but
the $R$-orbit of $[b(e^0), v^1] \cup [v^1, b(e^1)]$ covers $P_{m_R}$
because $R$ acts transitively on $v^i$'s. So, the remaining of proof
is done by comparing Table \ref{table: isotropy at (1,0,0)} with
Table \ref{table: isotropy for cyclic dihedral}. \qed
\end{proof}

Now, we repeat these arguments for $R$'s satisfying $\pr(R) =
\tetra, \octa, \icosa.$

\begin{lemma}
 \label{lemma: K_R for platonic}
For each finite $R$ of Table \ref{table: introduction} such that
$\pr(R) = \tetra, \octa, \icosa,$ $|\complexK_R|$ is $R$-invariant.
And, $R$ act transitively on vertices, edges, faces of
$\complexK_R,$ respectively.
\end{lemma}

\begin{remark}
 \label{remark: TxZ}
For later use, we need understand the case of $R = \tetra \times Z$
because it is not equal to the full symmetry of $|\complexK_R| =
|\complexK_\octa|.$ In this case, $R_{b(f^{-1})}$ is equal to
$\tetra_{b(f^{-1})} \cong \Z_3.$ Also, $-a_4^2, -a_4^2 b$ are in
$\tetra \times Z$ because $b, a_4^2 \in \tetra.$ Here, $-a_4^2,$
$-a_4^2 b$ are reflections through the $xy$-plane, $xz$-plane,
respectively. \qed
\end{remark}

\begin{lemma}
 \label{lemma: isotropy for platonic}
For each finite $R$ of Table \ref{table: introduction} such that
$\pr(R) = \tetra,$ $\octa,$ $\icosa,$ isotropy subgroups $R_{v^0},$
$R_{b(e^0)},$ $R_x$ are in Table \ref{table: all about K_R} where
$x$ is an interior point of $[v^0, b(e^0)].$ In the table, the
notation $\cong$ is used when an isotropy subgroup is not equal to
$\Z_n = \langle a_n \rangle$ or $\D_n = \langle a_n, b \rangle$ for
some $n$ but isomorphic to one of them.
\end{lemma}

\begin{lemma}
 \label{lemma: D_R for platonic}
For each finite $R$ of Table \ref{table: introduction} such that
$\pr(R) = \tetra, \octa, \icosa,$ the $R$-orbit of the $D_R$-entry
in Table \ref{table: introduction} covers $|\complexK_R^{(1)}|,$ and
$D_R$ is a minimal path satisfying such a property.
\end{lemma}

\begin{table}[ht!]
\begin{center}
{\footnotesize
\begin{tabular}{l||c|c|c|c|c|c}
$R$                                           & $\complexK_R$      & $D_R$              & $R_{v^0}$                        & $R_{b(e^0)}$                          & $R_x$                         & $R_{d^{-1}}$                    \\

\hhline{=#=|=|=|=|=|=}

$\D_n,$ $n>1$                                 & $\complexK_n$      & $[v^0, b(e^0)]$    & $\langle b \rangle$              & $\langle a_n b \rangle$               & $\langle \id \rangle$         & $\Z_n$                          \\
$\Z_n$                                        & $\complexK_n$      & $|e^0|$            & $\langle \id \rangle$            & $\langle \id \rangle$                 & $\langle \id \rangle$         & $\Z_n$                          \\

\hline

$\D_n \times Z,$ odd $n$                      & $\complexK_{2n}$   & $[v^0, b(e^0)]$    & $\langle b \rangle$              & $\langle -a_n^{(n+1)/2}b \rangle$     & $\langle \id \rangle$         & $\langle a_n, -b \rangle$       \\
$\langle a_n, -b \rangle,$ odd $n$            & $\complexK_{2n}$   & $[b(e^0), b(e^1)]$ & $\langle \id \rangle$            & $\langle -a_n^{(n+1)/2}b \rangle$     & $\langle \id \rangle$         & $\langle a_n, -b \rangle$       \\
$\Z_n \times Z,$ odd $n$                      & $\complexK_{2n}$   & $|e^0|$            & $\langle \id \rangle$            & $\langle \id \rangle$                 & $\langle \id \rangle$         & $\Z_n$                          \\

\hline

$\D_n \times Z,$ even $n$                     & $\complexK_n$      & $[v^0, b(e^0)]$    & $\langle -a_n^{n/2}, b \rangle$  & $\langle -a_n^{n/2}, a_n b \rangle$   & $\langle -a_n^{n/2} \rangle$  & $\langle a_n, -b \rangle$       \\
$\langle a_n, -b \rangle,$ even $n$           & $\complexK_n$      & $[v^0, b(e^0)]$    & $\langle -a_n^{n/2}b \rangle$    & $\langle -a_n^{n/2+1}b \rangle $      & $\langle \id \rangle$         & $\langle a_n, -b \rangle$       \\
$\langle -a_n, b \rangle,$ odd $n/2,$ $n>2$   & $\complexK_{n/2}$  & $|e^0|$            & $\langle -a_n^{n/2}, b \rangle$  & $\langle -a_n^{n/2}, a_n^2 b \rangle$ & $\langle -a_n^{n/2} \rangle$  & $\langle a_n^2, -a_n b \rangle$ \\
$\langle -a_n, b \rangle,$ even $n/2$         & $\complexK_n$      & $[v^0, b(e^0)]$    & $\langle b \rangle$              & $\langle -a_n^{n/2+1}b \rangle$       & $\langle \id \rangle$         & $\langle a_n^2, -a_n b \rangle$ \\
$\langle -a_n, -b \rangle,$ odd $n/2,$ $n>2$  & $\complexK_n$      & $[b(e^0), b(e^1)]$ & $\langle -a_n^{n/2} \rangle$     & $\langle -a_n^{n/2}, a_n b \rangle$   & $\langle -a_n^{n/2} \rangle$  & $\langle a_n^2, -b \rangle$     \\
$\langle -a_n, -b \rangle,$ even $n/2$        & $\complexK_n$      & $[v^0, b(e^0)]$    & $\langle -a_n^{n/2}b \rangle$    & $\langle a_n b \rangle$               & $\langle \id \rangle$         & $\langle a_n^2, -b \rangle$     \\
$\Z_n \times Z,$ even $n,$ $n>2$              & $\complexK_n$      & $|e^0|$            & $\langle -a_n^{n/2} \rangle$     & $\langle -a_n^{n/2} \rangle$          & $\langle -a_n^{n/2} \rangle$  & $\Z_n$                          \\
$\langle -a_n \rangle,$ odd $n/2,$ $n>2$      & $\complexK_{n/2}$  & $|e^0|$            & $\langle -a_n^{n/2} \rangle$     & $\langle -a_n^{n/2} \rangle$          & $\langle -a_n^{n/2} \rangle$  & $\langle a_n^2 \rangle$         \\
$\langle -a_n \rangle,$ even $n/2$            & $\complexK_n$      & $|e^0|$            & $\langle \id \rangle$            & $\langle \id \rangle$                 & $\langle \id \rangle$         & $\langle a_n^2 \rangle$         \\

\hline

$\tetra$                                      & $\complexK_\tetra$ & $[v^0, b(e^0)]$    &  $\cong \Z_3$                    & $\cong \Z_2$                          & $\langle \id \rangle$         & $\cong \Z_3$                    \\
$\octa$                                       & $\complexK_\octa$  & $[v^0, b(e^0)]$    &  $\cong \Z_4$                    & $\cong \Z_2$                          & $\langle \id \rangle$         & $\cong \Z_3$                    \\
$\icosa$                                      & $\complexK_\icosa$ & $[v^0, b(e^0)]$    &  $\cong \Z_5$                    & $\cong \Z_2$                          & $\langle \id \rangle$         & $\cong \Z_3$                    \\

\hline

$\langle \tetra, -o_0 \rangle$                & $\complexK_\tetra$ & $[v^0, b(e^0)]$    &  $\cong \D_3$                    & $\cong \Z_2 \times \Z_2$              & $\cong \Z_2$                  & $\cong \D_3$                    \\
$\tetra \times Z$                             & $\complexK_\octa$  & $|e^0|$            &  $\langle -a_4^2, b \rangle$     & $\langle -a_4^2 \rangle$              & $\langle -a_4^2 \rangle$      & $\cong \Z_3$                    \\
$\octa \times Z$                              & $\complexK_\octa$  & $[v^0, b(e^0)]$    &  $\cong \D_4$                    & $\cong \Z_2 \times \Z_2$              & $\cong \Z_2$                  & $\cong \D_3$                    \\
$\icosa \times Z$                             & $\complexK_\icosa$ & $[v^0, b(e^0)]$    &  $\cong \D_5$                    & $\cong \Z_2 \times \Z_2$              & $\cong \Z_2$                  & $\cong \D_3$                    \\

\hline

$\orthogonal(3)$                              &                    &                    &                                  &                                       &                               &                                 \\
$\orthogonal(2) \times Z$                     &                    & $\{ v^0 \}$        & $\langle b, -a_2 \rangle$        &                                       &                               & $\langle \SO(2), -b \rangle$    \\
$\langle \SO(2), -b \rangle$                  &                    & $\{ v^0 \}$        & $\langle -a_2 b \rangle$         &                                       &                               & $\langle \SO(2), -b \rangle$    \\
$\langle \SO(2), -a_2 \rangle$                &                    & $\{ v^0 \}$        & $\langle -a_2 \rangle$           &                                       &                               & $\SO(2)$                        \\
$\SO(3)$                                      &                    &                    &                                  &                                       &                               &                                 \\
$\orthogonal(2)$                              &                    & $\{ v^0 \}$        & $\langle b \rangle$              &                                       &                               & $\SO(2)$                        \\
$\SO(2)$                                      &                    & $\{ v^0 \}$        & $\langle \id \rangle$            &                                       &                               & $\SO(2)$                        \\
\end{tabular}}
\caption{\label{table: all about K_R} $\complexK_R, D_R,$ and
isotropy subgroups}
\end{center}
\end{table}

We summarize all results of Section \ref{section: closed subgroups},
\ref{section: simplicial complex} in Table \ref{table: all about
K_R} so that it will be repeatedly referred. In Table \ref{table:
all about K_R}, $R_{d^{-1}}$ is also calculated. In Section
\ref{section: equivariant clutching maps}, \ref{section: cyclic
dihedral cases}, we need the following lemma on isotropy subgroups:

\begin{lemma} \label{lemma: intersection of isotropy}
For each finite $R$ of Table \ref{table: introduction} such that $R
\ne \Z_n,$ $\langle a_n, -b \rangle$ for any $n,$ isotropy subgroups
$R_{d^i}$'s satisfy
\begin{enumerate}
  \item $R_{d^{-1}} \cap R_{d^i} = R_{[d^{-1}, d^i]}$
  and $[d^{-1}, d^i] \subset |\complexK_R|^{R_{d^{-1}} \cap R_{d^i}}$
  for $i \in I,$
  \item $R_{d^0} \cap R_{d^1} = R_{D_R}$
  and $D_R \subset |\complexK_R|^{R_{d^0} \cap R_{d^1}}$
\end{enumerate}
where $R_X$ for a subset $X$ of $|\complexK_R|$ is the subgroup of
$R$ fixing $X.$
\end{lemma}

\begin{proof}
For the natural $R$-action on $S^2$ and two points $x \ne \pm
x^\prime$ of $S^2,$ we have
\begin{equation}
\tag{*} R_x \cap R_{x^\prime} = R_C \quad \text{and} \quad C =
(S^2)^{R_x \cap R_{x^\prime}}
\end{equation}
where $C$ is the great circle containing $x$ and $x^\prime.$ From
this, we easily obtain proof for cases when $\complexK_R =$
$\complexK_\tetra,$ $\complexK_\octa,$ $\complexK_\icosa,$ or
$\complexK_m$ with $m \ge 3.$ Then, there are remaining 8 cases by
Lemma \ref{lemma: cases of m_R less than 3}. Four cases of these are
$\Z_n$ or $\langle a_n, -b \rangle$ for some $n$ so that four cases
are remaining. We can apply (*) to three of remaining four. The
remaining case is $\Z_n \times Z$ with $n=1.$ Proof for this is
easy. \qed
\end{proof}

\section{Equivariant clutching construction}
 \label{section: clutching construction}

Let a compact Lie group $G_\chi$ act linearly (not necessarily
effectively) on $S^2$ through a representation $\rho: G_\chi
\rightarrow \orthogonal(3).$ Assume that $\rho(G_\chi)=R$ for some
finite $R$ in Table \ref{table: introduction}. In the below, our
treatment is different according to $R.$ First, we deal with cases
when $\pr(R) = \tetra, \octa, \icosa$ in Section \ref{section:
clutching construction}$\sim$\ref{section: platonic case}. In a
similar way, we deal with cases when $\pr(R) = \Z_n, \D_n$ in
Section \ref{section: clutching construction cyclic
dihedral}$\sim$\ref{section: cyclic dihedral cases}. Then, cases of
one-dimensional $R$'s are dealt with.

Assume that $\pr(R) = \tetra, \octa, \icosa.$ Let $\lineK_R$ be the
simplicial complex $\amalg_{f \in \complexK_R} f,$ i.e. the disjoint
union of faces of $\complexK_R.$ In this section, we would consider
$|\complexK_R|$ as the quotient of the underlying space $|\lineK_R|
= \amalg_{f \in \complexK_R} |f|.$ And, we would consider an
equivariant vector bundle over $|\complexK_R|$ as an equivariant
clutching construction of an equivariant vector bundle over
$|\lineK_R|.$ For this, we would define equivariant clutching map
and its generalization preclutching map. And, we state equivalent
conditions under which a preclutching map is an equivariant
clutching map. Before these, we need introduce notations on some
relevant simplicial complices. Since we should deal with various
cases at the same time, these notations are necessary. Examples of
the notations are illustrated in Figure \ref{figure: example of
notation} and Figure \ref{figure: simple form TxZ}.

\begin{figure}[ht!]
\begin{center}
\mbox{ \subfigure[$\complexK_\tetra$ near $v^0$ seen from above]{
\begin{pspicture}(-2.5,-2.5)(2.5,2.5)\footnotesize
\pspolygon[fillstyle=solid,fillcolor=lightgray,linestyle=none](2,1)(-2,1)(0,-2)(2,1)
\psline[linewidth=1.5pt](0,0)(2,1)
\psline[linewidth=1.5pt](0,0)(-2,1)
\psline[linewidth=1.5pt](0,0)(0,-2)
\psline[linewidth=1.5pt](2,1)(-2,1)(0,-2)(2,1)


\uput[u](0,1.2){$f^0$} \uput[dr](0.9,-0.9){$f^{-1}$}
\uput[dl](-0.9,-0.9){$f^2$}

\uput[d](0.8,0.5){$e^0$} \uput[ul](1.15,-0.6){$e^1$}
\uput[r](-0.1,-0.7){$e^2$}

\uput[dl](0.1,0.1){$v^0$} \uput[ur](1.9,0.9){$v^1$}
\uput[d](0,-1.9){$v^2$}
\end{pspicture}

}

\subfigure[$\lineK_\tetra$ near $\{ \bar{v}_j^0 ~|~ j \in \Z_3 \}$]{
\begin{pspicture}(-2.5,-3.5)(2.5,1.5)\footnotesize
\pspolygon[fillstyle=solid,fillcolor=lightgray,linestyle=none](0,0)(2,1)(-2,1)(0,0)
\pspolygon[fillstyle=solid,fillcolor=lightgray,linestyle=none](-0.5,-1)(-0.5,-3)(-2.5,0)(-0.5,-1)
\pspolygon[fillstyle=solid,fillcolor=lightgray,linestyle=none](0.5,-1)(0.5,-3)(2.5,0)(0.5,-1)

\psline[linewidth=1.5pt](0,0)(2,1)(-2,1)(0,0)
\psline[linewidth=1.5pt](-0.5,-1)(-0.5,-3)(-2.5,0)(-0.5,-1)
\psline[linewidth=1.5pt](0.5,-1)(0.5,-3)(2.5,0)(0.5,-1)

\uput[u](0,1.2){$\bar{f}^0$} \uput[dr](1.4,-1.9){$\bar{f}^{-1}$}
\uput[dl](-1.4,-1.9){$\bar{f}^2$}

\uput[d](0,0.1){$\bar{v}_1^0$} \uput[ul](0.6,-1.1){$\bar{v}_0^0$}
\uput[ur](-0.6,-1.1){$\bar{v}_2^0$}

\uput[ur](2.4,-0.1){$\bar{v}_0^1$} \uput[d](0.5,-2.9){$\bar{v}_0^2$}
\uput[ur](1.9,0.9){$\bar{v}_2^1$} \uput[d](-0.5,-2.9){$\bar{v}_1^2$}

\uput[d](1.3,-0.5){$\bar{e}^0$} \uput[ul](1.65,-1.6){$\bar{e}^1$}
\uput[r](0.4,-1.5){$\bar{e}^2$}

\uput[ul](0.9,0.3){$c(\bar{e}^0)$}
\uput[l](-0.3,-1.5){$c(\bar{e}^2)$}

\end{pspicture}
} }

\mbox{ \subfigure[$\hatL_\tetra$ near $\{ \hat{v}_{j, \pm}^0 ~|~ j
\in \Z_3 \}$]{
\begin{pspicture}(-3,-3.5)(3,3)\footnotesize
\psline[linewidth=1.5pt](0.5,1)(2,1.75)
\psline[linewidth=1.5pt](-0.5,1)(-2,1.75)
\psline[linewidth=1.5pt](1.5,2.5)(-1.5,2.5)

\psline[linewidth=1.5pt](1.5,-0.75)(3,0)
\psline[linewidth=1.5pt](1,-1.5)(1,-3)
\psline[linewidth=1.5pt](3.5,-0.5)(1.75,-3)

\psline[linewidth=1.5pt](-1.5,-0.75)(-3,0)
\psline[linewidth=1.5pt](-1,-1.5)(-1,-3)
\psline[linewidth=1.5pt](-3.5,-0.5)(-1.75,-3)

\uput[u](1,-1.6){$\hat{v}_{0,-}^0$}
\uput[ul](1.6,-0.85){$\hat{v}_{0,+}^0$}
\uput[ur](2.9,-0.1){$\hat{v}_{0,-}^1$}
\uput[ur](3.4,-0.6){$\hat{v}_{0,+}^1$}
\uput[dl](2.1,-2.9){$\hat{v}_{0,-}^2$}
\uput[d](1,-2.9){$\hat{v}_{0,+}^2$}

\uput[u](-1,-1.6){$\hat{v}_{2,+}^0$}
\uput[ur](-1.6,-0.85){$\hat{v}_{2,-}^0$}
\uput[dr](-2.1,-2.9){$\hat{v}_{1,+}^2$}
\uput[d](-1,-2.9){$\hat{v}_{1,-}^2$}

\uput[d](0.5,1.1){$\hat{v}_{1,-}^0$}
\uput[d](-0.5,1.1){$\hat{v}_{1,+}^0$}
\uput[r](1.9,1.75){$\hat{v}_{2,+}^1$}
\uput[r](1.4,2.5){$\hat{v}_{2,-}^1$}

\uput[dr](2.15,-0.365){$\hat{e}^0$}
\uput[dr](2.65,-1.65){$\hat{e}^1$} \uput[r](0.9,-2.25){$\hat{e}^2$}

\uput[l](-0.9,-2.25){$c(\hat{e}^2)$}
\uput[ul](1.35,1.275){$c(\hat{e}^0)$}

\psline[linearc=2]{<->}(2,0)(1.9,0.5)(1.5,1)
\uput[ur](1.8,0.5){$|c|$}

\psline[linearc=1]{<->}(-0.5,-2.25)(0,-2.45)(0.5,-2.25)
\uput[d](0,-2.45){$|c|$}

\psline[linearc=2]{<->}(-1.5,1)(-1.9,0.5)(-2,0)
\uput[ul](-1.8,0.5){$|c|$}

\end{pspicture}
} }
\end{center}
\caption{\label{figure: example of notation} Relation between
$\complexK_R,$ $\lineK_R,$ $\hatL_R$ in the case when $\complexK_R =
\complexK_\tetra$ and $j_R=3$}
\end{figure}

First, we define some notations on $\complexK_R$ and $\lineK_R.$ We
denote simply by $\pi,$ $|\pi|$ natural quotient maps from
$\lineK_R,$ $|\lineK_R|$ to $\complexK_R, |\complexK_R|,$
respectively. By definition, $|\pi| ~ \big|_{|f|}$ is bijective for
each face $f \in \lineK_R.$ From this, the $G_\chi$-actions on
$\complexK_R,$ $|\complexK_R|$ induce $G_\chi$-actions on
$\lineK_R,$ $|\lineK_R|$ so that $\pi,$ $|\pi|$ are equivariant,
respectively. We use notations $\bar{v},$ $\bar{e},$ $\bar{f}$ to
denote a vertex, an edge, a face of $\lineK_R,$ respectively. We use
the notation $\bar{x}$ to denote an arbitrary point of
$|\lineK_R^{(1)}|.$ When $\bar{v},$ $\bar{e},$ $\bar{f},$ $\bar{x}$
are understood, we use notations $v, e, f, x$ to denote images $\pi
(\bar{v}),$ $\pi (\bar{e}),$ $\pi (\bar{f}),$ $|\pi|(\bar{x}),$
respectively. Denote by $\bar{f}^{-1},$ $\bar{f}^i$ faces of
$\lineK_R$ such that $\pi(\bar{f}^{-1}) = f^{-1}$ and
$\pi(\bar{f}^i) = f^i,$ and denote by $\bar{v}^i,$ $\bar{e}^i$
simplices of $\bar{f}^{-1}$ such that $\pi(\bar{v}^i) = v^i$ and
$\pi(\bar{e}^i) = e^i.$ And, denote by $\bar{d}^0,$ $\bar{d}^1$
points of $|\bar{f}^{-1}|$ such that $|\pi|(\bar{d}^0) = d^0$ and
$|\pi|(\bar{d}^1) = d^1,$ and denote by $\bar{D}_R$ the path
$|\pi|^{-1}(D_R) \cap |\bar{f}^{-1}|,$ i.e. $\bar{D}_R = [\bar{d}^0,
\bar{d}^1].$ Define the integer $j_R$ as the cardinality of
$\pi^{-1}(v^i)$ for $i \in \Z_3,$ i.e. $j_R = 3,$ $4,$ $5$ according
to $\complexK_R = \complexK_\tetra,$ $\complexK_\octa,$
$\complexK_\icosa,$ respectively. Let $B$ be the subset $\{ ~
b(\bar{f}) ~|~ \bar{f} \in \lineK_R ~ \}$ of $|\lineK_R|$ on which
$R$ (and $G_\chi$) acts transitively by Lemma \ref{lemma: K_R for
platonic}, and $B$ is often confused with $|\pi|(B) = \{ ~ b(f) ~|~
f \in \complexK_R ~ \}.$ So far, we have defined superscript $i$ for
simplices in $\lineK_R.$ Next, we define $\bar{x}_j$ with subscript
$j$ for any point $\bar{x}$ of $|\lineK_R^{(1)}|.$
\begin{notation}
~
\begin{enumerate}
  \item For a vertex $\bar{v}$ in $\lineK_R^{(1)}$ and $v =
\pi(\bar{v}),$ we label vertices in $\pi^{-1} (v)$ with $\bar{v}_j$
to satisfy
\begin{enumerate}
  \item[i)] $\pi^{-1} (v) = \{ \bar{v}_j ~|~ j \in \Z_{j_R} \},$
  \item[ii)] $\bar{v}_0 = \bar{v},$
  \item[iii)] in each face $|\bar{f}_j|$ containing $\bar{v}_j$
  for $j \in \Z_{j_R},$ we can take a small neighborhood
  $U_j$ of $\bar{v}_j$ so that
  $|\pi|(U_j)$'s are arranged in the counterclockwise way around
  $v.$
\end{enumerate}
  \item For a non-vertex $\bar{x}$ in $|\lineK_R^{(1)}|$
  and $x =
|\pi| (\bar{x}),$ we label two points in $|\pi|^{-1} (x)$ with $\{
\bar{x}_j ~|~ j \in \Z_2 \}$ to satisfy $\bar{x}_0 = \bar{x}.$
\end{enumerate}
For simplicity, we denote $(\bar{v}^i)_j,$ $(\bar{d}^i)_j$ by
$\bar{v}_j^i,$ $\bar{d}_j^i,$ respectively.
\end{notation}

We need introduce two more simplicial complices. Denote by
$\lineL_R$ and $\hatL_R$ the 1-skeleton $\lineK_R^{(1)}$ of
$\lineK_R$ and the disjoint union $\amalg_{\bar{e} \in \lineL_R} ~
\bar{e},$ respectively. Then, $\lineL_R$ is a subcomplex of
$\lineK_R,$ and can be regarded as a quotient of $\hatL_R.$ These
relations are expressed by two natural simplicial maps
\begin{equation*}
\imath_\lineL : \lineL_R \rightarrow \lineK_R, \quad p_\lineL :
\hatL_R \rightarrow \lineL_R
\end{equation*}
where $\imath_\lineL$ is the inclusion and $p_\lineL$ is the
quotient map whose preimage of each vertex and edge of $\lineL_R$
consists of two vertices and one edge of $\hatL_R,$ respectively.
Two maps on underlying spaces are denoted by
\begin{equation*}
\imath_{|\lineL|} : |\lineL_R| \rightarrow |\lineK_R|, \quad
p_{|\lineL|} : |\hatL_R| \rightarrow |\lineL_R|.
\end{equation*}
The $G_\chi$-actions on $\lineK_R,$ $|\lineK_R|$ naturally induce
$G_\chi$-actions on these relevant simplicial complices $\lineL_R,
\hatL_R$ and their underlying spaces. We need introduce notations on
simplices of $\hatL_R$ and points of $|\hatL_R|.$ We use notations
$\hat{v}$ and $\hat{e}$ to denote a vertex and an edge of $\hatL_R,$
respectively. And, we use the notation $\hat{x}$ to denote an
arbitrary point in $|\hatL_R|.$ When $\hat{v},$ $\hat{e},$ $\hat{x}$
are understood, we use notations $\bar{v},$ $\bar{e},$ $\bar{x}$ to
denote $p_\lineL (\hat{v}),$ $p_\lineL (\hat{e}),$ $p_{|\lineL|}
(\hat{x}),$ respectively. Two edges $\hat{e}, \hat{e}^\prime$ of
$\hatL_R$ (and their images $\bar{e}, \bar{e}^\prime$ in $\lineL_R$)
are called \textit{adjacent} if $\hat{e} \ne \hat{e}^\prime$ and
$\pi( p_\lineL (\hat{e}) ) = \pi( p_\lineL (\hat{e}^\prime) ).$ And,
two faces $\bar{f}, \bar{f}^\prime$ of $\lineK_R$ are called
\textit{adjacent} if their images $f, f^\prime$ are adjacent.

Next, we introduce superscript $i$ and subscripts $+, -$ for
vertices and edges of $\hatL_R.$ Before it, we introduce a
simplicial map. Let $c: \hatL_R \rightarrow \hatL_R$ be the
simplicial map whose underlying space map $|c| : |\hatL_R|
\rightarrow |\hatL_R|$ is defined as
\begin{itemize}
  \item[] for any adjacent $\hat{e}, \hat{e}^\prime \in \hatL_R,$
        each point $\hat{x}$ in $|\hat{e}|$ is sent to the point $|c|(\hat{x})$
        in $|\hat{e}^\prime|$ to satisfy $|\pi|(p_{|\lineL|}(\hat{x}))
        = |\pi|(p_{|\lineL|}(c(\hat{x}))).$
\end{itemize}
For example, $\hat{e}$ and $c(\hat{e})$ are adjacent for any edge
$\hat{e}$ in $\hatL_R.$ Easily, $c$ and $|c|$ are
$G_\chi$-equivariant. For notational simplicity, we define $c$ also
on edges of $\lineL_R$ to satisfy $c( p_{\lineL}(\hat{e}) ) =
p_{\lineL}( c(\hat{e}) )$ for each edge $\hat{e}.$

\begin{notation}
~
\begin{enumerate}
  \item For a vertex $\bar{v}$ in $\lineL_R,$ we label two vertices in
$p_{\lineL}^{-1} (\bar{v})$ with $\hat{v}_\pm$ to satisfy
\begin{equation*}
p_{\lineL} (c(\hat{v}_+)) = \bar{v}_1 ~ \text{ and } ~ p_{\lineL}
(c(\hat{v}_-)) = \bar{v}_{-1} = \bar{v}_{j_R-1}.
\end{equation*}
  \item For a non-vertex $\bar{x}$ in $|\lineL_R|,$
  we label the point in $p_{|\lineL|}^{-1} (\bar{x})$ with
  $\hat{x}_+$ or $\hat{x}_-,$ i.e. $\hat{x}_+ = \hat{x}_-.$
\end{enumerate}

For simplicity, denote $\hat{x}_\pm$ for $\bar{x}= \bar{v}^i,
\bar{v}_j^i, \bar{d}^i, \bar{d}_j^i$ by $\hat{v}_\pm^i,
\hat{v}_{j,\pm}^i, \hat{d}_\pm^i, \hat{d}_{j,\pm}^i,$ respectively.
So, if $\bar{d}^i$ is a barycenter of an edge, then $\hat{d}_+^i =
\hat{d}_-^i.$ And, denote by $\hat{e}^i$ the edge in $\hatL_R$ such
that $p_\lineL (\hat{e}^i) = \bar{e}^i$ for $i \in \Z_3.$
\end{notation}

\begin{figure}[ht!]
\begin{center}
\mbox{ \subfigure[$\complexK_\octa$]{
\begin{pspicture}(-3,-3)(3,2.5)\footnotesize
\psline[linewidth=0.5pt](0,2.5)(-3,-0.5)(0,-2.5)(3,0.5)(0,2.5)
\psline[linewidth=0.5pt](-3,-0.5)(1,-0.5)(3,0.5)
\psline[linewidth=0.5pt](0,2.5)(1,-0.5)(0,-2.5)
\psline[linewidth=0.5pt, linestyle=dotted](-3,-0.5)(-1,0.5)(3,0.5)
\psline[linewidth=0.5pt, linestyle=dotted](0,2.5)(-1,0.5)(0,-2.5)

\uput[l](-3,-0.5){$v^0$} \uput[ul](1,-0.5){$v^1$}
\uput[d](0,-2.5){$v^2$}

\uput[d](-1,-0.4){$e^0$} \uput[l](0.6,-1.5){$e^1$}
\uput[dl](-1.4,-1.5){$e^2$}

\put(-0.7,-1.5){$f^{-1}$} \put(-0.7,0.5){$f^0$}
\put(1.1,-0.9){$f^1$}

\end{pspicture}
} }

\mbox{ \subfigure[$\hatL_R$ near $\hat{\mathbf{D}}_R$]{
\begin{pspicture}(-5,-3)(5,3.5)\footnotesize
\psline[linewidth=3pt](-1,1)(1,1)
\psline[linewidth=3pt](-1,-1)(1,-1)

\psline[linewidth=0.5pt](-1.5, 1.5)(-1,3)
\psline[linewidth=0.5pt](1.5, 1.5)(1,3)
\psline[linewidth=0.5pt](-1.5, -1.5)(-1,-3)
\psline[linewidth=0.5pt](1.5, -1.5)(1,-3)

\psline[linewidth=0.5pt](3.5,1.5)(4,3)
\psline[linewidth=0.5pt](3.5,-1.5)(4,-3)
\psline[linewidth=0.5pt](-3.5,1.5)(-4,3)
\psline[linewidth=0.5pt](-3.5,-1.5)(-4,-3)

\psline[linewidth=0.5pt](4,1)(5.5,1)
\psline[linewidth=0.5pt](-4,1)(-5.5,1)
\psline[linewidth=0.5pt](4,-1)(5.5,-1)
\psline[linewidth=0.5pt](-4,-1)(-5.5,-1)

\psdots[dotsize=5pt](-1,1) \psdots[dotsize=5pt](1,1)
\psdots[dotsize=5pt](-1,-1) \psdots[dotsize=5pt](1,-1)

\psdots[dotsize=5pt](-1.5,1.5) \psdots[dotsize=5pt](1.5,1.5)
\psdots[dotsize=5pt](-1.5,-1.5) \psdots[dotsize=5pt](1.5,-1.5)

\psdots[dotsize=5pt](3.5,1.5) \psdots[dotsize=5pt](3.5,-1.5)
\psdots[dotsize=5pt](-3.5,1.5) \psdots[dotsize=5pt](-3.5,-1.5)

\psdots[dotsize=5pt](4,1) \psdots[dotsize=5pt](-4,1)
\psdots[dotsize=5pt](4,-1) \psdots[dotsize=5pt](-4,-1)

\uput[dl](-1.5,-1.5){$\hat{v}_{0,-}^0$}
\uput[ul](-1,-1){$\hat{v}_{0,+}^0$}
\uput[dl](-1,1){$\hat{v}_{1,-}^0$}
\uput[ul](-1.5,1.5){$\hat{v}_{1,+}^0$}
\uput[ur](-3.5,1.5){$\hat{v}_{2,-}^0$}
\uput[dr](-4,1){$\hat{v}_{2,+}^0$}
\uput[ur](-4,-1){$\hat{v}_{3,-}^0$}
\uput[dr](-3.5,-1.5){$\hat{v}_{3,+}^0$}

\uput[ur](1,-1){$\hat{v}_{0,-}^1$}
\uput[dr](1.5,-1.5){$\hat{v}_{0,+}^1$}
\uput[dl](3.5,-1.5){$\hat{v}_{1,-}^1$}
\uput[ul](4,-1){$\hat{v}_{1,+}^1$} \uput[dl](4,1){$\hat{v}_{2,-}^1$}
\uput[ul](3.5,1.5){$\hat{v}_{2,+}^1$}
\uput[ur](1.5,1.5){$\hat{v}_{3,-}^1$}
\uput[dr](1,1){$\hat{v}_{3,+}^1$}

\uput[u](0,1){$c(\hat{e}^0)$} \uput[d](0,-1){$\hat{e}^0$}
\uput[r](-1.25,-2.25){$\hat{e}^2$} \uput[l](1.25,-2.25){$\hat{e}^1$}
\uput[l](-3.75,-2.25){$c(\hat{e}^2)$}
\uput[r](3.75,-2.25){$c(\hat{e}^1)$}
\end{pspicture} } }
\end{center}
\caption{\label{figure: simple form TxZ} $\hatL_R$ and
$\hat{\mathbf{D}}_R$ in the case when $R = \tetra \times Z,$
$\complexK_R = \complexK_\octa,$ $D_R = |e^0|$}
\end{figure}

Until now, we have finished introducing notations in Figure
\ref{figure: example of notation}. By using these notations, we
introduce one-dimensional fundamental domain in $|\hatL_R|.$ For
$\bar{D}_R= [\bar{d}^0, \bar{d}^1] \subset |\lineK_R|,$ we define
$\hat{D}_R$ in $|\hatL_R|$ as $[\hat{d}_+^0, \hat{d}_-^1]$ so that
$p_{|\lineL|}( \hat{D}_R )=\bar{D}_R.$ And, denote by
$\hat{\mathbf{D}}_R$ the set $( |\pi| \circ p_{|\lineL|})^{-1} ( D_R
)$ in $|\hatL_R|$ which is equal to
\begin{equation*}
[\hat{d}_+^0, \hat{d}_-^1] ~ \bigcup ~ |c| \big( [\hat{d}_+^0,
\hat{d}_-^1] \big) ~ \bigcup ~ \Big( \bigcup_{v \in D_R} ( \pi \circ
p_{\lineL} )^{-1} (v) \Big).
\end{equation*}
The union of thick points and edges of Figure \ref{figure: simple
form TxZ}.(b) is $\hat{\mathbf{D}}_R$ in the case of $R = \tetra
\times Z.$

For convenience in calculation, we parameterize each edge of
$|\hatL_R|$ linearly by $s \in [0,1]$ to satisfy
\begin{enumerate}
  \item $\hat{v}_+ = 0,$ $\hat{v}_- = 1,$
  $b(\hat{e}) = 1/2$ for each vertex $\bar{v}$ of $\lineK_R$
  and each edge $\hat{e}$ of $\hatL_R,$
  \item $|c|(s) = 1-s$ for each edge $\hat{e}$ of $\hatL_R$
  and $s \in |\hat{e}|.$
\end{enumerate}
We repeatedly use this parametrization.

Now, we describe an equivariant vector bundle over $|\complexK_R|$
as an equivariant clutching construction of an equivariant vector
bundle over $|\lineK_R|.$ Let $V_B$ be a $G_\chi$-vector bundle over
$B$ such that $(\res_H^{G_\chi} V_B)|_{b(\bar{f})}$ is
$\chi$-isotypical at each $b(\bar{f})$ in $B.$ If we denote by
$V_{\bar{f}}$ the isotropy representation of $V_B$ at each
$b(\bar{f}),$ then $V_B \cong G_\chi \times_{(G_\chi)_{b(\bar{f})}}
V_{\bar{f}}$ because $G_\chi$ acts transitively on $B.$ And,
$\res_H^{(G_\chi)_{b(\bar{f})}} V_{\bar{f}}$'s are all isomorphic
because they are all $\chi$-isotypical. We define $\Vect_{G_\chi}
(|\complexK_R|, \chi)_{V_B}$ as the set
\begin{equation*}
\big\{ \quad [E] \in \Vect_{G_\chi} (|\complexK_R|, \chi) \quad
\big| \quad E|_{B} \cong V_B \quad \big\}.
\end{equation*}
Similarly, $\Vect_{G_\chi} (|\lineK_R|, \chi)_{V_B}$ is defined.
Observe that $\Vect_{G_\chi} (|\lineK_R|, \chi)_{V_B}$ has the
unique element $[F_{V_B}]$ for the bundle $F_{V_B} = G_\chi
\times_{(G_\chi)_{b(\bar{f}^{-1})}} ( |\bar{f}^{-1}| \times
V_{\bar{f}^{-1}} )$ because $|\lineK_R| \cong G_\chi
\times_{(G_\chi)_{b(\bar{f}^{-1})}} |\bar{f}^{-1}|$ is equivariant
homotopically equivalent to $B.$ Henceforward, we use
trivializations
\begin{equation}
\label{equation: trivialization}
\begin{array}{ll} |\bar{f}| \times
V_{\bar{f}} & \quad \text{ for } \big(
\res_{(G_\chi)_{b(\bar{f})}}^{G_\chi} F_{V_B} \big)
\big|_{|\bar{f}|}, \\
|\bar{e}| \times
\res_{(G_\chi)_{b(\bar{e})}}^{(G_\chi)_{b(\bar{f})}} V_{\bar{f}} &
\quad \text{ for } \big( \res_{(G_\chi)_{b(\bar{e})}}^{G_\chi}
F_{V_B} \big) \big|_{|\bar{e}|}
\end{array}
\end{equation}
for each face $\bar{f}$ and edge $\bar{e} \in \bar{f}.$ Observe that
each $E \in \Vect_{G_\chi} (|\complexK_R|, \chi)_{V_B}$ can be
constructed by gluing the pull-back bundle $|\pi|^* E$ along edges.
Let us describe this more precisely. Let $\widetilde{|\pi|}$ be the
bundle morphism covering $|\pi|$
\begin{equation*}
\begin{CD}
|\pi|^* E      @>\widetilde{|\pi|}>>      E                 \\
@VVV     @VVV   \\
|\lineK_R|  @>|\pi|>>    |\complexK_R|.        \\
\end{CD}
\end{equation*}
Consider the following equivalence relation $\sim$ on vectors of
$|\pi|^* E$ :
\begin{itemize}
  \item[] for any two $u,$ $u^\prime$ in $|\pi|^* E,$
  $u \sim u^\prime$ if and only if
  $\widetilde{|\pi|} (u) = \widetilde{|\pi|} (u^\prime).$
\end{itemize}
Since $\widetilde{|\pi|}$ is equivariant, the quotient $|\pi|^* E /
\sim$ of $|\pi|^* E$ through the relation delivers the equivariant
vector bundle structure inherited from $|\pi|^* E.$ Here, note that
it suffices to define the relation only on vectors in $|\pi|^* E
\big|_{|\lineL_R|}$ to define the quotient which is trivially
isomorphic to $E.$ We call the construction of the bundle $|\pi|^* E
/ \sim$ \textit{equivariant clutching construction}. Since $F_{V_B}
\cong |\pi|^* E$ for any $E,$ we may rewrite the construction by
using $F_{V_B}$ instead of $|\pi|^* E.$ Pick an equivariant
isomorphism $\bar{A} : F_{V_B} \rightarrow |\pi|^* E.$ Then, the
relation $\sim$ induces the following equivalence relation
$\sim^\prime$ on vectors of $F_{V_B}$ :
\begin{itemize}
  \item[] for any two $u,$ $u^\prime$ in $F_{V_B},$
  $u \sim^\prime u^\prime$ if and only if
  $\bar{A}(u) \sim \bar{A}(u^\prime).$
\end{itemize}
Then, the quotient $F_{V_B} / \sim^\prime$ delivers the equivariant
vector bundle structure inherited from $F_{V_B},$ and trivially $(
F_{V_B} / \sim^\prime ) \cong E.$ Let us describe the relation on
vectors in $F_{V_B} \big|_{|\lineL_R|}$ more precisely. By using
trivialization (\ref{equation: trivialization}) of $F_{V_B},$ the
quotient $F_{V_B} / \sim^\prime$ can be also constructed by gluing
$F_{V_B}$ along edges through
\begin{equation}
\label{equation: gluing} |\bar{e}| \times V_{\bar{f}}
\longrightarrow |c(\bar{e})| \times V_{\bar{f}^\prime}, \quad ( ~
p_{|\lineL|}(\hat{x}), u ~ ) \mapsto \Big( ~ p_{|\lineL|} \big(
|c|(\hat{x}) \big), ~ \varphi_{\hat{e}} (\hat{x}) ~ u ~ \Big)
\end{equation}
via some continuous maps
\begin{equation*}
\varphi_{\hat{e}} : |\hat{e}| \rightarrow \Iso ( V_{\bar{f}},
V_{\bar{f}^\prime} )
\end{equation*}
for each edge $\hat{e},$ $\hat{x} \in |\hat{e}|,$ $u \in
V_{\bar{f}}$ where $\bar{e} = p_{\lineL}(\hat{e})$ and $\bar{e} \in
\bar{f},$ $c(\bar{e}) \in \bar{f}^\prime.$ Here, the notation $\Iso$
with no subscript means the set of nonequivariant isomorphisms. The
union $\Phi = \bigcup_{\hat{e} \in \hatL_R} \varphi_{\hat{e}}$ is
called an \textit{equivariant clutching map} of $E$ with respect to
$V_B.$ The relation $\sim^\prime$ on vectors in $F_{V_B}
|_{|\lineL_R|}$ is defined by $\Phi,$ and the quotient $F_{V_B} /
\sim^\prime$ is denoted by $F_{V_B} / \Phi.$ And, the equivariant
vector bundle $F_{V_B} / \Phi$ is called \textit{determined} by
$\Phi$ with respect to $V_B.$ When we use the phrase `with respect
to $V_B$', it is assumed that we use the bundle $F_{V_B}$ and its
trivialization (\ref{equation: trivialization}) in gluing.
Equivariance of $\widetilde{|\pi|}$ and $\bar{A}$ guarantees
equivariance of $\Phi,$ i.e.
\begin{equation*}
(g \cdot \Phi)(\hat{x}) = g \Phi(g^{-1} \hat{x}) g^{-1} =
\Phi(\hat{x})
\end{equation*}
for all $g \in G_\chi, \hat{x} \in |\hatL_R|.$ We denote by $p_\Phi$
the quotient map from $F_{V_B}$ to $F_{V_B} / \Phi.$ Here, note that
$\Phi$ is defined on $|\hatL_R|.$ That is why we define $\hatL_R.$
Sometimes, we regard $\Phi$ as the map
\begin{equation*}
p_{|\lineL|}^* F_{V_B} \rightarrow p_{|\lineL|}^* F_{V_B},
\quad(\hat{x}, u) \mapsto \big( |c|(\hat{x}), \Phi (\hat{x}) u \big)
\end{equation*}
by using trivialization (\ref{equation: trivialization}) for each
$(\hat{x}, u) \in |\hat{e}| \times V_{\bar{f}}$ where $\bar{e} \in
\bar{f}.$  An equivariant clutching map of some bundle in
$\Vect_{G_\chi} (|\complexK_R|, \chi)_{V_B}$ with respect to $V_B$
is called simply an \textit{equivariant clutching map} with respect
to $V_B,$ and let $\Omega_{V_B}$ be the set of all equivariant
clutching maps with respect to $V_B.$ In the next section, we will
see that we need calculate the (nonequivariant) homotopy
$\pi_0(\Omega_{V_B})$ to classify equivariant vector bundles. To do
it, we need to restrict an equivariant clutching map in
$\Omega_{V_B}$ to $\hat{D}_R.$ We explain for this. Let
$\Omega_{\hat{D}_R, V_B}$ be the set
\begin{equation*}
\{ \Phi |_{\hat{D}_R} ~ | ~ \Phi \in \Omega_{V_B} \}.
\end{equation*}
If two equivariant clutching maps coincide on $\hat{D}_R,$ then they
are identical by equivariance and definition of one-dimensional
fundamental domain. So, the restriction map $\Omega_{V_B}
\rightarrow \Omega_{\hat{D}_R, V_B}, ~ \Phi \mapsto
\Phi|_{\hat{D}_R}$ is bijective, and we obtain a bijection $\pi_0
(\Omega_{V_B}) \cong \pi_0 (\Omega_{\hat{D}_R, V_B})$ between two
homotopies. It is conceivable that it is easier to deal with
$\Omega_{\hat{D}_R, V_B}$ than $\Omega_{V_B}$ because of smaller
domain of definition. This is why we restrict an equivariant
clutching map to $\hat{D}_R.$ We call a map $\Phi$ in $\Omega_{V_B}$
the \textit{extension} of $\Phi|_{\hat{D}_R}$ in $\Omega_{\hat{D}_R,
V_B}.$ And, denote the bundle $F_{V_B} / \Phi$ also by $F_{V_B} / ~
\Phi|_{\hat{D}_R}.$

Next, we define preclutching map, a generalization of equivariant
clutching map. Let $C^0 (|\hatL_R|, V_B)$ be the set of continuous
functions $\Phi$ on $|\hatL_R|$ satisfying
$\Phi|_{|\hat{e}|}(\hat{x}) \in \Iso(V_{\bar{f}},
V_{\bar{f}^\prime})$ for each $\hat{e}$ and $\hat{x} \in |\hat{e}|$
where $\bar{e} \in \bar{f},$ $c(\bar{e}) \in \bar{f}^\prime.$ Note
that we can define the quotient $F_{V_B} / \Phi$ also for any $\Phi
\in C^0 (|\hatL_R|, V_B)$ as we have done in (\ref{equation:
gluing}) though $F_{V_B} / \Phi$ need not deliver a suitable
equivariant vector bundle structure or even nonequivariant vector
bundle structure. Let $C^0 (\hat{D}_R, V_B)$ be the set
\begin{equation*}
\Big\{ ~ \Phi |_{\hat{D}_R} ~ \Big| ~ \Phi \in C^0 \big( |\hatL_R|,
V_B \big) ~ \Big\}.
\end{equation*}
A function $\Phi$ in $C^0 (|\hatL_R|, V_B)$ or a function
$\Phi_{\hat{D}_R}$ in $C^0 (\hat{D}_R, V_B)$ is called a
\textit{preclutching map} with respect to $V_B.$ Then, it is a
natural question under which conditions a preclutching map becomes
an equivariant clutching map. We can answer this question for a
preclutching map in $C^0 (|\hatL_R|, V_B).$ A preclutching map
$\Phi$ in $C^0 (|\hatL_R|, V_B)$ is an equivariant clutching map
with respect to $V_B$ if and only if it satisfies the following
conditions:
\begin{itemize}
  \item[N1.] $\Phi(|c|(\hat{x})) = \Phi(\hat{x})^{-1}$ for each $\hat{x} \in |\hatL_R|,$
  \item[N2.] For each vertex $\bar{v} \in \lineK_R,$
  \begin{equation*}
  \Phi(\hat{v}_{j_R-1,+}) \cdots \Phi(\hat{v}_{j,+}) \cdots
  \Phi(\hat{v}_{0,+}) = \id
  \end{equation*}
  for $j \in \Z_{j_R},$
  \item[E1.] $\Phi(g\hat{x}) = g \Phi(\hat{x}) g^{-1}$ for each $\hat{x} \in |\hatL_R|, g \in G_\chi.$
\end{itemize}
We explain for this more precisely. As a slight generalization of
the classical result \cite[p. 20, 21]{At}, if $\Phi$ satisfies
Condition N1., N2., then the quotient $F_{V_B} / \Phi$ becomes a
nonequivariant vector bundle though it need not be an equivariant
vector bundle. Moreover, if $\Phi$ also satisfies Condition E1.,
then $F_{V_B} / \Phi$ becomes an equivariant vector bundle so that
$\Phi$ is an equivariant clutching map with respect to $V_B.$ We
will answer the same question for a preclutching map in $C^0
(\hat{D}_R, V_B)$ in Section \ref{section: equivariant clutching
maps}.

\section{Relations among
$\Vect_{G_\chi} (S^2, \chi)_{V_B},$ $A_{G_\chi} (S^2, \chi),$ $\pi_0
( \Omega_{ V_B } )$}
 \label{section: relation}

In this section, we investigate relations among
\begin{equation*}
\Vect_{G_\chi} (S^2, \chi)_{V_B}, \quad A_{G_\chi} (S^2, \chi),
\quad \pi_0 ( \Omega_{V_B} ).
\end{equation*}
Our classification of the paper is based on these relations. Before
it, we state two basic facts on equivariant vector bundles. First,
two equivariantly homotopic equivariant clutching maps give
isomorphic equivariant vector bundles.
\begin{lemma}
 \label{lemma: homotopy gives isomorphism}
For two maps $\Phi$ and $\Phi^\prime$ in $\Omega_{V_B},$ if $\Phi$
and $\Phi^\prime$ are homotopic in $\Omega_{V_B},$ i.e. $[\Phi] =
[\Phi^\prime]$ in $\pi_0 ( \Omega_{V_B} ),$ then $\big[ F_{V_B} /
\Phi \big] = \big[ F_{V_B} / \Phi^\prime \big]$ in $\Vect_{G_\chi}
(S^2, \chi)_{V_B}.$
\end{lemma}

\begin{proof}
This is a slight generalization of the classical result \cite[Lemma
1.4.6. and Section 1.6.]{At}. So, we omit the proof. \qed
\end{proof}

Lemma \ref{lemma: homotopy gives isomorphism} gives a sufficient
condition for isomorphism. Sometimes, we need an equivalent
condition. When we consider a map in $\Omega_{V_B}$ as defined on
$p_{|\lineL|}^* F_{V_B},$ we have the following equivalent
condition:
\begin{lemma}
 \label{lemma: equivalent condition for isomorphism}
For any $\Phi$ and $\Phi^\prime$ in $\Omega_{V_B},$ $\big[ F_{V_B} /
\Phi \big] = \big[ F_{V_B} / \Phi^\prime \big]$ in $\Vect_{G_\chi}
(S^2, \chi)_{V_B}$ if and only if there is a $G_\chi$-isomorphism
$\Theta : F_{V_B} \rightarrow F_{V_B}$ such that $( p_{|\lineL|}^*
\Theta ) \Phi = \Phi^\prime ( p_{|\lineL|}^* \Theta )$ where
$p_{|\lineL|}^* \Theta : p_{|\lineL|}^* F_{V_B} \rightarrow
p_{|\lineL|}^* F_{V_B} $ is the pull-back of $\Theta.$
\end{lemma}

\[
\begin{CD}
p_{|\lineL|}^* F_{V_B}   @>{p_{|\lineL|}^* \Theta}>> p_{|\lineL|}^* F_{V_B} \\
@VV{\Phi}V    @VV{\Phi^\prime}V \\
p_{|\lineL|}^* F_{V_B}   @>{p_{|\lineL|}^* \Theta}>> p_{|\lineL|}^* F_{V_B} \\
\end{CD}
\]

\begin{proof}
This is also a slight generalization of the classical result
\cite[p. 22, (ii) and Section 1.6.]{At}. First, we prove
sufficiency. Let $A: F_{V_B} / \Phi \rightarrow F_{V_B} /
\Phi^\prime$ be a $G_\chi$-isomorphism. Then, we can show that there
exists the $G_\chi$-isomorphism $\bar{A} : F_{V_B} \rightarrow
F_{V_B}$ satisfying the following commutative diagram
\begin{equation}
\label{diagram: endomorphism}
\begin{CD}
F_{V_B}         @>\bar{A}>>             F_{V_B}   \\
@VV{p_\Phi}V    @VV{p_{\Phi^\prime}}V             \\
F_{V_B} / \Phi  @>A>>                   F_{V_B} / \Phi^\prime.  \\
\end{CD}
\end{equation}
From this, $( p_{|\complexL|}^* \bar{A} ) \Phi = \Phi^\prime (
p_{|\complexL|}^* \bar{A} )$ is obtained.

Next, we prove necessity. If we put $\bar{A} = \Theta,$ there exists
the unique $G_\chi$-isomorphism $A$ satisfying (\ref{diagram:
endomorphism}) by the assumption. So, we obtain a proof. \qed
\end{proof}

Consider the map $\imath_\Omega : \pi_0 ( \Omega_{V_B} ) \rightarrow
\Vect_{G_\chi} (S^2, \chi)_{V_B}$ mapping $[\Phi]$ to $\big[ F_{V_B}
/ \Phi \big].$ This is well-defined by Lemma \ref{lemma: homotopy
gives isomorphism}, and also surjective because each bundle in
$\Vect_{G_\chi} (S^2, \chi)_{V_B}$ can be considered as an
equivariant clutching construction. Then, the map $p_\Omega : \pi_0
( \Omega_{V_B} ) \rightarrow A_{G_\chi} (S^2, \chi)$ defined as
$p_\Omega = p_\vect \circ \imath_\Omega$ satisfies the following
diagram:
\begin{equation}
\label{equation: diagram} \SelectTips{cm}{} \xymatrix{ \pi_0 (
\Omega_{V_B} ) \ar[r]^-{\imath_\Omega}
\ar[dr]_-{p_\Omega} & \Vect_{G_\chi} (S^2, \chi)_{V_B} \\
& A_{G_\chi} (S^2, \chi) \ar@{<-}[u]_-{p_\vect} }.
\end{equation}
Let $p_{\pi_0} : \Omega_{V_B} \rightarrow \pi_0 ( \Omega_{V_B} )$ be
the natural quotient map. For different elements in $A_{G_\chi}
(S^2, \chi),$ their preimages through $(p_\Omega \circ
p_{\pi_0})^{-1}$ do not intersect each other so that we obtain a
decomposition of $\Omega_{V_B}.$ We describe this decomposition more
precisely. For each $(W_{d^i})_{i \in I^+} \in A_{G_\chi} (S^2,
\chi),$ put
\begin{equation*}
V_B = G_\chi \times_{(G_\chi)_{d^{-1}}} W_{d^{-1}}, \quad F_{V_B} =
G_\chi \times_{(G_\chi)_{d^{-1}}} ( |\bar{f}^{-1}| \times W_{d^{-1}}
),
\end{equation*}
and by using these define $\Omega_{(W_{d^i})_{i \in I^+}}$ as the
subset $(p_\Omega \circ p_{\pi_0})^{-1} \big( (W_{d^i})_{i \in I^+}
\big )$ of $\Omega_{V_B}.$ Henceforward, we will use these $V_B$ and
$F_{V_B}$ whenever we deal with $\Omega_{(W_{d^i})_{i \in I^+}}.$
Then, given a bundle $V_B,$ the set $\Omega_{V_B}$ is equal to the
disjoint union
\begin{equation*}
\bigcup_{(W_{d^i})_{i \in I^+} \in A_{G_\chi} (S^2, \chi) \text{
with } W_{d^{-1}} = V_{\bar{f}^{-1}} } \quad \Omega_{(W_{d^i})_{i
\in I^+}}.
\end{equation*}
Since $\Omega_{V_B}$ and $\Omega_{\hat{D}_R, V_B}$ are in one-to-one
correspondence, we may consider $\imath_\Omega,$ $p_\Omega,$
$p_{\pi_0}$ as defined also on $\pi_0 ( \Omega_{\hat{D}_R, V_B} ),$
$\pi_0 ( \Omega_{\hat{D}_R, V_B} ),$ $\Omega_{\hat{D}_R, V_B},$
respectively. So, if we define $\Omega_{\hat{D}_R, (W_{d^i})_{i \in
I^+}}$ as the subset $(p_\Omega \circ p_{\pi_0})^{-1} \big(
(W_{d^i})_{i \in I^+} \big )$ of $\Omega_{\hat{D}_R, V_B},$ then we
obtain the decomposition of $\Omega_{\hat{D}_R, V_B}$
\begin{equation}
\label{equation: decomposition of Omega} \bigcup_{(W_{d^i})_{i \in
I^+} \in A_{G_\chi} (S^2, \chi) \text{ with } W_{d^{-1}} =
V_{\bar{f}^{-1}} } \quad \Omega_{\hat{D}_R, (W_{d^i})_{i \in I^+}}.
\end{equation}
From this decomposition, it suffices to focus on $\Omega_{\hat{D}_R,
(W_{d^i})_{i \in I^+}}$ to understand $\Omega_{\hat{D}_R, V_B}.$
Now, we state a classification result.

\begin{proposition}
 \label{proposition: proposition for isomorphism}
\begin{enumerate}
  \item Assume that $\pi_0 ( \Omega_{\hat{D}_R,
(W_{d^i})_{i \in I^+}} )$ is nonempty and $c_1 : \pi_0 (
\Omega_{\hat{D}_R, (W_{d^i})_{i \in I^+}} ) \rightarrow H^2 (S^2),$
$[\Phi_{\hat{D}_R}] \mapsto c_1( F_{V_B} / ~ \Phi_{\hat{D}_R} )$ is
injective for each $(W_{d^i})_{i \in I^+}$ in $A_{G_\chi} ( S^2 ).$
Then, $p_\vect$ is surjective, and
\begin{equation*}
p_{\vect} \times c_1 : \Vect_{G_\chi} (S^2, \chi) \rightarrow
A_{G_\chi} ( S^2, \chi ) \times H^2 (S^2)
\end{equation*}
is injective.
  \item Assume that $\pi_0 ( \Omega_{\hat{D}_R,
(W_{d^i})_{i \in I^+}} )$ consists of exactly one element for each
$(W_{d^i})_{i \in I^+}$ in $A_{G_\chi} ( S^2, \chi ).$ Then,
$p_{\vect}$ is an isomorphism.
\end{enumerate}
\end{proposition}

\begin{proof}
We prove only (1), and (2) is easier. Surjectivity of $p_\vect$ is
trivial by assumption. For arbitrary $[E] \ne [E^\prime]$ in
$\Vect_{G_\chi} (S^2, \chi),$ if $p_\vect([E])$ $\ne
p_\vect([E^\prime]),$ then there is nothing to prove. Assume that
two elements $[E] \ne [E^\prime]$ in the set $\Vect_{G_\chi} (S^2,
\chi)$ satisfy $p_\vect([E])$ $= p_\vect([E^\prime])$ $=(W_{d^i})_{i
\in I^+}$ for some $(W_{d^i})_{i \in I^+}.$ Then, it suffices to
show $c_1(E) \ne c_1(E^\prime)$ to prove injectivity. Put $V_B =
G_\chi \times_{(G_\chi)_{d^{-1}}} W_{d^{-1}}.$ Then, $E|_B \cong
E^\prime|_B \cong V_B$ because their isotropy representations at
$d^{-1}$ are all $W_{d^{-1}}.$ That is, $[E]$ and $[E^\prime]$ are
in $\Vect_{G_\chi} (S^2, \chi)_{V_B}.$ Since $\imath_\Omega$ is
surjective, $E \cong F_{V_B} / \Phi$ and $E^\prime \cong F_{V_B} /
\Phi^\prime$ for some $\Phi$ and $\Phi^\prime$ in $\Omega_{V_B},$
especially in $\Omega_{(W_{d^i})_{i \in I^+}}.$ By Lemma \ref{lemma:
homotopy gives isomorphism}, $[\Phi] \ne [\Phi^\prime]$ in $\pi_0
(\Omega_{(W_{d^i})_{i \in I^+}})$ because $[E] \ne [E^\prime].$ So,
$c_1 \big( F_{V_B} / \Phi \big) \ne c_1 \big( F_{V_B} / \Phi^\prime
\big)$ by assumption. Therefore, we obtain a proof for injectivity.
\qed
\end{proof}

By this lemma, we only have to calculate $\pi_0 ( \Omega_{\hat{D}_R,
(W_{d^i})_{i \in I^+}} )$ to classify equivariant vector bundles in
many cases. In fact, we can apply this lemma except the case when
$R=\rho(G_\chi)$ is equal to $\Z_n \times Z$ with odd $n$ or
$\langle -a_n \rangle$ with even $n/2$ as we shall see in Section
\ref{section: cyclic dihedral cases}. When we can not apply this
lemma, we should apply Lemma \ref{lemma: equivalent condition for
isomorphism} directly.

\section{Equivariant pointwise clutching map}
 \label{section: pointwise clutching map}

Let $\Phi$ be an equivariant clutching map determining $E$ in
$\Vect_{G_\chi} ( S^2, \chi )_{V_B},$ i.e. the map $\Phi$ glues
$F_{V_B}$ along $| \lineL_R |$ to give $E.$ Let us investigate this
gluing process pointwisely. For each $\bar{x} \in | \lineL_R |$ and
$x = |\pi|(\bar{x}),$ let $\bar{\mathbf{x}} = \pi^{-1} ( x ) = \{
\bar{x}_j | j \in \Z_m \}$ for some $m.$ Then, the map $\Phi$ glues
the $(G_\chi)_x$-bundle $\big( \res_{(G_\chi)_x}^{G_\chi} F_{V_B}
\big) \big|_{\bar{\mathbf{x}}}$ along $\bar{\mathbf{x}}$ to give the
$(G_\chi)_x$-representation $E_x,$ and we call this process
\textit{equivariant pointwise gluing}. Here, note that
$(G_\chi)_{\bar{x}_j} < (G_\chi)_x$ for each $j \in \Z_m$ and
\begin{equation}
\label{equation: extension} \res_{(G_\chi)_{\bar{x}_j}}^{(G_\chi)_x}
E_x \cong (F_{V_B})_{\bar{x}_j}
\end{equation}
by equivariance of $\Phi.$ In dealing with equivariant clutching
maps, technical difficulties occur in equivariant pointwise gluings
because gluing by $\Phi$ can be considered as just a continuous
collection of equivariant pointwise gluings at points in $| \lineL_R
|.$ In this section, we prove results on equivariant pointwise
gluing. To deal with equivariant pointwise gluing, we need the
concept of representation extension. For compact Lie groups $N_1 <
N_2,$ let $W_2$ be an $N_2$-representation and $W_1$ be an
$N_1$-representation. Then, $W_2$ is called a \textit{representation
extension} or an $N_2$-\textit{extension} of $W_1$ if
$\res^{N_2}_{N_1} W_2 \cong W_1.$ For example, $E_x$ is an
$(G_\chi)_x$-extension of $(F_{V_B})_{\bar{x}_j}$ for each $j \in
\Z_m$ by (\ref{equation: extension}). And, let $\ext_{N_1}^{N_2}
W_1$ be the set
\begin{equation*}
\{ W_2 \in \Rep(N_2) ~ | ~ \res_{N_1}^{N_2} W_2 \cong W_1 \}.
\end{equation*}

Let us investigate equivariant pointwise gluings more precisely
under a little bit general setting. Let a compact Lie group $N_2$
act on a finite set $\bar{\mathbf{x}} = \{ \bar{x}_j | ~ j \in \Z_m
\}$ for $m \ge 2,$ and let $N_0$ and $N_1$ be the kernel of the
action and the isotropy subgroup $(N_2)_{\bar{x}_0},$ respectively.
Let $F$ be an $N_2$-vector bundle over $\bar{\mathbf{x}}.$
Consider an arbitrary map
\begin{equation*}
\psi : \bar{\mathbf{x}} \rightarrow \amalg_{j \in \Z_m} \Iso(
F_{\bar{x}_j}, F_{\bar{x}_{j+1}} )
\end{equation*}
such that $\psi( \bar{x}_j ) \in \Iso( F_{\bar{x}_j},
F_{\bar{x}_{j+1}} ).$ Call such a map \textit{pointwise preclutching
map} with respect to $F.$ By using $\psi,$ we glue
$F_{\bar{x}_j}$'s, i.e. a vector $u$ in $F_{\bar{x}_j}$ is
identified with $\psi(\bar{x}_j) u$ in $F_{\bar{x}_{j+1}}$ for each
$j.$ Let $F/\psi$ be the quotient of $F$ through this
identification, and let $p_\psi : F \rightarrow F/\psi$ be the
quotient map. Let $\imath_\psi : F_{\bar{x}_0} \rightarrow F/\psi$
be the composition of the natural injection $\imath_{\bar{x}_0} :
F_{\bar{x}_0} \rightarrow F$ and the quotient map $p_\psi.$
\begin{equation}
\label{figure: diagram of equivariant pointwise gluing}
\SelectTips{cm}{} \xymatrix{ F_{\bar{x}_0}
\ar[r]^-{\imath_{\bar{x}_0}}
\ar[d]_-{\imath_\psi} & F \\
F/\psi \ar@{<-}[ur]_-{p_\psi} }
\end{equation}
We would find conditions on $\psi$ under which the quotient $F/\psi$
inherits an $N_2$-representation structure from $F$ and the map
$\imath_\psi$ becomes an $N_1$-isomorphism from $F_{\bar{x}_0}$ to
$\res_{N_1}^{N_2} (F/\psi).$ For notational simplicity, denote
\begin{equation*}
\psi(\bar{x}_{j^\prime}) \cdots \psi(\bar{x}_{j+1}) \psi(\bar{x}_j)
u
\end{equation*}
by $\psi^{j^\prime-j+1} u$ for $u \in F_{\bar{x}_j}$ and $j \le
j^\prime$ in $\Z.$

\begin{lemma}
 \label{lemma: equivalent condition for psi}
For a pointwise preclutching map $\psi$ with respect to $F,$ the
quotient $F/\psi$ carries an $N_2$-representation structure so that
$p_\psi$ is $N_2$-equivariant and $\imath_\psi$ is an
$N_1$-isomorphism if and only if the following conditions hold :
\begin{enumerate}
  \item $\psi^m = \id$ in $\Iso(F_{\bar{x}_j})$ for each $j \in
  \Z_m.$ So, $\psi^k$ is well-defined for all $k \in \Z_m.$
  \item $\psi^{j_3 - j_1} = g \psi( \bar{x}_{j_2} ) g^{-1}$ in $F_{\bar{x}_{j_1}}$ for
  each $j_1 \in \Z_m,$ $g \in N_2$ when $g^{-1} \bar{x}_{j_1} = \bar{x}_{j_2}$
        and $g \bar{x}_{j_2 +1} = \bar{x}_{j_3}$ for some $j_2, j_3 \in \Z_m.$
\end{enumerate}
\end{lemma}

\begin{proof}
To begin with, it is obvious that $\imath_\psi$ is surjective. And,
note that $\imath_\psi$ is injective if and only if Condition (1)
holds.

Next, we show that $F/\psi$ carries an $N_2$-representation
structure such that $p_\psi$ is equivariant if and only if Condition
(2) holds. The possible group action on $F/\psi$ to guarantee
equivariance of $p_\psi$ is as follows:
\begin{equation*}
g \cdot u = p_\psi ( g \bar{u} )
\end{equation*}
for each $g \in N_2, u \in F/\psi, \bar{u} \in p_\psi^{-1} (u).$
This is well-defined if and only if $p_\psi ( g \bar{u} ) = p_\psi (
g \psi^l (\bar{u}) )$ for each $l \in \Z,$ i.e. $\psi^k g \bar{u} =
g \psi^l \bar{u}$ for some $k.$ Putting $\bar{u}^\prime = g
\bar{u},$ this can be written as $\psi^k \bar{u}^\prime = g \psi^l
g^{-1} \bar{u}^\prime.$ Since $g \psi^l g^{-1} = (g \psi g^{-1})^l,$
this holds for each $l$ if and only if it holds for $l=1,$ i.e.
Condition (2). This gives a proof. \qed
\end{proof}

A pointwise preclutching map satisfying (1), (2) of Lemma
\ref{lemma: equivalent condition for psi} is called an
\textit{equivariant pointwise clutching map} with respect to $F.$
Let $\mathcal{A}$ be the set of all equivariant pointwise clutching
maps with respect to $F,$ and we topologize $\mathcal{A}$ with the
subspace topology of $\prod_{j \in \Z_m} \Iso (F_{\bar{x}_j} ,
F_{\bar{x}_{j+1}}).$ An $N_2$-representation $W$ is called
\textit{determined} by $\psi \in \mathcal{A}$ with respect to $F$ if
$W \cong F/\psi.$ In the next corollary, we can see that
representation extension is related to equivariant pointwise
clutching map and especially guarantees nonemptiness of
$\mathcal{A}.$

\begin{corollary}
 \label{corollary: equivalent condition for psi}
For an $N_2$-extension $W$ of $F_{\bar{x}_0},$ assume that there
exists an $N_2$-morphism $p: F \rightarrow W$ such that
$p|_{F_{\bar{x}_j}}$ is a nonequivariant isomorphism for each $j \in
\Z_m.$ Let $\psi$ be the pointwise preclutching map defined by
$\psi(\bar{x}_j) = ( p|_{F_{\bar{x}_{j+1}}} )^{-1} \circ
(p|_{F_{\bar{x}_j}})$ for $j \in \Z_m.$ Then, $\psi$ is in
$\mathcal{A},$ and there is an $N_2$-isomorphism $\imath: W
\rightarrow F/\psi$ such that $p_\psi = \imath \circ p,$ i.e. $\psi$
determines $W$ with respect to $F.$ Especially, $\mathcal{A}$ is
nonempty.
\end{corollary}

\begin{proof}
Easy check. \qed
\end{proof}

To calculate $\pi_0 ( \Omega_{ V_B } )$ later, we need to understand
topology of $\mathcal{A}.$ For this, we would consider $F/\psi$ as
an additional structure over the fixed $F_{\bar{x}_0}$ for each
$\psi \in \mathcal{A}$ as follows:
\begin{lemma}
\label{lemma: star action}
 Define a operation $\star_\psi$ of $N_2$ on
$F_{\bar{x}_0}$ as $g \star_\psi u = \imath_\psi^{-1} p_\psi(gu)$
for $\psi \in \mathcal{A},$ $g \in N_2,$ $u \in F_{\bar{x}_0}.$
Then,
\begin{enumerate}
  \item $\star_\psi$ is an $N_2$-action on $F_{\bar{x}_0}$ such that
  $\imath_\psi : ( F_{\bar{x}_0}, \star_\psi ) \rightarrow F/\psi$ is an
  $N_2$-isomorphism.
  \item $g \star_\psi u = \psi^{m-j} g u = g \psi^k u$ when $g
\bar{x}_0 = \bar{x}_j$ and $g \bar{x}_k = \bar{x}_0.$ Especially, $g
\star_\psi u = g u$ for $g \in N_1.$
\end{enumerate}
Henceforth, we consider $F/\psi$ as $( F_{\bar{x}_0}, \star_\psi ).$
\end{lemma}

\begin{proof}
To prove (1), we show that $\imath_\psi ( g \star_\psi u ) = g \cdot
\imath_\psi (u)$ for each $g \in N_2,$ $u \in F_{\bar{x}_0}$ as
follows:
\begin{align*}
\imath_\psi ( g \star_\psi u ) &= \imath_\psi ( \imath_\psi^{-1}
p_\psi(gu) ) = p_\psi(gu), \\
g \cdot \imath_\psi (u) &= g \cdot p_\psi (u) = p_\psi(gu).
\end{align*}
Since $F/\psi$ already delivers an $N_2$-action and $\imath_\psi$ is
bijective, this shows that $\star_\psi$ is an action. Therefore, we
prove (1).

Next, we prove (2). Note that $\imath_\psi^{-1} p_\psi
|_{F_{\bar{x}_0}}$ is an identity. Using this,
\begin{equation*}
g \star_\psi u = \imath_\psi^{-1} p_\psi(gu) = \imath_\psi^{-1}
p_\psi( \psi^{m-j} gu ) = \psi^{m-j} gu.
\end{equation*}
Also, $\psi^{m-j} gu = g \psi^k u$ by Lemma \ref{lemma: equivalent
condition for psi}.(2). Therefore, we prove (2). \qed
\end{proof}

To obtain more precise results on $\mathcal{A},$ we assume that the
$N_2$-bundle $F$ over $\bar{\mathbf{x}}$ satisfies one of the
following two conditions:
\begin{itemize}
  \item[F1.] $N_2$ fixes
  $\bar{\mathbf{x}} = \{ \bar{x}_j ~ | ~ j \in \Z_m \}$
  with $m=2,$ and $F_{\bar{x}_0} \cong F_{\bar{x}_1},$
  \item[F2.] $N_2$ acts transitively on
  $\bar{\mathbf{x}} = \{ \bar{x}_j ~ | ~ j \in \Z_m \}$
  with $m \ge 2.$
\end{itemize}
These conditions are not too restrictive as the following example
shows:

\begin{example}
\label{example: pointwise clutching}
 Given the bundle $F_{V_B}$ over
$|\lineK_R|$ of Section \ref{section: clutching construction}, let
$\bar{\mathbf{x}} = \pi^{-1} ( x ) = \{ \bar{x}_j ~ | ~ j \in \Z_m
\}$ for each $\bar{x} \in | \lineL_R |,$ $x = |\pi|(\bar{x}),$ some
$m \ge 2.$ Put $N_2 = (G_\chi)_x$ and $F = \big(
\res_{(G_\chi)_x}^{G_\chi} F_{V_B} \big) \big|_{\bar{\mathbf{x}}}.$
By Table \ref{table: all about K_R} and definition of $F_{V_B},$ we
can check that $F$ satisfies Condition F1. or F2. according to
$\bar{x}.$ Given a map $\Phi$ in $\Omega_{V_B},$ we can define a map
$\psi$ in $\mathcal{A}$ as follows:
\begin{enumerate}
\item if $\bar{x}$ is a vertex, then
$\psi_{\bar{x}} (\bar{x}_j) = \Phi (\hat{x}_{j,+})$ for each $j \in
\Z_{j_R},$
\item if $\bar{x}$ is not a vertex and $\bar{x} = p_{|\lineL|}(\hat{x}),$
then
\begin{equation*}
\psi (\bar{x}_0) = \Phi (\hat{x}) \quad \text{and} \quad \psi
(\bar{x}_1) = \Phi \big( |c|(\hat{x}) \big).
\end{equation*}
\end{enumerate}
Then, we have $F/\psi \cong \big( F_{V_B} / \Phi \big)_x$ for each
$x.$ \qed
\end{example}

Let $\mathcal{A}^j$ be the set
\begin{equation*}
\{ \psi(\bar{x}_j) ~|~ \psi \in \mathcal{A} \},
\end{equation*}
and let $\mathcal{A}^{j,j^\prime}$ be the set
\begin{equation*}
\{ ( \psi(\bar{x}_j), \psi(\bar{x}_{j^\prime}) ) ~|~ \psi \in
\mathcal{A} \}.
\end{equation*}
In the below, it will be witnessed that $\mathcal{A}$ is
homeomorphic to $\mathcal{A}^j$ or $\mathcal{A}^{j,j^\prime}$ in
many cases.

\begin{lemma}
 \label{lemma: pointwise clutching for m=2 nontransitive}
Assume that $F$ satisfies Condition F1. Then, $\mathcal{A}$ is
homeomorphic to nonempty $\mathcal{A}^0 = \Iso_{N_2} (
F_{\bar{x}_0}, F_{\bar{x}_1} )$ through the the evaluation map
\begin{equation*}
\mathcal{A} \rightarrow \mathcal{A}^0, ~ \psi \mapsto
\psi(\bar{x}_0).
\end{equation*}
\end{lemma}

\begin{proof}
By Lemma \ref{lemma: equivalent condition for psi}.(1), each $\psi
\in \mathcal{A}$ satisfies $\psi^2 = \id$ so that $\psi$ is
determined by $\psi( \bar{x}_0 ).$ From this, $\mathcal{A}$ is
homeomorphic to $\mathcal{A}^0$ through the evaluation. And, $\psi$
satisfies Lemma \ref{lemma: equivalent condition for psi}.(2) if and
only if $\psi (\bar{x}_j)$ is $\Iso_{N_2} ( F_{\bar{x}_j},
F_{\bar{x}_{j+1}} )$-valued because $N_2$ fixes $\bar{x}_0,
\bar{x}_1.$ Nonemptiness is guaranteed by Condition F1. \qed
\end{proof}

In the remaining of this section, we assume that $F$ satisfies
Condition F2. so that $F \cong N_2 \times_{N_1} F_{\bar{x}_0}.$
Under this assumption, the zeroth homotopy of $\mathcal{A}$ will be
related to $\ext_{N_1}^{N_2} F_{\bar{x}_0}.$ For this, we need a
technical lemma. We would express a map $\psi$ in $\mathcal{A}$ as
$m$ endomorphisms of $F_{\bar{x}_0}.$ This will be useful in dealing
with $\mathcal{A}.$ For each $j \in \Z_m,$ pick an element $g_j \in
N_2$ such that $g_j \bar{x}_j = \bar{x}_0,$ and denote by
$\mathbf{g}$ the $m$-tuple $(g_j)_{j \in \Z_m}.$ Also, express a
pointwise preclutching map $\psi$ by the $m$-tuple $(
\psi(\bar{x}_j) )_{j \in \Z_m}.$ For each $\psi =
(\psi(\bar{x}_j))_{j \in \Z_m},$ define the map
\begin{equation*}
\psi_{\mathbf{g}} : \bar{\mathbf{x}} \rightarrow
\Iso(F_{\bar{x}_0}), ~ \bar{x}_j \mapsto g_{j+1} \psi(\bar{x}_j)
g_j^{-1},
\end{equation*}
and express it by the $m$-tuple $( g_{j+1} \psi(\bar{x}_j) g_j^{-1}
)_{j \in \Z_m}.$ And, denote by $\mathcal{A}_\mathbf{g}$ the set
\begin{equation*}
\{ ~ \psi_{\mathbf{g}} ~|~ \psi \in \mathcal{A} ~ \}.
\end{equation*}
Here, we consider an action of $\Iso_{N_1} ( F_{\bar{x}_0} )$ on
these $\psi_{\mathbf{g}}$'s. For $B \in \Iso_{N_1} ( F_{\bar{x}_0}
),$ denote by $B \psi_{\mathbf{g}} B^{-1}$ the $m$-tuple
\begin{equation*}
(B g_{j+1} \psi(\bar{x}_j) g_j^{-1} B^{-1})_{j \in \Z_m}.
\end{equation*}
Then, this action preserves $\mathcal{A}_{\mathbf{g}}$ as follows:

\begin{lemma} \label{lemma: conjugate psi gives equivalent representation}
\begin{enumerate}
  \item For $\psi \in \mathcal{A}$ and
  $B \in \Iso_{N_1} (F_{\bar{x}_0}),$
  put
  \begin{equation*}
  \psi^\prime (\bar{x}_j) = g_{j+1}^{-1} B g_{j+1} \psi(\bar{x}_j) g_j^{-1} B^{-1}
  g_j
  \end{equation*}
  for each $j \in \Z_m$ so that
  $\psi^\prime$ satisfies
  $B \psi_\mathbf{g} B^{-1} = \psi_\mathbf{g}^\prime.$
  Then, $\psi^\prime \in \mathcal{A}.$
  \item For two elements $\psi, \psi^\prime$ in $\mathcal{A},$ $F/\psi \cong
F/\psi^\prime$ as $N_2$-representation if and only if $B
\psi_{\mathbf{g}} B^{-1} = \psi_{\mathbf{g}}^\prime$ for some $B \in
\Iso_{N_1} (F_{\bar{x}_0}).$
\end{enumerate}
\end{lemma}

\begin{proof}
To prove (1), we would show that $\psi^\prime$ satisfies two
conditions of Lemma \ref{lemma: equivalent condition for psi}. It is
easy to show that $\psi^\prime$ satisfies Condition (1) of Lemma
\ref{lemma: equivalent condition for psi}. Condition (2) of Lemma
\ref{lemma: equivalent condition for psi} is written as
$(\psi^\prime )^{j_3 - j_1} u = g \psi^\prime( \bar{x}_{j_2} )
g^{-1} u$ in $F_{\bar{x}_{j_1}}$ for $j_1 \in \Z_m,$ $g \in N_2$
when $g \bar{x}_{j_2} = \bar{x}_{j_1}$ and $g \bar{x}_{j_2 +1} =
\bar{x}_{j_3}.$ Here,
\begin{align*}
& g \psi^\prime( \bar{x}_{j_2} ) g^{-1} \\
= & g \Big( g_{j_2 +1}^{-1} B g_{j_2 +1} \psi( \bar{x}_{j_2} )
g_{j_2}^{-1} B^{-1} g_{j_2} \Big)
g^{-1} \\
= & \Big( g g_{j_2 +1}^{-1} B g_{j_2 +1} g^{-1} \Big) \Big( g \psi(
\bar{x}_{j_2} ) g^{-1} \Big) \Big( g g_{j_2}^{-1} B^{-1} g_{j_2}
g^{-1} \Big).
\end{align*}
Note that $g g_{j_2 +1}^{-1} = g_{j_3}^{-1} ( g_{j_3} g g_{j_2
+1}^{-1} )$ and $g g_{j_2}^{-1} = g_{j_1}^{-1} ( g_{j_1} g
g_{j_2}^{-1} )$ where $g_{j_3} g g_{j_2 +1}^{-1}$ and $g_{j_1} g
g_{j_2}^{-1}$ are in $N_1.$ With these,
\begin{align*}
g g_{j_2 +1}^{-1} B g_{j_2 +1} g^{-1} & = g_{j_3}^{-1} ( g_{j_3} g
g_{j_2 +1}^{-1} ) B ( g_{j_3} g g_{j_2 +1}^{-1} )^{-1} g_{j_3} \\
& = g_{j_3}^{-1} B g_{j_3},
\end{align*}
and
\begin{align*}
g g_{j_2}^{-1} B^{-1} g_{j_2} g^{-1} &= g_{j_1}^{-1} ( g_{j_1} g
g_{j_2}^{-1} ) B^{-1} ( g_{j_1} g g_{j_2}^{-1} )^{-1} g_{j_1} \\
&= g_{j_1}^{-1} B^{-1} g_{j_1}
\end{align*}
because $B \in \Iso_{N_1} ( F_{\bar{x}_0} ).$ So,
\begin{align*}
g \psi^\prime( \bar{x}_{j_2} ) g^{-1} &= \Big( g_{j_3}^{-1} B
g_{j_3} \Big) \Big( g \psi( \bar{x}_{j_2} )
g^{-1} \Big) \Big( g_{j_1}^{-1} B^{-1} g_{j_1} \Big) \\
&= \Big( g_{j_3}^{-1} B g_{j_3} \Big) \psi^{j_3 - j_1} \Big(
g_{j_1}^{-1} B^{-1} g_{j_1} \Big) \\
&= \Big( g_{j_3}^{-1} B g_{j_3} \Big) \Big( g_{j_3}^{-1} B^{-1}
g_{j_3}
\psi^\prime( \bar{x}_{j_3 -1} ) g_{j_3 -1}^{-1} B g_{j_3 -1} \Big) \cdots \\
& \qquad \Big( g_{j_1 +1}^{-1} B^{-1} g_{j_1 +1} \psi^\prime(
\bar{x}_{j_1} ) g_{j_1}^{-1} B g_{j_1} \Big) \Big( g_{j_1}^{-1}
B^{-1}
g_{j_1} \Big) \\
&= \psi^\prime( \bar{x}_{j_3 -1} ) \cdots \psi^\prime( \bar{x}_{j_1}
)
\end{align*}
in $F_{\bar{x}_{j_1}},$ and this proves (1).

Next, we prove (2). For sufficiency, assume that $F/\psi \cong
F/\psi^\prime.$ Considering $F/\psi$ as $( F_{\bar{x}_0}, \star_\psi
)$ by Lemma \ref{lemma: star action}, there exists an
$N_2$-isomorphism $B : ( F_{\bar{x}_0}, \star_\psi ) \rightarrow (
F_{\bar{x}_0}, \star_{\psi^\prime} ),$ and especially $B \in
\Iso_{N_1} (F_{\bar{x}_0})$ by Lemma \ref{lemma: star action}.(2).
So, $B g_j \star_\psi g_0^{-1} u = g_j \star_{\psi^\prime} B
g_0^{-1} u$ for each $j \in \Z_m,$ $u \in F_{\bar{x}_0},$ and if we
substitute $B^{-1} u$ into $u,$ this is written as
\begin{equation*}
( B g_j \psi^j g_0^{-1} B^{-1} ) u = ( g_j {\psi^\prime}^j g_0^{-1}
) u
\end{equation*}
by Lemma \ref{lemma: star action}.(2). Substituting $j=1$ in this,
we obtain $B g_1 \psi g_0^{-1} B^{-1} = g_1 \psi^\prime g_0^{-1}$ on
$F_{\bar{x}_0}.$ By using this and mathematical induction, we would
show $B \psi_{\mathbf{g}} B^{-1} = \psi_{\mathbf{g}}^\prime.$ Assume
that $B g_j \psi g_{j-1}^{-1} B^{-1} = g_j \psi^\prime g_{j-1}^{-1}$
on $F_{\bar{x}_0}$ from $j=1$ to $j=k-1.$ Then,
\begin{align*}
& ( g_k \psi^\prime g_{k-1}^{-1} ) ( g_{k-1} \psi^\prime
g_{k-2}^{-1} )
\cdots ( g_1 \psi^\prime g_0^{-1} ) \\
&= g_k {\psi^\prime}^k g_0^{-1}     \\
&= B g_k \psi^k g_0^{-1} B^{-1}     \\
&= ( B g_k \psi g_{k-1}^{-1} B^{-1} ) ( B
g_{k-1} \psi g_{k-2}^{-1} B^{-1} ) \cdots ( B g_1 \psi g_0^{-1} B^{-1} ) \\
&= ( B g_k \psi g_{k-1}^{-1} B^{-1} ) ( g_{k-1} \psi^\prime
g_{k-2}^{-1} ) \cdots ( g_1 \psi^\prime g_0^{-1} ).
\end{align*}
Comparing the first line with the last, we have $B g_k \psi
g_{k-1}^{-1} B^{-1} = g_k \psi^\prime g_{k-1}^{-1}.$ By mathematical
induction, we obtain $B g_j \psi g_{j-1}^{-1} B^{-1} = g_j
\psi^\prime g_{j-1}^{-1}$ on $F_{\bar{x}_0}$ for each $j \in \Z_m,$
and this shows $B \psi_{\mathbf{g}} B^{-1} =
\psi_{\mathbf{g}}^\prime.$

Last, we prove necessity of (2). By assumption, we have $B g_{j+1}
\psi g_j^{-1} B^{-1} = g_{j+1} \psi^\prime g_j^{-1}$ for each $j \in
\Z_m$ on $F_{\bar{x}_0}.$ Then,
\begin{align*}
& B g_j \psi^j g_0^{-1} B^{-1} \\
= &( B g_j \psi g_{j-1}^{-1} B^{-1} ) ( B
g_{j-1} \psi g_{j-2}^{-1} B^{-1} ) \cdots ( B g_1 \psi g_0^{-1} B^{-1} ) \\
= &( g_j \psi^\prime g_{j-1}^{-1} ) ( g_{j-1} \psi^\prime
g_{j-2}^{-1} ) \cdots ( g_1 \psi^\prime g_0^{-1} ) \\
= & g_j {\psi^\prime}^j g_0^{-1}
\end{align*}
for each $j \in \Z_m$ on $F_{\bar{x}_0}.$ Acting these on $B g_0 u$
for $u \in F_{\bar{x}_0},$ we obtain
\begin{equation*}
B g_j \psi^j u = g_j {\psi^\prime}^j B u
\end{equation*}
because $B \in \Iso_{N_1} (F_{\bar{x}_0}).$ And, this is written as
$B g_j \star_\psi u = g_j \star_{\psi^\prime} B u$ by Lemma
\ref{lemma: star action}.(2). This means that $B$ is equivariant for
$\langle N_1, g_j | j \in \Z_m \rangle$ so that $B$ is an
$N_2$-isomorphism between $( F_{\bar{x}_0}, \star_\psi )$ and $(
F_{\bar{x}_0}, \star_{\psi^\prime} ).$ Therefore, we obtain proof.
\qed
\end{proof}

This lemma means that each orbit of $\mathcal{A}_\mathbf{g}$ under
the $\Iso_{N_1} ( F_{\bar{x}_0} )$-action corresponds to a
representation. More precisely, we have the following:

\begin{theorem}
 \label{theorem: bijectivity with extensions}
Assume that $F$ satisfies Condition F2. Then, the map
\begin{equation*}
\pi_0 (\mathcal{A}) \longrightarrow \ext_{N_1}^{N_2} F_{\bar{x}_0},
\quad [\psi] \mapsto F/\psi
\end{equation*}
is bijective. So, $\mathcal{A}$ is nonempty if and only if
$F_{\bar{x}_0}$ has an $N_2$-extension.
\end{theorem}

\begin{proof}
To begin with, we show that each $N_2$-extension of $F_{\bar{x}_0}$
is obtained as $F/\psi$ for some $\psi \in \mathcal{A}.$ Let $W$ be
an $N_2$-extension of $F_{\bar{x}_0}$ such that $\res_{N_1}^{N_2} W$
is equal to $F_{\bar{x}_0}.$ Consider a map $p: N_2 \times_{N_1}
F_{\bar{x}_0} \rightarrow W, ~ [ g, u] \mapsto gu.$ Since $F \cong
N_2 \times_{N_1} F_{\bar{x}_0},$ the map $p$ satisfies conditions of
Corollary \ref{corollary: equivalent condition for psi}. Therefore,
there exists $\psi$ such that $F/\psi \cong W$ by Corollary
\ref{corollary: equivalent condition for psi}. That is, surjectivity
is proved.

Next, we show that two $\psi$ and $\psi^\prime$ in $\mathcal{A}$ are
connected by a path in $\mathcal{A}$ if and only if $F/\psi \cong
F/\psi^\prime.$ First, we show sufficiency. Let $\psi_t$ for $t \in
[0,1]$ be a continuous path in $\mathcal{A}$ such that $\psi_0 =
\psi$ and $\psi_1 = \psi^\prime.$ Then, all $( F_{\bar{x}_0},
\star_{\psi_t} )$'s are isomorphic because the path of their
characters are in the discrete space $\Rep(N_2).$ Second, we show
necessity. If $F/\psi \cong F/\psi^\prime,$ then Lemma \ref{lemma:
conjugate psi gives equivalent representation} says that $B
\psi_{\mathbf{g}} B^{-1} = \psi_{\mathbf{g}}^\prime$ so that
$\psi_{\mathbf{g}}$ and $\psi_{\mathbf{g}}^\prime$ are connected in
$\mathcal{A}_{\mathbf{g}}$ by a path because $\Iso_{N_1}
(F_{\bar{x}_0})$ is a product of general linear groups by Schur's
Lemma. Since the map from $\mathcal{A}$ to
$\mathcal{A}_{\mathbf{g}}$ sending $\psi$ to $\psi_{\mathbf{g}}$ is
a homeomorphism, $\psi$ and $\psi^\prime$ are connected by a path in
$\mathcal{A},$ and we obtain necessity. Therefore, we obtain
injectivity.

Nonemptiness is clear from the one-to-one correspondence. \qed
\end{proof}

Denote by $\mathcal{A}_\psi$ the path-component of $\mathcal{A}$
\begin{equation*}
\{ ~ \psi^\prime \in \mathcal{A} ~ | ~ F/\psi \cong F/\psi^\prime ~
\},
\end{equation*}
and by $\mathcal{A}_{\psi, \mathbf{g}}$ the path-component of
$\mathcal{A}_{\mathbf{g}}$
\begin{equation*}
\{ ~ \psi_{\mathbf{g}}^\prime \in \mathcal{A}_{\mathbf{g}} ~ | ~
F/\psi \cong F/\psi^\prime ~ \}.
\end{equation*}
Note that $\mathcal{A}_\psi$ and $\mathcal{A}_{\psi, \mathbf{g}}$
are homeomorphic. To calculate homotopy of equivariant clutching
maps in next sections, we need to calculate $\pi_1 (
\mathcal{A}_\psi )$ for each $\psi \in \mathcal{A}.$ To do it, we
would investigate the shape of $\mathcal{A}_\psi.$

\begin{lemma}
 \label{lemma: topology of mathcal A}
For each $\psi \in \mathcal{A},$ $\mathcal{A}_\psi$ is homeomorphic
to
\begin{equation*}
\Iso_{N_1} ( F/\psi ) \big/ \Iso_{N_2} ( F/\psi ).
\end{equation*}
\end{lemma}

\begin{proof}
In Lemma \ref{lemma: conjugate psi gives equivalent
representation}.(2), we have shown that $\mathcal{A}_{\psi,
\mathbf{g}}$ is equal to the orbit of $\psi_{\mathbf{g}}$ by
$\Iso_{N_1} (F_{\bar{x}_0}).$ To understand this orbit, we need to
know the isotropy subgroup $\Iso_{N_1} ( F_{\bar{x}_0}
)_{\psi_\mathbf{g}},$ i.e.
\begin{equation*}
\{ B \in \Iso_{N_1} ( F_{\bar{x}_0} ) ~ | ~ B \psi_\mathbf{g} B^{-1}
= \psi_\mathbf{g} \}.
\end{equation*}
By definition, $B \psi_\mathbf{g} B^{-1} = \psi_\mathbf{g}$ is
expressed as $B g_{j+1} \psi(\bar{x}_j) g_j^{-1} B^{-1} = g_{j+1}
\psi(\bar{x}_j) g_j^{-1}$ for each $j \in \Z_m.$ Since $g_{j+1} \psi
g_j^{-1} = g_{j+1} g_j^{-1} g_j \psi g_j^{-1} = (g_{j+1} g_j^{-1})
\psi^k$ on $F_{\bar{x}_0}$ for some k by Lemma \ref{lemma:
equivalent condition for psi}.(2), this is written as
\begin{equation*}
B ( g_{j+1} g_j^{-1} \star_\psi  B^{-1} u  ) = g_{j+1} g_j^{-1}
\star_\psi u.
\end{equation*}
Since $g_0 \in N_1,$ we have $N_2 = \langle N_1, g_{j+1} g_j^{-1} |
j \in \Z_m \rangle.$ Therefore, $B \in \Iso_{N_1} ( F_{\bar{x}_0}
)_{\psi_\mathbf{g}}$ if and only if $B \in \Iso_{N_2}
(F_{\bar{x}_0}, \star_\psi).$ Therefore, $A_{\psi, \mathbf{g}}$ is
homeomorphic to the quotient $\Iso_{N_1} ( F/\psi ) / \Iso_{N_2} (
F/\psi )$ because $( F_{\bar{x}_0}, \star_\psi ) \cong F/\psi.$ \qed
\end{proof}

In some special cases, we can understand $\mathcal{A}_\psi$ more
precisely. Denote by $\mathcal{A}_\psi^j$ the set
\begin{equation*}
\{ ~ \psi^\prime (\bar{x}_j) ~|~ \psi^\prime \in \mathcal{A}_\psi ~
\},
\end{equation*}
and by $\mathcal{A}_\psi^{j,j^\prime}$ the set
\begin{equation*}
\{ ~ \big( \psi^\prime (\bar{x}_j), \psi^\prime (\bar{x}_{j^\prime})
\big) ~|~ \psi^\prime \in \mathcal{A}_\psi ~ \}.
\end{equation*}

\begin{proposition}
 \label{proposition: psi for cyclic}
Assume that $N_2 = \langle N_0, a_0 \rangle$ with some $a_0 \in N_2$
such that $a_0 \bar{x}_j = \bar{x}_{j+1}$ for each $j \in \Z_m$ so
that $N_2 / N_0 \cong \Z_m$ and $N_1 = N_0.$ Then,
\begin{enumerate}
  \item A pointwise preclutching map $\psi$ with respect to $F$ is
  in $\mathcal{A}$ if and only if $\psi^m = \id,$ $\psi(
\bar{x}_0 ) \in \Iso_{N_0} (F_{\bar{x}_0}, F_{\bar{x}_1}),$ and
$\psi( \bar{x}_j ) = a_0^j \psi(\bar{x}_0) a_0^{-j}$ for each $j \in
\Z_m.$
  \item $\mathcal{A}, \mathcal{A}_\psi$ are homeomorphic to
  $\mathcal{A}^0, \mathcal{A}_\psi^0$
  for any $\psi \in \mathcal{A},$ respectively.
  \item If $F_{\bar{x}_0}$ is $N_0$-isotypical, then
  $\mathcal{A}_\psi$ is simply connected
  for each $\psi$ in $\mathcal{A}.$
\end{enumerate}
\end{proposition}

\begin{proof}
For (1), we only have to show that $\psi$ satisfies two conditions
of Lemma \ref{lemma: equivalent condition for psi} if and only if
$\psi^m = \id,$ $\psi( \bar{x}_0 ) \in \Iso_{N_0} (F_{\bar{x}_0},
F_{\bar{x}_1}),$ and $\psi( \bar{x}_j ) = a_0^j \psi(\bar{x}_0)
a_0^{-j}$ for $j \in \Z_m.$ First, we prove sufficiency. If we
substitute each $h \in N_0$ into $g$ in Condition (2) of Lemma
\ref{lemma: equivalent condition for psi}, $\psi( \bar{x}_j )$ is
$N_0$-equivariant for each $j \in \Z_m.$ Also, if we substitute
$a_0^j$ into $g$ in the same formula, then we obtain $\psi(
\bar{x}_j ) = a_0^j \psi(\bar{x}_0) a_0^{-j}.$ Next, we prove
necessity. Note that $\psi( \bar{x}_j ) = a_0^j \psi(\bar{x}_0)
a_0^{-j}$ is $N_0$-equivariant for each $j \in \Z_m$ because $N_0$
is normal in $N_2$ and $\psi( \bar{x}_0 )$ is $N_0$-equivariant. An
arbitrary $g$ in $N_2$ is expressed as $a_0^j h$ for some $h \in
N_0, j \in \Z_m.$ Then, since $g^{-1} \bar{x}_{j_1} = \bar{x}_{j_2}$
and $g \bar{x}_{j_2 + 1} = \bar{x}_{j_3}$ with $j_2 = j_1 - j, j_3 =
j_1 + 1$ for each $j_1 \in \Z_m,$ we have
\begin{align*}
g \psi( \bar{x}_{j_2} ) g^{-1} u &= ( a_0^j h ) \psi(
\bar{x}_{j_2} ) ( a_0^j h )^{-1} u \\
&= a_0^j \psi( \bar{x}_{j_2} ) a_0^{-j} u \\
&= a_0^{j+ j_2} \psi( \bar{x}_0 ) a_0^{-j- j_2} u \\
&= a_0^{j_1} \psi( \bar{x}_0 ) a_0^{-j_1} u \\
&= \psi( \bar{x}_{j_1} ) u
\end{align*}
for each $u \in F_{\bar{x}_{j_1}}$ so that $\psi$ satisfies
Condition (2) of Lemma \ref{lemma: equivalent condition for psi}.
Therefore, we obtain a proof of (1).

(2) is easy because a map $\psi$ in $\mathcal{A}$ is determined by
$\psi( \bar{x}_0 ).$

Before we prove (3), we prove that the $N_0$-character of
$F_{\bar{x}_0}$ is fixed by $N_2.$ Assume that $F_{\bar{x}_0}$ is
$\chi$-isotypical for some $\chi \in \Irr(N_0).$ By Theorem
\ref{theorem: bijectivity with extensions}, existence of $\psi$
guarantees existence of an $N_2$-extension, say $V.$ For each $g \in
N_2,$ we have
\begin{equation}
\label{equation: fixed character} ^g F_{\bar{x}_0} \cong
\res_{gN_0g^{-1}}^{N_2} ~ ^g V
 \cong \res_{N_0}^{N_2} ~ ^g V \cong \res_{N_0}^{N_2} V
 \cong F_{\bar{x}_0}
\end{equation}
by using normality of $N_0.$ By Definition \ref{definition:
conjugate representation}, $ ^g F_{\bar{x}_0}$ is $(g \cdot
\chi)$-isotypical so that $g \cdot \chi = \chi$ for each $g.$

Now, we prove (3). Put $F_{\bar{x}_0} \cong lU$ for some $l \in \N$
and $U \in \Irr ( N_0 ).$ We already know that the character of $U$
is fixed by $N_2.$ Let $\bar{U}$ be an $N_2$-extension of $U$ whose
existence is guaranteed by Theorem \ref{theorem: extension by
cyclic}. By Corollary \ref{corollary: expressed by extensions},
$F/\psi$ is isomorphic to $( l_0 \bar{U} \otimes \Omega(0) ) \oplus
\cdots \oplus ( l_{m-1} \bar{U} \otimes \Omega(m-1) )$ for some
$l_k$'s in $\Z$ so that $ l = \sum l_k.$ Schur's Lemma says that
$\Iso_{N_0} ( F/\psi ) \cong \Iso_{N_0} ( lU ) \cong \GL(l, \C)$ and
$\Iso_{N_2} ( F/\psi ) \cong \GL(l_0, \C) \times \cdots \times
\GL(l_{m-1}, \C).$ So, $\Iso_{N_0} ( F/\psi ) / \Iso_{N_2} ( F/\psi
)$ is homeomorphic to $\GL(l, \C) / \GL(l_0, \C) \times \cdots
\times \GL(l_{m-1}, \C).$ By long exact sequence of homotopy of a
fibration,
\begin{align*}
\longrightarrow \pi_1 \Big( \GL (l_0, \C) \times \cdots \times \GL
(l_{m-1}, \C) \Big)
\overset{i}{\longrightarrow} \pi_1 \Big( \GL (l, \C) \Big) \\
\longrightarrow \pi_1 \Big( \GL (l, \C) / \GL (l_0, \C) \times
\cdots \times \GL (l_{m-1}, \C) \Big) \longrightarrow 0.
\end{align*}
Here, $i$ is surjective. So, the quotient space is simply connected,
and we obtain a proof of (3). \qed
\end{proof}

\begin{proposition}
 \label{proposition: psi for dihedral}
Assume that $N_2 = \langle {N_0}, a_0 , b_0 \rangle$ with some $a_0,
b_0 \in N_2$ such that $a_0 \bar{x}_j = \bar{x}_{j+1}$ and $b_0
\bar{x}_j = \bar{x}_{-j+1}$ for each $j \in \Z_m.$ So, $N_2 / {N_0}
\cong \D_m$ and $N_1 = \langle N_0, b_0 a_0 \rangle.$ Then,
$\mathcal{A},$ $\mathcal{A}_\psi$ are homeomorphic to
$\mathcal{A}^0,$ $\mathcal{A}_\psi^0$ for any $\psi \in
\mathcal{A},$ respectively.
\end{proposition}

\begin{proof}
If we restrict the action to $\langle {N_0}, a_0 \rangle,$ then the
proof is the same with Proposition \ref{proposition: psi for
cyclic}.(2). \qed
\end{proof}

\begin{proposition}
 \label{proposition: psi for Z_2xZ_2}
Assume that $N_2 = \langle N_0, \alpha_1, \alpha_2, \alpha_3
\rangle$ with some $\alpha_i$'s in $N_2$ such that $\alpha_1
\bar{x}_j = \bar{x}_{-j+1},$ $\alpha_2 \bar{x}_j = \bar{x}_{j+2},$
$\alpha_3 \bar{x}_j = \bar{x}_{-j+3}$ for each $j \in \Z_m$ with
$m=4.$ So, $N_2 / N_0 \cong \Z_2 \times \Z_2$ and $N_1 = N_0.$ Then,
$\mathcal{A},$ $\mathcal{A}_\psi$ are homeomorphic to
$\mathcal{A}^{0, 3},$ $\mathcal{A}_\psi^{0, 3}$ for any $\psi \in
\mathcal{A},$ respectively.
\end{proposition}

\begin{proof}
It suffices to show that each $\psi$ in $\mathcal{A}$ is determined
by $\psi(\bar{x}_0),$ $\psi(\bar{x}_3).$ Substituting $g = \alpha_i$
in Lemma \ref{lemma: equivalent condition for psi}.(2), we obtain
\begin{equation*}
\psi ( \bar{x}_2 ) = \alpha_3 \psi ( \bar{x}_0 ) \alpha_3^{-1},
\quad \psi ( \bar{x}_1 ) = \alpha_2 \psi ( \bar{x}_3 )
\alpha_2^{-1},
\end{equation*}
so we obtain a proof. \qed
\end{proof}

We often need to restrict our arguments on $\mathcal{A}$ to $\{
\bar{x}_j, \bar{x}_{j^\prime} \}$ as follows: let
$\mathcal{A}_{j,j^\prime}$ with $j \ne j^\prime$ be the set of
equivariant pointwise clutching maps with respect to the $N_2(\{
\bar{x}_j, \bar{x}_{j^\prime} \})$-bundle $\big( \res_{N_2(\{
\bar{x}_j, \bar{x}_{j^\prime} \})}^{N_2} F \big) |_{ \{ \bar{x}_j,
\bar{x}_{j^\prime} \} }$ where $N_2(\{ \bar{x}_j, \bar{x}_{j^\prime}
\})$ is the subgroup preserving $\{ \bar{x}_j, \bar{x}_{j^\prime}
\}.$ Here, each $\psi$ in $\mathcal{A}_{j,j^\prime}$ is defined on
$\{ \bar{x}_j, \bar{x}_{j^\prime} \},$ and satisfies
\begin{equation*}
\psi(\bar{x}_j) \in \Iso(F_{\bar{x}_j}, F_{\bar{x}_{j^\prime}}) ~
\text{ and } ~ \psi(\bar{x}_{j^\prime}) \in
\Iso(F_{\bar{x}_{j^\prime}}, F_{\bar{x}_j}).
\end{equation*}
We need to know if the restricted bundle $\big( \res_{N_2(\{
\bar{x}_j, \bar{x}_{j^\prime} \})}^{N_2} F \big) |_{ \{ \bar{x}_j,
\bar{x}_{j^\prime} \} }$ satisfies Condition F1. or F2. If $F$
satisfies Condition F1., then
\begin{equation*}
\bar{\mathbf{x}} = \{ \bar{x}_j, \bar{x}_{j^\prime} \} ~ \text{ and
} ~ \big( \res_{N_2(\{ \bar{x}_j, \bar{x}_{j^\prime} \})}^{N_2} F
\big) |_{ \{ \bar{x}_j, \bar{x}_{j^\prime} \} } = F
\end{equation*}
so that the restricted bundle trivially satisfies Condition F1. And,
it is trivial that $\mathcal{A} = \mathcal{A}_{j,j^\prime}.$ If $F$
satisfies Condition F2., then $\big( \res_{N_2(\{ \bar{x}_j,
\bar{x}_{j^\prime} \})}^{N_2} F \big) |_{ \{ \bar{x}_j,
\bar{x}_{j^\prime} \} }$ satisfies Condition F2. because $N_2(\{
\bar{x}_j, \bar{x}_{j^\prime} \})$ acts transitively on $\{
\bar{x}_j, \bar{x}_{j^\prime} \}.$ We obtain a useful lemma on
$\mathcal{A}_{j,j^\prime}.$

\begin{lemma}
 \label{lemma: restricted pointwise clutching}
The map $\res_{j,j^\prime}: \mathcal{A} \rightarrow
\mathcal{A}_{j,j^\prime}, ~ \psi \mapsto \res_{j,j^\prime}(\psi)$ is
well-defined where
\begin{equation*}
\res_{j,j^\prime}(\psi)(\bar{x}_j) =
 \psi^{j^\prime-j}(\bar{x}_j), \quad \res_{j,j^\prime}(\psi)(\bar{x}_{j^\prime}) =
\psi^{j-j^\prime}(\bar{x}_{j^\prime}).
\end{equation*}
And, we have the isomorphism
\begin{equation*}
\res_{N_2(\{ \bar{x}_j, \bar{x}_{j^\prime} \})}^{N_2} (F/\psi) \cong
\big( \res_{N_2(\{ \bar{x}_j, \bar{x}_{j^\prime} \})}^{N_2} F \big)
|_{ \{ \bar{x}_j, \bar{x}_{j^\prime} \} } \Big/
\res_{j,j^\prime}(\psi)
\end{equation*}
for each $\psi \in \mathcal{A}.$
\end{lemma}

\begin{proof}
By Lemma \ref{lemma: equivalent condition for psi}, we can check
that $\res_{j,j^\prime}(\psi)$ is in $\mathcal{A}_{j,j^\prime}.$
And, the injection from $\big( \res_{N_2(\{ \bar{x}_j,
\bar{x}_{j^\prime} \})}^{N_2} F \big) |_{ \{ \bar{x}_j,
\bar{x}_{j^\prime} \} }$ to $F$ induces the isomorphism through
$p_{\res_{j,{j^\prime}}(\psi)}$ and $p_\psi.$ \qed
\end{proof}

Since any $\psi$ in $\mathcal{A}$ glues all fibers of $F$ to obtain
a single vector space $F/\psi,$ $\psi$ might be considered to glue
each pair of fibers of $F.$ That is, $\psi$ determines the function
$\bar{\psi}$ defined on $\bar{\mathbf{x}} \times \bar{\mathbf{x}} -
\Delta$ sending a pair $(\bar{x}, \bar{x}^\prime)$ to the element
$\bar{\psi}(\bar{x}, \bar{x}^\prime)$ in $\Iso(F_{\bar{x}},
F_{\bar{x}^\prime})$ such that each $u$ in $F_{\bar{x}}$ is
identified with $\bar{\psi}(\bar{x}, \bar{x}^\prime) u$ in
$F_{\bar{x}^\prime}$ by $\psi,$ i.e. $\bar{\psi}(\bar{x},
\bar{x}^\prime)$ satisfies $p_\psi (u) = p_\psi (
\bar{\psi}(\bar{x}, \bar{x}^\prime) u )$ where $\Delta$ is the
diagonal. Call $\bar{\psi}$ the \textit{saturation} of $\psi.$ Since
the index $j$ is not used in defining $\bar{\psi},$ it is often
convenient to use $\bar{\psi}$ instead of $\psi.$ Denote by
$\bar{\mathcal{A}}$ the set $\{ \bar{\psi} |~ \psi \in \mathcal{A}
\},$ and call it the \textit{saturation} of $\mathcal{A}.$ And,
denote $F / \psi, p_\psi$ also by $F / \bar{\psi}, p_{\bar{\psi}},$
respectively.

\section{Some lemmas on fundamental groups}
\label{section: lemmas on fundamental group}

In this section we prove three lemmas needed to calculate homotopy
of equivariant clutching maps. Two of them are just rewriting of
Schur's Lemma. The other is on relative homotopy.

\begin{lemma} \label{lemma: reduce to smaller matrix}
For $\chi \in \Irr(H),$ let $W$ be a $\chi$-isotypical
$H$-representation. For the natural inclusion $\imath: \Iso_H (W)
\rightarrow \Iso (W),$ the map
\begin{equation*}
\imath_* : \pi_1 ( \Iso_H (W) ) \rightarrow \pi_1 ( \Iso (W) )
\end{equation*}
is injective and equal to the multiplication by $\chi (id)$ up to
sign.
\end{lemma}

\begin{proof}
An example is given instead of a detailed proof. Let $U_0$ be an
irreducible $H$-representation with the character $\chi.$ Put $W = 2
U_0.$ Assume that $\chi (id) = 3.$ By Schur's Lemma, $\Iso_H (W)
\cong \GL(2, \C)$ and $\imath$ is a map from $\GL(2, \C)$ to $\GL
(6, \C)$ such that
\begin{equation*}
\left(
                         \begin{array}{cc}
                           a & b \\
                           c & d \\
                         \end{array}
                       \right)
\mapsto \left(
\begin{array}{ccc|ccc}
    a & 0 & 0 & b & 0 & 0 \\
    0 & a & 0 & 0 & b & 0 \\
    0 & 0 & a & 0 & 0 & b \\
    \hline
    c & 0 & 0 & d & 0 & 0 \\
    0 & c & 0 & 0 & d & 0 \\
    0 & 0 & c & 0 & 0 & d \\
  \end{array}
\right).
\end{equation*}
Then, proof for this case is easily followed. \qed
\end{proof}

\begin{lemma} \label{lemma: reduce to smaller matrix cyclic extension}
For $\chi \in \Irr(H),$ let $N_2$ be a compact Lie group such that
$H \lhd N_2$ and $N_2 / H \cong \Z_m,$ and let $W$ be an
$N_2$-representation such that $\res_H^{N_2} W$ is
$\chi$-isotypical. For the natural inclusion $\imath: \Iso_{N_2} (W)
\rightarrow \Iso_H (W),$ the map $\imath_* : \pi_1 ( \Iso_{N_2} (W)
) \rightarrow \pi_1 ( \Iso_H (W) )$ is surjective.
\end{lemma}

\begin{proof}
As in (\ref{equation: fixed character}), the character $\chi$ is
fixed by $N_2.$ Let $U$ be an irreducible $H$-representation whose
character is equal to $\chi,$ and let $\bar{U}$ be an
$N_2$-extension of $U$ whose existence is guaranteed by Theorem
\ref{theorem: extension by cyclic}. Then, $W$ can be decomposed as
$\oplus_{k \in \Z_m} l_k \bar{U}_k$ for some nonnegative integers
$l_k$'s by Corollary \ref{corollary: expressed by extensions} where
we denote by $\bar{U}_k$ the representation $(\bar{U} \otimes
\Omega(k)).$ Then, Schur's Lemma says that
\begin{align*}
\Iso_{N_2} ( W ) & \cong \Iso_{N_2} ( l_0 \bar{U}_0 ) \times \cdots \times \Iso_{N_2} ( l_{m-1} \bar{U}_{m-1} ), \\
& \cong \GL(l_0 ,\C) \times \cdots \times \GL(l_{m-1} ,\C), \\
\Iso_H (W) & \cong \GL(l, \C)
\end{align*}
where $l = \sum l_k.$ Since $\imath$ is the natural inclusion from
to $\GL(l_0 ,\C) \times \cdots \times \GL(l_{m-1} ,\C)$ to $\GL(l,
\C),$ we obtain a proof. \qed
\end{proof}

Here, recall a notation.
\begin{notation}
Let $X$ be a topological space. For two points $y_0$ and $y_1$ in
$X$ and a path $\gamma : [0,1] \rightarrow X$ such that $\gamma(0) =
y_0$ and $\gamma(1) = y_1,$ denote by $\overline{\gamma}$ the
function defined as
\begin{equation*}
\overline{\gamma} : \pi_1 (X, y_0) \longrightarrow \pi_1 (X, y_1),
\quad [\sigma] \mapsto [\gamma^{-1}.\sigma.\gamma].
\end{equation*}
\end{notation}

\begin{lemma} \label{lemma: relative homotopy}
Let $X$ be a path connected topological space with an abelian $\pi_1
(X).$ Let $A$ and $B$ be path connected subspaces of $X.$ Also, let
$\imath_1$ and $\imath_2$ denote inclusions from $A$ and $B$ to $X,$
respectively. Pick two points $y_0 \in A,$ $y_1 \in B,$ and a path
$\gamma : [0,1] \rightarrow X$ such that $\gamma (0) = y_0$ and
$\gamma (1) = y_1.$ Then, we have a one-to-one correspondence
\begin{align*}
\Pi : \pi_1 (X, y_1) \Big/ \Big\{ \overline{\gamma}
\big({\imath_1}_* \pi_1 (A, y_0) \big) +{\imath_2}_* \pi_1 (B, y_1)
\Big\} & \longrightarrow \Big[ [0,1], 0, 1 ;
X, A, B \Big], \\
[\sigma] & \longmapsto [\gamma.\sigma].
\end{align*}
\end{lemma}

\begin{proof}
The only issue is well-definedness of $\Pi$ and $\Pi^{-1}$ which is
just a tedious check. \qed
\end{proof}

\section{Equivariant clutching maps on one-dimensional
fundamental domain}
 \label{section: equivariant clutching maps}

Assume that $\rho(G_\chi)=R$ for some finite $R$ in Table
\ref{table: introduction} such that $\pr(R) = \tetra, \octa,
\icosa.$ In this section, we find conditions on a preclutching map
$\Phi_{\hat{D}_R}$ in $C^0 (\hat{D}_R, V_B)$ to guarantee that
$\Phi_{\hat{D}_R}$ be the restriction of an equivariant clutching
map as promised in Section \ref{section: clutching construction}. By
using this, we show that $\Omega_{\hat{D}_R, (W_{d^i})_{i \in I}}$
is nonempty for each $(W_{d^i})_{i \in I^+} \in A_{G_\chi} (S^2,
\chi).$ For these, we define notations on equivariant pointwise
clutching maps with respect to the $(G_\chi)_x$-bundle $\big(
\res_{(G_\chi)_x}^{G_\chi} F_{V_B} \big)|_{|\pi|^{-1}(x)}$ for each
$\bar{x} \in | \lineL_R |$ and $x = |\pi|(\bar{x}),$ and prove some
lemmas on them. Then, we apply results of Section \ref{section:
pointwise clutching map} in dealing with $\Omega_{\hat{D}_R,
(W_{d^i})_{i \in I^+}}.$ Concrete calculation of homotopy of
equivariant clutching maps for each case is done in the next
section.

First, we define a set $\mathcal{A}_{\bar{x}}$ of equivariant
pointwise clutching maps for each $\bar{x}$ in $|\lineL_R|.$ For
each $\bar{x}$ in $|\lineL_R|$ and $x = |\pi|(\bar{x}),$ put
$\bar{\mathbf{x}} = |\pi|^{-1} (x) = \{ \bar{x}_j | j \in \Z_m \}$
for $m=j_R$ or 2 according to whether $\bar{x}$ is a vertex or not.
Then, let $\mathcal{A}_{\bar{x}}$ be the set of equivariant
pointwise clutching maps with respect to the $(G_\chi)_x$-bundle
$\big( \res_{(G_\chi)_x}^{G_\chi} F_{V_B} \big)
|_{\bar{\mathbf{x}}},$ i.e. $N_2 = (G_\chi)_x$ and $F = \big(
\res_{(G_\chi)_x}^{G_\chi} F_{V_B} \big) |_{\bar{\mathbf{x}}}$ in
notations of Section \ref{section: pointwise clutching map} (see
Example \ref{example: pointwise clutching}). Here, we need to
explain for codomain of maps in $\mathcal{A}_{\bar{x}}.$ For each
$\bar{x} \in |\lineL_R|$ and $\psi_{\bar{x}} \in
\mathcal{A}_{\bar{x}},$ $\psi_{\bar{x}}(\bar{x}_j)$ is in
\begin{equation*}
\Iso \Big( (F_{V_B})_{\bar{x}_j}, (F_{V_B})_{\bar{x}_{j+1}} \Big)
\end{equation*}
for $j \in \Z_{j_R}$ or $\Z_2.$ If $\bar{x}_j \in |\bar{f}|$ and
$\bar{x}_{j+1} \in |\bar{f}^\prime|$ for some $\bar{f},$
$\bar{f}^\prime,$ then $\psi_{\bar{x}}(\bar{x}_j)$ is henceforth
regarded as in $\Iso \big( V_{\bar{f}}, V_{\bar{f}^\prime} \big).$
This is justified because we have fixed the trivialization $\big(
\res_{(G_\chi)_{b(\bar{f})}}^{G_\chi} F_{V_B} \big)
\big|_{|\bar{f}|} = |\bar{f}| \times V_{\bar{f}}$ for each face
$\bar{f} \in \lineK_R.$ We define one more set
$\mathcal{A}_{\bar{e}}$ of equivariant pointwise clutching maps for
each edge $\hat{e}$ and its images $\bar{e} = p_{\lineL}(\hat{e}),$
$|e|= |\pi|(|\bar{e}|).$ For each $\hat{x}$ in $|\hat{e}|$ and
$\bar{x} = p_{|\lineL|} (\hat{x}),$ put
\begin{equation*}
\bar{x}_0^\prime = \bar{x}, \quad \bar{x}_1^\prime = p_{|\lineL|}
(|c|(\hat{x})), \quad \text{and} \quad \bar{\mathbf{x}}^\prime=\{
\bar{x}_j^\prime | j \in \Z_2 \}.
\end{equation*}
Consider the set $\mathcal{A}_{\hat{x}}$ of equivariant pointwise
clutching maps with respect to the $(G_\chi)_{|e|}$-bundle $\big(
\res_{(G_\chi)_{|e|}}^{G_\chi} F_{V_B} \big)
|_{\bar{\mathbf{x}}^\prime}$ where $(G_\chi)_{|e|}$ is the subgroup
of $G_\chi$ fixing $|e|,$ i.e. $N_2 = (G_\chi)_{|e|}$ and $F = \big(
\res_{(G_\chi)_{|e|}}^{G_\chi} F_{V_B} \big)
|_{\bar{\mathbf{x}}^\prime}$ in notations of Section \ref{section:
pointwise clutching map}. As for $\mathcal{A}_{\bar{x}},$ each map
$\psi_{\hat{x}}$ in $\mathcal{A}_{\hat{x}}$ is considered to satisfy
\begin{equation*}
\psi_{\hat{x}} (\bar{x}_0^\prime) \in \Iso(V_{\bar{f}},
V_{\bar{f}^\prime}) \quad \text{and} \quad \psi_{\hat{x}}
(\bar{x}_1^\prime) \in \Iso(V_{\bar{f}^\prime}, V_{\bar{f}})
\end{equation*}
when $\bar{x}_0^\prime \in |\bar{f}|$ and $\bar{x}_1^\prime \in
|\bar{f}^\prime|.$ Here, observe that $\mathcal{A}_{\hat{x}}$ is in
one-to-one correspondence with $\mathcal{A}_{\hat{y}}$ for any two
$\hat{x}, \hat{y}$ in $|\hat{e}|,$ i.e. an element $\psi_{\hat{x}}$
in $\mathcal{A}_{\hat{x}}$ and an element $\psi_{\hat{y}}$ in
$\mathcal{A}_{\hat{y}}$ are corresponded to each other when
$\psi_{\hat{x}} (\bar{x}_j^\prime) = \psi_{\hat{y}}
(\bar{y}_j^\prime)$ for $j \in \Z_2.$ This is because the
$(G_\chi)_{|e|}$-bundle $\big( \res_{(G_\chi)_{|e|}}^{G_\chi}
F_{V_B} \big) |_{\bar{\mathbf{x}}^\prime}$ is all isomorphic
regardless of $\hat{x} \in |\hat{e}|.$ It is very useful to identify
all $\mathcal{A}_{\hat{x}}$'s for $\hat{x} \in |\hat{e}|$ in this
way, so we denote the identified set by $\mathcal{A}_{\bar{e}}.$
That is, each element $\psi_{\bar{e}}$ in $\mathcal{A}_{\bar{e}}$ is
considered as contained in $\mathcal{A}_{\hat{x}}$ for any $\hat{x}
\in |\hat{e}|$ according to the context.

Next, we would define a $G_\chi$-action on saturations. First, we
define notations on saturations. For each $x= |\pi|(\bar{x})$ and
$\psi_{\bar{x}} \in \mathcal{A}_{\bar{x}},$ denote saturations of
$\mathcal{A}_{\bar{x}}$ and $\psi_{\bar{x}}$ by
$\bar{\mathcal{A}}_x$ and $\bar{\psi}_x,$ respectively. Since index
set is irrelevant in defining $\bar{\mathcal{A}}_x,$ the saturation
depends not on $\bar{x}$ but on $x.$ This is why we use the
subscript $x$ instead of $\bar{x}.$ For any $g \in G_\chi,$ the
function $g \cdot \bar{\psi}_x$ is contained in
$\bar{\mathcal{A}}_{g x}$ where $g \cdot \bar{\psi}_x$ is defined as
\begin{equation*}
(g \cdot \bar{\psi}_x) ( \bar{y}, \bar{y}^\prime ) = g \bar{\psi}_x
(g^{-1} \bar{y}, g^{-1} \bar{y}^\prime) g^{-1}
\end{equation*}
for any $\bar{y} \ne \bar{y}^\prime$ in $\pi^{-1}(g x).$ That is, we
obtain $g \bar{\mathcal{A}}_x = \bar{\mathcal{A}}_{gx}.$ Especially,
it is easily shown that $g \cdot \bar{\psi}_x = \bar{\psi}_x$ for
each $g \in (G_\chi)_x$ by equivariance of $\psi_{\bar{x}}.$ From
this, it is noted that if $g^\prime x = gx$ for some $g^\prime, g
\in G_\chi,$ then $g^\prime \cdot \bar{\psi}_x = g \cdot
\bar{\psi}_x$ because $g^\prime = g (g^{-1} g^\prime)$ with $g^{-1}
g^\prime \in (G_\chi)_x,$ i.e. $g \cdot \bar{\psi}_x$ is dependent
on $gx.$ We have defined a $G_\chi$-action on saturations. Since
$\mathcal{A}_{\bar{x}}$ and $\bar{\mathcal{A}}_x$ are in one-to-one
correspondence, $\mathcal{A}_{\bar{x}}$'s also deliver the
$G_\chi$-action induced from the $G_\chi$-action on
$\bar{\mathcal{A}}_x$'s, i.e. $g \cdot \psi_{\bar{x}} \in
\mathcal{A}_{g \bar{x}}$ for each $\psi_{\bar{x}} \in \mathcal{A}_{
\bar{x}}$ is defined and $\mathcal{A}_{g \bar{x}} = g
\mathcal{A}_{\bar{x}}.$ Here, we prove a useful lemma on this
action. Before it, we explain for the superscript $g,$ and state an
elementary fact.

\begin{definition}
\label{definition: conjugate representation} Let $K$ be a closed
subgroup of a compact Lie group $G.$ For a given element $g \in G$
and $W \in \Rep(K),$ the $g K g^{-1}$-representation $^g W$ is
defined to be the vector space $W$ with the new $g K g^{-1}$-action
\begin{equation*}
g K g^{-1} \times W \rightarrow W, \quad (k, u) \mapsto g^{-1} k g u
\end{equation*}
for $k \in g K g^{-1},$ $u \in W.$
\end{definition}

\begin{lemma}
 \label{lemma: elementary lemma on isotropy representation}
Let $G$ be a compact Lie group acting on a topological space $X.$
And, let $E$ be an equivariant vector bundle over $X.$ Then, $^g E_x
\cong E_{g x}$ for each $g \in G$ and $x \in X.$ Also,
\begin{equation*}
\res_{G_x \cap G_{x^\prime}}^{G_x} E_x \cong \res_{G_x \cap
G_{x^\prime}}^{G_{x^\prime}} E_{x^\prime}
\end{equation*}
for any two points $x, x^\prime$ in the same component of the fixed
set $X^{G_x \cap G_{x^\prime}}.$
\end{lemma}

By Lemma \ref{lemma: elementary lemma on isotropy representation}
and Lemma \ref{lemma: intersection of isotropy}, $p_\vect$ is
well-defined in cases when $\pr(R) = \tetra, \octa, \icosa.$

\begin{lemma} \label{lemma: conjugate pointwise clutching}
For each $\bar{x} \in |\lineL_R|,$ $x= |\pi|(\bar{x}),$ $g \in
G_\chi,$ $\psi_{\bar{x}} \in \mathcal{A}_{\bar{x}},$ we have an
$(G_\chi)_{gx}$-isomorphism
\begin{equation*}
\big( \res_{(G_\chi)_{gx}}^{G_\chi} F_{V_B} \big) |_{g
\bar{\mathbf{x}}}
 \Big/ g \cdot \bar{\psi}_x \cong ~ ^g \Big( ~
\big(\res_{(G_\chi)_x}^{G_\chi} F_{V_B} \big) |_{\bar{\mathbf{x}}}
\Big/ \bar{\psi}_x ~ \Big).
\end{equation*}
\end{lemma}

\begin{proof}
Put $L= \langle g, (G_\chi)_x \rangle,$ and let $k_0 \in \N$ be the
smallest natural number satisfying $g^{k_0} \in (G_\chi)_x$ where
such a number exists because $R = \rho(G_\chi)$ is finite. To prove
this lemma by using Lemma \ref{lemma: elementary lemma on isotropy
representation}, we would construct an $L$-bundle $\mathcal{F}$ over
the orbit $L x = \{ x, gx, \cdots, g^{k_0-1}x \}.$ Put
$\bar{\mathcal{F}} = \big( \res_L^{G_\chi} F_{V_B} \big) |_{L
\bar{\mathbf{x}}}$ over $L \bar{\mathbf{x}},$ and consider the
nonequivariant bundle $\mathcal{F}$ over the orbit $L x$ whose fiber
$\mathcal{F}_{g^k x}$ at $g^k x$ is equal to $\bar{\mathcal{F}}
|_{g^k \bar{\mathbf{x}}} \big/ ( g^k \cdot \bar{\psi}_x )$ for $k =
0, \cdots, k_0-1.$ And, let $P : \bar{\mathcal{F}} \rightarrow
\mathcal{F}$ be the map such that the restriction of $P$ to
$\bar{\mathcal{F}} |_{g^k \bar{\mathbf{x}}}$ is equal to $p_{g^k
\cdot \bar{\psi}_x}$ where $p_{g^k \cdot \bar{\psi}_x} :
\bar{\mathcal{F}} |_{g^k \bar{\mathbf{x}}} \rightarrow
\bar{\mathcal{F}} |_{g^k \bar{\mathbf{x}}} \big/ ( g^k \cdot
\bar{\psi}_x )$ is the quotient map of (\ref{figure: diagram of
equivariant pointwise gluing}). Then, we would define an $L$-action
on $\mathcal{F}$ as $l u = P( l \bar{u} )$ for each $l \in L, u \in
\mathcal{F},$ and any $\bar{u} \in P^{-1}(u)$ so that $P$ becomes
$L$-equivariant. As long as this is well-defined, it is easily shown
that it becomes an action because it is defined by the $L$-action on
$\bar{\mathcal{F}}$ through $P.$ So, we would prove well-definedness
of the action. For this, it suffices to show $P(l \bar{u}) = P(l
\bar{u}^\prime)$ for each $l \in L$ and each $\bar{u},
\bar{u}^\prime$ in $\bar{\mathcal{F}}$ satisfying $P(\bar{u}) =
P(\bar{u}^\prime).$ If $\bar{u} \in (F_{V_B})_{g^k \bar{x}_j}$ and
$\bar{u}^\prime \in (F_{V_B})_{g^k \bar{x}_{j^\prime}}$ for some $j,
j^\prime,$ then $P(\bar{u}) = P(\bar{u}^\prime)$ is written as
\begin{equation}
\tag{*} (g^k \cdot \bar{\psi}_x) ( g^k \bar{x}_j, g^k
\bar{x}_{j^\prime} ) \bar{u} = \bar{u}^\prime.
\end{equation}
Note that $l \bar{u} \in (F_{V_B})_{l g^k \bar{x}_j}$ and $l
\bar{u}^\prime \in (F_{V_B})_{l g^k \bar{x}_{j^\prime}}.$ And, put
$l=g^{k^\prime} l^\prime$ with $l^\prime \in L_{g^k x}$ and some
integer $k^\prime$ so that $l g^k \bar{\mathbf{x}} = g^{k+k^\prime}
\bar{\mathbf{x}}$ because $l^\prime$ fixes $g^k \bar{\mathbf{x}}.$
Remembering that the restriction of $P$ to $\bar{\mathcal{F}} |_{l
g^k \bar{\mathbf{x}}} = \bar{\mathcal{F}} |_{g^{k+k^\prime}
\bar{\mathbf{x}}}$ is equal to $p_{g^{k+k^\prime} \cdot
\bar{\psi}_x},$ $P(l \bar{u}) = P(l \bar{u}^\prime)$ is shown as
\begin{align*}
& (g^{k+k^\prime} \cdot \bar{\psi}_x)
( lg^k \bar{x}_j, lg^k \bar{x}_{j^\prime} ) l \bar{u}   \\
=& g^{k^\prime} (g^k \cdot \bar{\psi}_x) ( g^{-k^\prime} lg^k
\bar{x}_j, g^{-k^\prime} lg^k \bar{x}_{j^\prime} )
g^{-k^\prime} l \bar{u}   \\
=& g^{k^\prime} l^\prime l^{\prime -1} (g^k \cdot \bar{\psi}_x) (
l^\prime g^k \bar{x}_j, l^\prime g^k \bar{x}_{j^\prime} )
l^\prime \bar{u}   \\
=& g^{k^\prime} l^\prime (l^{\prime -1} g^k \cdot \bar{\psi}_x) (
g^k \bar{x}_j, g^k \bar{x}_{j^\prime} )
\bar{u}   \\
=& l (g^k \cdot \bar{\psi}_x) ( g^k \bar{x}_j, g^k
\bar{x}_{j^\prime} )
\bar{u}   \\
=& l \bar{u}^\prime
\end{align*}
where we use (*) in the last line. So, we obtain well-definedness of
$L$-action on $\mathcal{F}.$ By definition, isotropy representations
$\mathcal{F}_x$ and $\mathcal{F}_{gx}$ are equal to representations
\begin{equation*}
\big(\res_{(G_\chi)_x}^{G_\chi} F_{V_B} \big) |_{\bar{\mathbf{x}}}
\big/ \bar{\psi}_x \quad \text{and} \quad \big(
\res_{(G_\chi)_{gx}}^{G_\chi} F_{V_B} \big) |_{g \bar{\mathbf{x}}}
 \Big/ g \cdot \bar{\psi}_x,
\end{equation*}
respectively. Then, the lemma follows from Lemma \ref{lemma:
elementary lemma on isotropy representation}. \qed
\end{proof}

To investigate $\Omega_{\hat{D}_R, (W_{d^i})_{i \in I^+}},$ we need
prove a basic lemma on relations between $\mathcal{A}_{\bar{e}^i}$
and $\mathcal{A}_{\bar{x}}$ for $\bar{x} = \bar{v}^i$ or
$\bar{v}^{i+1}.$ Also, we prove lemmas on evaluation of equivariant
pointwise clutching maps.

\begin{lemma}
 \label{lemma: inclusions of A_x's}
Put $\bar{\mathbf{v}}^k = |\pi|^{-1}(v^k) = \{ \bar{x}_j ~ | ~
\bar{x}_j = \bar{v}_j^k \text{ for } j \in \Z_{j_R} \}$ for $k=i,
i+1.$ And, put $\bar{\mathbf{v}}^{\prime i} = \{ \bar{v}_0^i,
\bar{v}_1^i \}$ and $\bar{\mathbf{v}}^{\prime i+1} = \{
\bar{v}_0^{i+1}, \bar{v}_{-1}^{i+1} \}$ so that $\bar{v}_1^i,$
$\bar{v}_{-1}^{i+1}$ $\in c(\bar{e}^i).$
\begin{enumerate}
  \item $\mathcal{A}_{\bar{x}} \subset \mathcal{A}_{\bar{e}^i}$
  for each interior $\bar{x}$ in $|\bar{e}^i|.$
  \item $\mathcal{A}_{\bar{e}^i} = \mathcal{A}_{\bar{x}}$
  for each interior $\bar{x} \ne b(\bar{e}^i)$ in $|\bar{e}^i|.$
  Moreover, $\mathcal{A}_{\bar{e}^i} = \mathcal{A}_{b(\bar{e}^i)}$
  if $(G_\chi)_{b(e^i)} = (G_\chi)_x$ for $x= |\pi|(\bar{x}).$
  \item $\mathcal{A}_{\bar{v}^i}^0 \subset
  \mathcal{A}_{\bar{e}^i}^0$ and $\mathcal{A}_{\bar{v}^{i+1}}^{j_R-1} \subset
  \mathcal{A}_{\bar{e}^i}^1.$
  \item For each $\psi_{\bar{v}^i}$ in $\mathcal{A}_{\bar{v}^i},$
  we have $\psi_{\bar{v}^i} (\bar{v}_0^i) = \psi_{\bar{e}^i} (\bar{v}_0^i)$
  for the unique $\psi_{\bar{e}^i} \in \mathcal{A}_{\bar{e}^i},$ and
  \begin{equation*}
  \res_{(G_\chi)_{|e^i|}}^{(G_\chi)_{v^i}}
  \Big\{ \big( \res_{(G_\chi)_{v^i}}^{G_\chi}
  F_{V_B} \big) |_{\bar{\mathbf{v}}^i}
  \Big/ \psi_{\bar{v}^i} \Big\}
  \cong \big( \res_{(G_\chi)_{|e^i|}}^{G_\chi}
  F_{V_B} \big) |_{\bar{\mathbf{v}}^{\prime i}} \Big/
  \psi_{\bar{e}^i}.
  \end{equation*}
  \item For each $\psi_{\bar{v}^{i+1}}$ in $\mathcal{A}_{\bar{v}^{i+1}},$
  we have $\psi_{\bar{v}^{i+1}} (\bar{v}_{-1}^{i+1}) =
  \psi_{\bar{e}^i} (\bar{v}_{-1}^{i+1})$
  for the unique $\psi_{\bar{e}^i} \in \mathcal{A}_{\bar{e}^i},$ and
  \begin{equation*}
  \res_{(G_\chi)_{|e^i|}}^{(G_\chi)_{v^{i+1}}}
  \Big\{ \big( \res_{(G_\chi)_{v^{i+1}}}^{G_\chi} F_{V_B} \big) |_{\bar{\mathbf{v}}^{i+1}}
  \Big/ \psi_{\bar{v}^{i+1}} \Big\}
  \cong \big( \res_{(G_\chi)_{|e^i|}}^{G_\chi} F_{V_B} \big)
  |_{\bar{\mathbf{v}}^{\prime i+1}} \Big/
  \psi_{\bar{e}^i}.
  \end{equation*}
\end{enumerate}
\end{lemma}

\begin{proof}
(1) follows from $(G_\chi)_{|e^i|} \subset (G_\chi)_x$ when
$x=|\pi|(\bar{x}).$ (2) follows from $(G_\chi)_{|e^i|} =
(G_\chi)_x.$ Similarly, since $(G_\chi)_{|e^i|} \subset
G_\chi(\bar{\mathbf{v}}^{\prime i})$ and $(G_\chi)_{|e^i|} \subset
G_\chi(\bar{\mathbf{v}}^{\prime i+1}),$ (3) holds by Lemma
\ref{lemma: restricted pointwise clutching} where
$G_\chi(\bar{\mathbf{v}}^{\prime i})$ and
$G_\chi(\bar{\mathbf{v}}^{\prime i+1})$ are subgroups of $G_\chi$
preserving $\bar{\mathbf{v}}^{\prime i}$ and
$\bar{\mathbf{v}}^{\prime i+1},$ respectively. The first statement
of (4) follows by (3) and Lemma \ref{lemma: pointwise clutching for
m=2 nontransitive}. By Lemma \ref{lemma: restricted pointwise
clutching}, we have the second statement of (4). Similarly, (5) is
also obtained. \qed
\end{proof}

\begin{lemma} \label{lemma: injective for A_x 1}
Assume that $\pr(R) = \tetra,$ $\octa,$ $\icosa$ and $R \ne \tetra
\times Z.$ For each vertex $\bar{v}$ in $\lineK_R,$ the evaluation
map $\mathcal{A}_{\bar{v}} \rightarrow \mathcal{A}_{\bar{v}}^0,$
$\psi_{\bar{v}} \mapsto \psi_{\bar{v}}(\bar{v})$ is homeomorphic.
\end{lemma}

\begin{proof}
Put $v = |\pi|(\bar{v}).$ In these cases, $R_v \cong \Z_m$ or $\D_m$
with $m=j_R$ by Table \ref{table: all about K_R}, and proof is done
by Proposition \ref{proposition: psi for cyclic}.(2) and
\ref{proposition: psi for dihedral}. \qed
\end{proof}

\begin{lemma} \label{lemma: injective for A_x 2}
Assume that $R = \tetra \times Z.$ For each vertex $\bar{v}$ in
$\lineK_R,$ the evaluation map
\begin{equation*}
\mathcal{A}_{\bar{v}} \rightarrow \mathcal{A}_{\bar{v}}^0 \times
\mathcal{A}_{\bar{v}}^{-1}, ~ \psi_{\bar{v}} \mapsto \big( ~
\psi_{\bar{v}}(\bar{v}_0), ~ \psi_{\bar{v}}(\bar{v}_{-1}) ~ \big)
\end{equation*}
is homeomorphic.
\end{lemma}

\begin{proof}
Put $v = |\pi|(\bar{v}).$ In these cases, $R_v \cong \Z_2 \times
\Z_2$ by Table \ref{table: all about K_R}, and proof is done by
Proposition \ref{proposition: psi for Z_2xZ_2}. \qed
\end{proof}

Now, we state conditions on a preclutching map $\Phi_{\hat{D}_R}$ in
$C^0 (\hat{D}_R, V_B)$ to guarantee that $\Phi_{\hat{D}_R}$ be the
restriction of an equivariant clutching map. When $R = \tetra \times
Z,$ pick an element $t_0 \in (G_\chi)_{b(f^{-1})}$ such that $t_0
v^0 = v^1.$ So, $t_0$ satisfies $t_0 \bar{v}_j^0 = \bar{v}_j^1$ and
$t_0 \bar{v}_{j,\pm}^0 = \bar{v}_{j,\pm}^1$ for $j \in \Z_4$ as
illustrated in Figure \ref{figure: simple form TxZ}. Also, we define
a terminology.
\begin{definition}
\label{definition: determined} For $\bar{x} \in |\lineL_R|,$ $x =
|\pi|(\bar{x}),$ $\psi_{\bar{x}} \in \mathcal{A}_{\bar{x}},$ $\Phi
\in C^0(|\hatL_R|, V_B),$ the map $\psi_{\bar{x}}$ or its saturation
$\bar{\psi}_x$ is called \textit{determined} by $\Phi$ if $\Phi$
satisfies the following condition:
\begin{equation*}
\bar{\psi}_x \Big( p_{|\lineL|}(\hat{x}), p_{|\lineL|}(|c|(\hat{x}))
\Big) = \Phi(\hat{x})
\end{equation*}
for each $\hat{x} \in (|\pi| \circ p_{|\lineL|})^{-1}(x).$ The
condition is concretely written as
\begin{enumerate}
\item if $\bar{x}$ is a vertex, then
\begin{equation*}
\psi_{\bar{x}} (\bar{x}_j) = \Phi (\hat{x}_{j,+}) \quad \text{and}
\quad \psi_{\bar{x}}^{-1} (\bar{x}_j) = \Phi (\hat{x}_{j,-})
\end{equation*}
for each $j \in \Z_{j_R},$
\item if $\bar{x}$ is not a vertex, then
\begin{equation*}
\psi_{\bar{x}} (\bar{x}_0) = \Phi (\hat{x}_{0, +}) \quad \text{and}
\quad \psi_{\bar{x}} (\bar{x}_1) = \Phi (|c|(\hat{x}_{0, +})).
\end{equation*}
\end{enumerate}
\end{definition}

\begin{theorem} \label{theorem: clutching condition}
Assume that $\pr(R) = \tetra,$ $\octa,$ $\icosa.$ Then, a
preclutching map $\Phi_{\hat{D}_R}$ in $C^0 (\hat{D}_R, V_B)$ is in
$\Omega_{\hat{D}_R, V_B}$ if and only if there exists the unique
$\psi_{\bar{x}} \in \mathcal{A}_{\bar{x}}$ for each $\bar{x} \in
\bar{D}_R$ and $x = |\pi|(\bar{x})$ satisfying
\begin{enumerate}
  \item[E2.] $\bar{\psi}_x \Big( p_{|\lineL|}(\hat{x}),
  p_{|\lineL|}(|c|(\hat{x})) \Big) = \Phi_{\hat{D}_R} (\hat{x})$
  for each $\hat{x} \in \hat{D}_R,$
  \item[E3.] for each $\bar{x}, \bar{x}^\prime \in \bar{D}_R$
  and their images $x = |\pi|(\bar{x}), x^\prime = |\pi|(\bar{x}^\prime),$
  if $x^\prime = gx$ for some $g \in G_\chi,$ then
  $\bar{\psi}_{x^\prime} = g \cdot \bar{\psi}_x.$
\end{enumerate}
The set $( \psi_{\bar{x}} )_{\bar{x} \in \bar{D}_R}$ is called
\textit{determined} by $\Phi_{\hat{D}_R}.$
\end{theorem}

\begin{proof}
For necessity, it suffices to construct a map $\Phi$ in
$\Omega_{V_B}$ such that $\Phi|_{\hat{D}_R} = \Phi_{\hat{D}_R}.$ For
this, we would define $\Phi$ on $\hat{\mathbf{D}}_R,$ and then
extend the domain of definition of $\Phi$ to the whole $|\hatL_R|.$
First, we define $\Phi$ on $\hat{\mathbf{D}}_R$ so that each
$\psi_{\bar{x}}$ is determined by $\Phi,$ i.e. each $\psi_{\bar{x}}$
and $\Phi$ satisfy Definition \ref{definition: determined}. Then,
Condition E2. says that $\Phi|_{\hat{D}_R} = \Phi_{\hat{D}_R}.$
Next, we define $\Phi(\hat{x}) = g^{-1} \Phi(g \hat{x}) g$ for each
$\hat{x} \in |\hatL_R|$ and some $g \in G_\chi$ such that $g
\hat{x}$ is in $\hat{\mathbf{D}}_R.$ We need prove well-definedness
of this. Assume that $\hat{y} = g \hat{x}$ and $\hat{y}^\prime =
g^\prime \hat{x}$ are in $\hat{\mathbf{D}}_R$ for two elements $g,
g^\prime$ in $G_\chi$ so that $\hat{y}^\prime = g^\prime g^{-1}
\hat{y}.$ And, let $y$ and $y^\prime$ be images of $\hat{y}$ and
$\hat{y}^\prime$ through $|\pi| \circ p_{|\lineL|},$ respectively.
Then, $y^\prime = g^\prime g^{-1} y.$ These give us
$\bar{\psi}_{y^\prime} = (g^\prime g^{-1}) \cdot \bar{\psi}_y$ by
Condition E3. From this, we obtain
\begin{align*}
g^{\prime -1} \Phi(\hat{y}^\prime) g^\prime &= g^{\prime -1}
\bar{\psi}_{y^\prime} \Big(
p_{|\lineL|}(\hat{y}^\prime), p_{|\lineL|}(|c|(\hat{y}^\prime)) \Big) g^\prime \\
&= g^{-1} \bar{\psi}_y \Big( p_{|\lineL|}(\hat{y}),
p_{|\lineL|}(|c|(\hat{y})) \Big) g \\
&= g^{-1} \Phi(\hat{y}) g
\end{align*}
where we use equivariance of $|c|.$ So, well-definedness is proved.
It is easily checked that $\Phi$ satisfies Condition N1., N2., E1.
Therefore, $\Phi$ is the wanted equivariant clutching map.

For sufficiency, assume that $\Phi_{\hat{D}_R} = \Phi|_{\hat{D}_R}$
for some $\Phi \in \Omega_{V_B}.$ Then, we should choose the unique
$\psi_{\bar{x}} \in \mathcal{A}_{\bar{x}}$ for each $\bar{x} \in
\bar{D}_R$ satisfying Condition E2. and E3. When we show that
$\psi_{\bar{x}}$'s satisfy Condition E3., the condition $x^\prime =
gx$ holds only if $x^\prime,$ $x$ are equal points or $x^\prime,$
$x$ are different vertices by definition of one-dimensional
fundamental domain. The second situation happens only in the case of
$R = \tetra \times Z$ by Table \ref{table: introduction}, and in
other cases $D_R$ contains only one vertex. So, proof for
sufficiency is different according to $R.$

First, assume that $R \ne \tetra \times Z.$ At each $\bar{x} \in
\bar{D}_R,$ the unique $\psi_{\bar{x}}$ in $\mathcal{A}_{\bar{x}}$
is determined by $\Phi$ because $\Phi$ satisfies Condition N1., N2.,
E1. Moreover, it can be checked that $\psi_{\bar{x}}$ is the unique
element satisfying Condition E2. for each $\bar{x}$ by Lemma
\ref{lemma: injective for A_x 1}. And, as we have seen in the above,
we do not need to consider Condition E3. in these cases.

Second, assume that $R = \tetra \times Z.$ At each $\bar{x} \in
\bar{D}_R,$ the unique $\psi_{\bar{x}}$ in $\mathcal{A}_{\bar{x}}$
is determined by $\Phi$ because $\Phi$ satisfies Condition N1., N2.,
E1. Easily, these $\psi_{\bar{x}}$'s satisfy Condition E2., E3. so
that it remains to show their uniqueness. For $\bar{x} \in
\bar{D}_R-\{ \bar{v}^0, \bar{v}^1 \},$ $\psi_{\bar{x}}$ is the
unique element satisfying Condition E2. by Lemma \ref{lemma:
pointwise clutching for m=2 nontransitive}. Condition E2. says that
\begin{equation}
\tag{*} \psi_{\bar{v}^0} (\bar{v}^0) = \Phi_{\hat{D}_R}
(\hat{v}_+^0), \quad \psi_{\bar{v}^1}^{-1} (\bar{v}^1) =
\Phi_{\hat{D}_R} (\hat{v}_-^1).
\end{equation}
And, Condition E3. says that $\bar{\psi}_{v^1} = t_0 \cdot
\bar{\psi}_{v^0},$ and from this it is obtained that
\begin{equation}
\tag{**} \psi_{\bar{v}^1} (\bar{v}^1) = t_0
\Phi_{\hat{D}_R}(\hat{v}_+^0) t_0^{-1}.
\end{equation}
Formulas (*), (**) say that two values $\psi_{\bar{v}^1}(\bar{v}^1)$
and $\psi_{\bar{v}^1}(\bar{v}_{-1}^1) = (\psi_{\bar{v}^1}^{-1}
(\bar{v}^1))^{-1}$ are determined by $\Phi_{\hat{D}_R}.$ And, this
means that $\psi_{\bar{v}^1}$ is unique by Lemma \ref{lemma:
injective for A_x 2}. By Condition E3., uniqueness of
$\psi_{\bar{v}^0}$ is also obtained. \qed
\end{proof}

\begin{remark}
This theorem holds even though we might omit the word `unique' in
the statement of the theorem because uniqueness is not used in the
proof of necessity. \qed
\end{remark}

By using Theorem \ref{theorem: clutching condition}, we would
describe $\Omega_{\hat{D}_R, V_B}$ through
$\mathcal{A}_{\bar{x}}$'s. Define the following set of equivariant
pointwise clutching maps on $\bar{d}^i$'s.

\begin{definition} \label{definition: barA_G}
Denote by $\bar{A}_{G_\chi} (S^2, V_B)$ the set
\begin{equation*}
\{ (\psi_{\bar{d}^i})_{i \in I} ~ | ~ \psi_{\bar{d}^i} \in
\mathcal{A}_{\bar{d}^i} \text{ and } \bar{\psi}_{d^1} = g \cdot
\bar{\psi}_{d^0} \text{ if } d^1 = g d^0 \text{ for some } g \in
G_\chi \}.
\end{equation*}
An element $(\psi_{\bar{d}^i})_{i \in I}$ in $\bar{A}_{G_\chi} (S^2,
V_B)$ is \textit{determined} by $\Phi_{\hat{D}_R} \in
\Omega_{\hat{D}_R, V_B}$ if $\psi_{\bar{d}^i}$'s and
$\Phi_{\hat{D}_R}$ satisfy Condition E2. and E3. of Theorem
\ref{theorem: clutching condition}. Also, a triple $(W_{d^i})_{i \in
I^+}$ in $A_{G_\chi} (S^2, \chi)$ is \textit{determined} by
$(\psi_{\bar{d}^i})_{i \in I}$ in $\bar{A}_{G_\chi} (S^2, V_B)$ if
$W_{d^{-1}} \cong V_{\bar{f}^{-1}}$ and $W_{d^i}$ is determined by
$\psi_{\bar{d}^i}$ with respect to $\big(
\res_{(G_\chi)_{d^i}}^{G_\chi} F_{V_B} \big)|_{|\pi|^{-1}(d^i)}$ for
each $i \in I.$
\end{definition}

By Theorem \ref{theorem: clutching condition}, we can see that for
each $\Phi_{\hat{D}_R} \in \Omega_{\hat{D}_R, V_B}$ there exists an
element $(\psi_{\bar{d}^i})_{i \in I}$ in $\bar{A}_{G_\chi} (S^2,
V_B)$ which is determined by $\Phi_{\hat{D}_R}.$ In fact, it can be
checked that this element is unique by the proof of Theorem
\ref{theorem: clutching condition}.

\begin{corollary} \label{corollary: Omega}
The set $\Omega_{\hat{D}_R, V_B}$ is equal to the set
\begin{align*}
\Big\{ ~ & \Phi_{\hat{D}_R} \in C^0 (\hat{D}_R, V_B) ~ \Big| ~
\Phi_{\hat{D}_R} (\hat{x}) \in \mathcal{A}_{\bar{e}^0}^0 \text{
for each } \hat{x} \in [\hat{d}_+^0, \hat{d}_-^1], \text{ and } \\
&  \Phi_{\hat{D}_R} (\hat{d}_+^0) = \psi_{\bar{d}^0} (\bar{d}^0), ~
\Phi_{\hat{D}_R} (\hat{d}_-^1) = \psi_{\bar{d}^1}^{-1} (\bar{d}^1)
\text{ for some } (\psi_{\bar{d}^i})_{i \in I} \in \bar{A}_{G_\chi}
(S^2, V_B) ~ \Big\}.
\end{align*}
\end{corollary}

\begin{proof}
To prove this corollary, we would rewrite Theorem \ref{theorem:
clutching condition} by using $\mathcal{A}_{\bar{e}^0}$ and
$\bar{A}_{G_\chi} (S^2, V_B).$ By Theorem \ref{theorem: clutching
condition}, a preclutching map $\Phi_{\hat{D}_R}$ is in
$\Omega_{\hat{D}_R, V_B}$ if and only if a set $\Psi = (
\psi_{\bar{x}} )_{\bar{x} \in \bar{D}_R}$ is determined by
$\Phi_{\hat{D}_R}.$ As we have seen in the proof of the theorem,
$gx=x^\prime$ with $x \ne x^\prime$ in Condition E3. is possible
only when $x$ and $x^\prime$ are $d^i$'s (of course, more precisely
when they are vertices). So, $\Psi$ is determined by
$\Phi_{\hat{D}_R}$ if and only if $(\psi_{\bar{d}^i})_{i \in I}$
satisfies Condition E2. and E3. and $( \psi_{\bar{x}} )_{\bar{x} \in
\bar{D}_R-\{ \bar{d}^i | i \in I \}}$ satisfies Condition E2. Here,
$(\psi_{\bar{d}^i})_{i \in I}$ satisfies Condition E2. if and only
if
\begin{equation}
\tag{*} \Phi_{\hat{D}_R} (\hat{d}_+^0) = \psi_{\bar{d}^0}
(\bar{d}^0) \quad \text{ and } \quad \Phi_{\hat{D}_R} (\hat{d}_-^1)
= \psi_{\bar{d}^1}^{-1} (\bar{d}^1).
\end{equation}
Lemma \ref{lemma: inclusions of A_x's}.(3) says that (*) implies
$\Phi_{\hat{D}_R} (\hat{x}) \in \mathcal{A}_{\bar{e}^i}^0$ for
$\hat{x} = \hat{d}_+^0, \hat{d}_-^1.$ So, (*) could be redundantly
rewritten as
\begin{equation}
\tag{**} \Phi_{\hat{D}_R} (\hat{d}_+^0) = \psi_{\bar{d}^0}
(\bar{d}^0), \quad \Phi_{\hat{D}_R} (\hat{d}_-^1) =
\psi_{\bar{d}^1}^{-1} (\bar{d}^1), \quad \text{and }
\Phi_{\hat{D}_R} (\hat{x}) \in \mathcal{A}_{\bar{e}^0}^0
\end{equation}
for $\hat{x} = \hat{d}_+^0, \hat{d}_-^1.$ And,
$(\psi_{\bar{d}^i})_{i \in I}$ satisfies Condition E3. if and only
if
\begin{equation}
\tag{***} (\psi_{\bar{d}^i})_{i \in I} ~ \in ~ \bar{A}_{G_\chi}
(S^2, V_B).
\end{equation}
Next, we deal with $\psi_{\bar{x}}$'s in $( \psi_{\bar{x}}
)_{\bar{x} \in \bar{D}_R-\{ \bar{d}^i | i \in I \}}.$ They satisfy
Condition E2. if and only if $\psi_{\bar{x}}(\bar{x}) =
\Phi_{\hat{D}_R} (\hat{x}_+)$ for each $\bar{x}$ in $\bar{x} \in
\bar{D}_R - \{ \bar{d}^i | i \in I \}.$ And, this is satisfied if
and only if $\Phi_{\hat{D}_R} (\hat{x}_+) \in
\mathcal{A}_{\bar{x}}^0 = \mathcal{A}_{\bar{e}^0}^0$ and we have
chosen $\psi_{\bar{x}}$'s such that $\psi_{\bar{x}}(\bar{x}) =
\Phi_{\hat{D}_R} (\hat{x}_+)$ for each $\bar{x}.$ In summary, three
conditions of this, (**), (***) are equivalent conditions for $\Psi$
to be determined by $\Phi_{\hat{D}_R}.$ Therefore, we obtain a
proof. \qed
\end{proof}

By using this corollary, we would show nonemptiness of
$\Omega_{\hat{D}_R, (W_{d^i})_{i \in I^+}}.$ For this, we need a
lemma.

\begin{lemma}
 \label{lemma: existence of barA}
For each $(W_{d^i})_{i \in I^+} \in A_{G_\chi} (S^2, \chi),$ if we
put
\begin{equation*}
V_B = G_\chi \times_{(G_\chi)_{d^{-1}}} W_{d^{-1}}, \quad F_{V_B} =
G_\chi \times_{(G_\chi)_{d^{-1}}} ( |\bar{f}^{-1}| \times W_{d^{-1}}
),
\end{equation*}
then each $\mathcal{A}_{\bar{x}}$ for $\bar{x} \in [\bar{d}^0,
\bar{d}^1]$ is nonempty. And, we can pick an element
$(\psi_{\bar{d}^i})_{i \in I}$ in $\bar{A}_{G_\chi} (S^2, V_B)$
which determines $(W_{d^i})_{i \in I^+}.$
\end{lemma}

\begin{proof}
For each $i \in I,$ put $\bar{\mathbf{x}} = |\pi|^{-1}(d^i) = \{
\bar{x}_j | \bar{x}_j = \bar{d}_j^i \text{ for } j \in \Z_m \}$ for
$m = j_R$ or 2. If we put $F_i = \big(
\res_{(G_\chi)_{d^i}}^{G_\chi} F_{V_B} \big) |_{\bar{\mathbf{x}}}$
and $N_2 = (G_\chi)_{d^i},$ then $N_1 = (G_\chi)_{\bar{d}^i}.$
Example \ref{example: pointwise clutching} says that $F_i$ satisfies
Condition F1. or F2. Definitions of $F_i$ and $F_{V_B}$ say that
\begin{equation*}
(F_i)_{\bar{x}_0} \cong \res_{(G_\chi)_{d^{-1}} \cap
(G_\chi)_{d^i}}^{(G_\chi)_{d^{-1}}} W_{d^{-1}}
\end{equation*}
because $(G_\chi)_{\bar{x}_0} = (G_\chi)_{[\bar{d}^{-1}, \bar{d}^i]}
= (G_\chi)_{[d^{-1}, d^i]}$ is equal to $(G_\chi)_{d^{-1}} \cap
(G_\chi)_{d^i}$ by Lemma \ref{lemma: intersection of isotropy}. This
implies
\begin{equation*}
(F_i)_{\bar{x}_0} \cong \res_{(G_\chi)_{d^{-1}} \cap
(G_\chi)_{d^i}}^{(G_\chi)_{d^i}} W_{d^i}
\end{equation*}
by Definition \ref{definition: A_R}.(4), i.e. $W_{d^i}$ is a
$(G_\chi)_{d^i}$-extension of $(F_i)_{\bar{x}_0}.$ So, Theorem
\ref{theorem: bijectivity with extensions} says that
$\mathcal{A}_{\bar{d}^i}$ is nonempty. Moreover, we obtain
nonemptiness of $\mathcal{A}_{\bar{x}}$ for $\bar{x} \in [\bar{d}^0,
\bar{d}^1]$ by Lemma \ref{lemma: inclusions of A_x's}. To prove the
second statement, pick an element $\psi_{\bar{d}^i}$ in
$\mathcal{A}_{\bar{d}^i}$ which determines $W_{d^i}.$ If $d^1 = g
\cdot d^0$ for some $g \in G_\chi,$ then the element $g \cdot
\bar{\psi}_{d^0}$ in $\bar{\mathcal{A}}_{\bar{d}^1}$ satisfies
\begin{equation*}
F_1 \big/ g \cdot \bar{\psi}_{d^0} \cong ~ ^g W_{d^0} \cong W_{d^1}
\end{equation*}
by Lemma \ref{lemma: conjugate pointwise clutching} and Definition
\ref{definition: A_R}.(2). So, we may assume that $\bar{\psi}_{d^1}
= g \cdot \bar{\psi}_{d^0}.$ And, $W_{d^{-1}} \cong
V_{\bar{f}^{-1}}$ by definition of $V_B.$ Then,
$(\psi_{\bar{d}^i})_{i \in I}$ is in $\bar{A}_{G_\chi} (S^2, V_B)$
which determines $(W_{d^i})_{i \in I^+}.$ Therefore, we obtain a
proof. \qed
\end{proof}

\begin{proposition} \label{proposition: nonempty omega}
For each $(W_{d^i})_{i \in I^+} \in A_{G_\chi} (S^2, \chi),$ the set
$\Omega_{\hat{D}_R, (W_{d^i})_{i \in I^+}}$ is nonempty.
\end{proposition}

\begin{proof}
Put $V_B = G_\chi \times_{(G_\chi)_{d^{-1}}} W_{d^{-1}}$ so that
$\Omega_{\hat{D}_R, (W_{d^i})_{i \in I^+}}$ is contained in
$\Omega_{\hat{D}_R, V_B}.$ First, we would describe
$\Omega_{\hat{D}_R, (W_{d^i})_{i \in I^+}}$ by using Corollary
\ref{corollary: Omega}. By Lemma \ref{lemma: existence of barA}, we
can pick an element $(\psi_{\bar{d}^i})_{i \in I}$ in
$\bar{A}_{G_\chi} (S^2, V_B)$ which determines $(W_{d^i})_{i \in
I^+}.$ By Theorem \ref{theorem: clutching condition} and Definition
\ref{definition: barA_G}, each element $\Phi_{\hat{D}_R}$ in
$\Omega_{\hat{D}_R, (W_{d^i})_{i \in I^+}}$ satisfies
\begin{equation*}
\Phi_{\hat{D}_R} (\hat{d}_+^0) = \psi_{\bar{d}^0}^\prime
(\bar{d}^0), ~ \Phi_{\hat{D}_R} (\hat{d}_-^1) =
\psi_{\bar{d}^1}^{\prime -1} (\bar{d}^1)
\end{equation*}
for some $(\psi_{\bar{d}^i}^\prime)_{i \in I} \in \bar{A}_{G_\chi}
(S^2, V_B),$ i.e. $(\psi_{\bar{d}^i}^\prime)_{i \in I}$ is
determined by $\Phi_{\hat{D}_R}.$ Easily, $\psi_{\bar{d}^i}^\prime
\in (\mathcal{A}_{\bar{d}^i})_{\psi_{\bar{d}^i}}$ for $i \in I$
because $(\psi_{\bar{d}^i}^\prime)_{i \in I}$ should determine
$(W_{d^i})_{i \in I^+}.$ Moreover, Lemma \ref{lemma: inclusions of
A_x's} says that both $\psi_{\bar{d}^0}^\prime (\bar{d}^0)$ and
$\psi_{\bar{d}^1}^{\prime -1} (\bar{d}^1)$ are in
$(\mathcal{A}_{\bar{e}^0})^0$ which determine
$\res_{(G_\chi)_{|e^0|}}^{(G_\chi)_{d^0}} W_{d^0}$ and
$\res_{(G_\chi)_{|e^0|}}^{(G_\chi)_{d^1}} W_{d^1},$ respectively.
Since $(G_\chi)_{|e^0|} = (G_\chi)_{d^0} \cap (G_\chi)_{d^1}$ by
Lemma \ref{lemma: intersection of isotropy} and these two
representations are isomorphic by definition of $A_{G_\chi} (S^2,
\chi),$ Theorem \ref{theorem: bijectivity with extensions} says that
$\psi_{\bar{d}^0}^\prime (\bar{d}^0)$ and $\psi_{\bar{d}^1}^{\prime
-1} (\bar{d}^1)$ ( of course, also $\psi_{\bar{d}^0} (\bar{d}^0)$
and $\psi_{\bar{d}^1}^{-1} (\bar{d}^1)$ ) are in the same component
$(\mathcal{A}_{\bar{e}^0})_{\psi_{\bar{e}^0}}^0$ of
$(\mathcal{A}_{\bar{e}^0})^0$ for some $\psi_{\bar{e}^0} \in
\mathcal{A}_{\bar{e}^0}.$ Since $\psi_{\bar{d}^0} (\bar{d}^0)$ and
$\psi_{\bar{d}^1} (\bar{d}^1)$ do exist, such a $\psi_{\bar{e}^0}$
exists and $(\mathcal{A}_{\bar{e}^0})_{\psi_{\bar{e}^0}}$ is
nonempty. So, $\Omega_{\hat{D}_R, (W_{d^i})_{i \in I^+}}$ is
expressed as
\begin{align}
\label{equation: Omega} \Big\{ ~ & \Phi_{\hat{D}_R} \in C^0
(\hat{D}_R, V_B) ~ \Big| ~ \Phi_{\hat{D}_R} (\hat{x}) \in
(\mathcal{A}_{\bar{e}^0})_{\psi_{\bar{e}^0}}^0 \text{
for each } \hat{x} \in [\hat{d}_+^0, \hat{d}_-^1], \text{ and } \\
\notag &  \Phi_{\hat{D}_R} (\hat{d}_+^0) = \psi_{\bar{d}^0}^\prime
(\bar{d}^0), ~ \Phi_{\hat{D}_R} (\hat{d}_-^1) =
\psi_{\bar{d}^1}^{\prime -1} (\bar{d}^1) \text{ for some }
(\psi_{\bar{d}^i}^\prime)_{i \in I} \in \bar{A}_{G_\chi}
(S^2, V_B) \\
\notag & \text{ such that } \psi_{\bar{d}^i}^\prime \in
(\mathcal{A}_{\bar{d}^i})_{\psi_{\bar{d}^i}} \text{ for } i \in I ~
\Big\}.
\end{align}
However, we can construct a continuous function
$\Phi_{\hat{D}_R}^\prime : [\hat{d}_+^0, \hat{d}_-^1] \rightarrow
(\mathcal{A}_{\bar{e}^0})_{\psi_{\bar{e}^0}}^0 \subset \Iso
(V_{\bar{f}^{-1}}, V_{\bar{f}^0})$ such that
$\Phi_{\hat{D}_R}^\prime (\hat{d}_+^0) = \psi_{\bar{d}^0}
(\bar{d}^0)$ and $\Phi_{\hat{D}_R}^\prime (\hat{d}_-^1) =
\psi_{\bar{d}^1}^{-1} (\bar{d}^1)$ because $\psi_{\bar{d}^0}
(\bar{d}^0)$ and $\psi_{\bar{d}^1}^{-1} (\bar{d}^1)$ are in the
nonempty path connected
$(\mathcal{A}_{\bar{e}^0})_{\psi_{\bar{e}^0}}^0.$ Since
$\Phi_{\hat{D}_R}^\prime$ is contained in the set (\ref{equation:
Omega}), we obtain a proof. \qed
\end{proof}

\section{Proof for cases when $\pr(\rho(G_\chi)) = \tetra, \octa, \icosa$}
 \label{section: platonic case}

Now, we are ready to calculate homotopy of equivariant clutching
maps.

\begin{proposition} \label{proposition: platonic nonorientable case}
Assume that $\pr(R) = \tetra,$ $\octa,$ $\icosa$ and $R \ne \tetra,$
$\octa,$ $\icosa.$ Then, Theorem \ref{main: only by isotropy} holds
for these $R$'s.
\end{proposition}

\begin{proof}
We prove the proposition only for the case of $\rho(G_\chi)= \tetra
\times Z.$ Proof for other cases are similar. We would show that
$\pi_0 ( \Omega_{\hat{D}_R, (W_{d^i})_{i \in I^+}} )$ is one point
set for each $(W_{d^i})_{i \in I^+}$ $\in A_{G_\chi} (S^2, \chi)$ by
Proposition \ref{proposition: proposition for isomorphism}.(2). Put
\begin{equation*}
V_B = G_\chi \times_{(G_\chi)_{d^{-1}}} W_{d^{-1}}, \quad F_{V_B} =
G_\chi \times_{(G_\chi)_{d^{-1}}} ( |\bar{f}^{-1}| \times W_{d^{-1}}
)
\end{equation*}
for each $(W_{d^i})_{i \in I^+}$ $\in A_{G_\chi} (S^2, \chi).$

By Proposition \ref{proposition: nonempty omega}, we know that
$\pi_0 ( \Omega_{\hat{D}_R, (W_{d^i})_{i \in I^+}} )$ is nontrivial.
For two arbitrary $\Phi_{\hat{D}_R}$ and $\Phi_{\hat{D}_R}^\prime$
in $\Omega_{\hat{D}_R, (W_{d^i})_{i \in I^+}},$ let
$(\psi_{\bar{d}^i})_{i \in I}$ and $(\psi_{\bar{d}^i}^\prime)_{i \in
I}$ in $\bar{A}_{G_\chi} (S^2, V_B)$ be two elements determined by
$\Phi_{\hat{D}_R}$ and $\Phi_{\hat{D}_R}^\prime,$ respectively. We
would construct a homotopy connecting $\Phi_{\hat{D}_R}$ and
$\Phi_{\hat{D}_R}^\prime$ in $\Omega_{\hat{D}_R, (W_{d^i})_{i \in
I^+}}.$ First, we show that we may assume that
$(\psi_{\bar{d}^i}^\prime)_{i \in I} = (\psi_{\bar{d}^i})_{i \in
I}.$ Since $\psi_{\bar{d}^i}$ and $\psi_{\bar{d}^i}^\prime$ for $i
\in I$ determine the same representation $W_{d^i},$ these two are in
the same path component of $\mathcal{A}_{\bar{d}^i}$ by Theorem
\ref{theorem: bijectivity with extensions}. Take paths $\gamma^i :
[0,1] \rightarrow \mathcal{A}_{\bar{d}^i}$ for $i \in I$ such that
$\gamma^i(0) = \psi_{\bar{d}^i}^\prime$ and $\gamma^i(1) =
\psi_{\bar{d}^i}.$ In the case of $R=\tetra \times Z,$ we have $d^1
= t_0 \cdot d^0,$ and $t_0 \cdot \gamma^0$ satisfies $(t_0 \cdot
\gamma^0)(0) = \psi_{\bar{d}^1}^\prime$ and $(t_0 \cdot \gamma^0)(1)
= \psi_{\bar{d}^1}$ by Definition \ref{definition: barA_G}. So, we
may also assume that $\gamma^1 = t_0 \cdot \gamma^0.$ Recall that
the parametrization on $|\hat{e}^0| = [\hat{v}_+^0, \hat{v}_-^1]$ by
$s \in [0,1]$ satisfies $\hat{v}_+^0 = 0,$ $b(\hat{e}^0) = 1/2,$
$\hat{v}_-^1 = 1.$ We construct a homotopy $L (s,t) : [\hat{d}_+^i,
\hat{d}_-^{i+1}] \times [0,1] \rightarrow \mathcal{A}_{\bar{e}^i}^0$
as
\begin{equation*}
  \begin{array}{ll}
L (s,t) = \gamma^i \Big( (1-3s) t \Big) (\bar{d}^i) &
\text{for } s \in [0, \frac 1 3], \\
L (s,t) = \varphi^{\prime i} (3s-1) & \text{for } s \in [ \frac 1
3, \frac 2 3 ], \\
L (s,t) = \gamma^{i+1} \Big( (3s-2)t \Big)^{-1} (\bar{d}^{i+1})
& \text{for } s \in [\frac 2 3, 1]. \\
  \end{array}
\end{equation*}
Then, $L_t$ for each $t \in [0,1]$ is in $\Omega_{\hat{D}_R,
(W_{d^i})_{i \in I^+}}$ by Corollary \ref{corollary: Omega}, and $L$
connects $\Phi_{\hat{D}_R}^\prime$ with $L_1$ in $\Omega_{\hat{D}_R,
(W_{d^i})_{i \in I^+}}$ which determines $(\psi_{\bar{d}^i})_{i \in
I}.$ So, if we put $\Phi_{\hat{D}_R}^\prime = L_1,$ then we may
assume that $\Phi_{\hat{D}_R}$ and $\Phi_{\hat{D}_R}^\prime$
determine the same element $(\psi_{\bar{d}^i})_{i \in I}$ in
$\bar{A}_{G_\chi} (S^2, V_B).$

Now, we construct a homotopy between $\Phi_{\hat{D}_R}$ and
$\Phi_{\hat{D}_R}^\prime.$ As in the proof of Proposition
\ref{proposition: nonempty omega}, $\psi_{\bar{d}^0} (\bar{d}^0)$
and $\psi_{\bar{d}^1}^{-1} (\bar{d}^1)$ are in the same component
$(\mathcal{A}_{\bar{e}^0})_{\psi_{\bar{e}^0}}^0$ for some element
$\psi_{\bar{e}^0}$ in $\mathcal{A}_{\bar{e}^0},$ and
$\Phi_{\hat{D}_R},$ $\Phi_{\hat{D}_R}^\prime$ have values in
$(\mathcal{A}_{\bar{e}^0})_{\psi_{\bar{e}^0}}^0.$ Here, note that
$(\mathcal{A}_{\bar{e}^0})_{\psi_{\bar{e}^0}}$ is simply connected
by Lemma \ref{lemma: isotropy for platonic} and Proposition
\ref{proposition: psi for cyclic}.(3). By simply connectedness, we
can obtain a homotopy $L^\prime (s,t) : [\hat{d}_+^0, \hat{d}_-^1]
\times [0,1] \rightarrow \mathcal{A}_{\bar{e}^0}^0$ as
\begin{equation*}
\begin{array}{lll}
L^\prime (s,t) \in (\mathcal{A}_{\bar{e}^0})_{\psi_{\bar{e}^0}}^0,
\quad & L^\prime (0,t) = \psi_{\bar{d}^0} (\bar{d}^0), \quad &
L^\prime (1,t) =
\psi_{\bar{d}^1}^{-1} (\bar{d}^1), \\
L^\prime (s,0) = \Phi_{\hat{D}_R} (s), \quad & L^\prime (s,1) =
\Phi_{\hat{D}_R}^\prime (s) \quad &
\end{array}
\end{equation*}
for $s,t \in [0,1].$ Then, $L^\prime$ connects $\Phi_{\hat{D}_R}$
and $\Phi_{\hat{D}_R}^\prime$ in $\Omega_{\hat{D}_R, (W_{d^i})_{i
\in I^+}}$ by Corollary \ref{corollary: Omega}. Therefore, we obtain
a proof. Here, we remark that simply connectedness is critical in
obtaining $L^\prime.$ \qed
\end{proof}

\begin{proposition} \label{proposition: platonic orientable case}
Assume that $R$ is equal to one of $\tetra,$ $\octa,$ $\icosa.$
Then, Theorem \ref{main: by isotropy and chern} holds for these
cases.
\end{proposition}

\begin{proof}
When $\rho(G_\chi)$ is given, put
\begin{equation*}
V_B = G_\chi \times_{(G_\chi)_{d^{-1}}} W_{d^{-1}}, \quad F_{V_B} =
G_\chi \times_{(G_\chi)_{d^{-1}}} ( |\bar{f}^{-1}| \times W_{d^{-1}}
)
\end{equation*}
for each $(W_{d^i})_{i \in I^+}$ $\in A_{G_\chi} (S^2, \chi).$ We
can pick an element $(\psi_{\bar{d}^i})_{i \in I}$ in
$\bar{A}_{G_\chi} (S^2, V_B)$ which determines $(W_{d^i})_{i \in
I^+}$ by Lemma \ref{lemma: existence of barA}. In these cases, we
have
\begin{equation*}
\begin{array}{ll}
\bar{D}_R = [\bar{v}^0, b(\bar{e}^0)], & (G_\chi)_{v^0}/H \cong
\Z_{j_R},
\\
(G_\chi)_{b(e^0)}/H \cong \Z_2, & (G_\chi)_x/H = \langle \id \rangle
\end{array}
\end{equation*}
by Table \ref{table: all about K_R}. By using these, path components
of $\mathcal{A}_{\bar{v}^0}, ~ \mathcal{A}_{b(\bar{e}^0)}$ are
simply connected by Proposition \ref{proposition: psi for
cyclic}.(3), and we have $\mathcal{A}_{\bar{e}^0}^0 = \Iso_H
(V_{\bar{f}^{-1}}, V_{\bar{f}^0})$ by Lemma \ref{lemma: pointwise
clutching for m=2 nontransitive}. Here, $\mathcal{A}_{\bar{v}^0},$
$\mathcal{A}_{b(\bar{e}^0)},$ $\mathcal{A}_{\bar{e}^0}$ are all
nonempty by Lemma \ref{lemma: existence of barA}. Since $d^1 \ne g
d^0$ for any $g \in G_\chi,$ $\psi_{\bar{d}^0}^\prime$ and
$\psi_{\bar{d}^1}^\prime$ in the set (\ref{equation: Omega}) have no
relation. So, the set (\ref{equation: Omega}) is rewritten as
\begin{align*}
\Big\{ ~ \Phi_{\hat{D}_R} \in C^0 (\hat{D}_R, V_B) ~ \Big| ~ &
\Phi_{\hat{D}_R} (\hat{x}) \in \Iso_H (V_{\bar{f}^{-1}},
V_{\bar{f}^0}) \text{ for each }
\hat{x} \in [\hat{v}_+^0, b(\hat{e}^0)],  \\
& \text{ and } \Phi_{\hat{D}_R} (\hat{v}_+^0) \in
(\mathcal{A}_{\bar{v}^0})_{\psi_{\bar{v}^0}}^0, ~ \Phi_{\hat{D}_R}
\big( b(\hat{e}^0) \big) \in
(\mathcal{A}_{b(\bar{e}^0)})_{\psi_{b(\bar{e}^0)}}^0 \Big\}.
\end{align*}
Since path components
$(\mathcal{A}_{\bar{v}^0})_{\psi_{\bar{v}^0}}^0$ and
$(\mathcal{A}_{b(\bar{e}^0)})_{\psi_{b(\bar{e}^0)}}^0$ are simply
connected, the homotopy $\pi_0 (\Omega_{\hat{D}_R, (W_{d^i})_{i \in
I^+}} )$ is in one-to-one correspondence with the homotopy $\pi_1
\Big( \Iso_H (V_{\bar{f}^{-1}}, V_{\bar{f}^0}) \Big)$ by Lemma
\ref{lemma: relative homotopy}, and the injection $\imath_0$ from
$\Omega_0$ to $\Omega_{\hat{D}_R, (W_{d^i})_{i \in I^+}}$ induces
the bijection
\begin{equation*}
\pi_0 (\imath_0) : \pi_0 (\Omega_0) \rightarrow \pi_0
(\Omega_{\hat{D}_R, (W_{d^i})_{i \in I^+}})
\end{equation*}
where $\Omega_0$ is defined as
\begin{align*}
\Big\{ ~ \Phi_{\hat{D}_R} \in C^0 (\hat{D}_R, V_B) ~ \Big| ~ &
\Phi_{\hat{D}_R} (\hat{x}) \in \Iso_H (V_{\bar{f}^{-1}},
V_{\bar{f}^0}) \text{ for each }
\hat{x} \in [\hat{v}_+^0, b(\hat{e}^0)],  \\
& \text{ and } \Phi_{\hat{D}_R} (\hat{v}_+^0) = \psi_{\bar{v}^0}
(\bar{v}^0), ~ \Phi_{\hat{D}_R} \big( b(\hat{e}^0) \big) =
\psi_{b(\bar{e}^0)} (b(\bar{e}^0)) \Big\}.
\end{align*}

Now, we would show that Chern classes of equivariant vector bundles
determined by different classes in $\pi_0 \big( \Omega_{\hat{D}_R,
(W_{d^i})_{i \in I^+}} \big)$ with respect to $V_B$ are all
different by calculating Chern class. For this, we use the
parametrization on $\lineL_R$ introduced in Section \ref{section:
clutching construction}. For notational simplicity, put $* =
\psi_{\bar{v}^0} (\bar{v}^0)$ and $*^\prime = \psi_{b(\bar{e}^0)}
(b(\bar{e}^0)).$ Let $\gamma_0 : [0, 1/2] \rightarrow \Iso_H
(V_{\bar{f}^{-1}}, V_{\bar{f}^0})$ be an element of the set
$\Omega_0$ such that $\gamma_0 (0) =
*$ and $\gamma_0 (1/2) =
*^\prime.$ Also, let $\sigma_0 :
[0, 1/2] \rightarrow \Iso_H (V_{\bar{f}^{-1}}, V_{\bar{f}^0})$ be a
loop such that $\sigma_0 (0) = \sigma_0 (1/2) = *$ and $[\sigma_0]$
is a generator of $\pi_1 (\Iso_H (V_{\bar{f}^{-1}}, V_{\bar{f}^0}),
* ).$ Here, we introduce some notations.

\begin{notation}
For some $H$-representations $W_1$ and $W_2,$ let $X_1$ and $X_2$ be
two spaces $\Iso_H (W_1, W_2)$ and $\Iso_H (W_2, W_1),$
respectively. For paths $\delta_1, \delta_2 : [a,b] \rightarrow X_1$
satisfying $\delta_1 (b) = \delta_2 (a),$ denote by $\delta_1 .
\delta_2 : [a,b] \rightarrow X_1$ the path defined by
\begin{equation*}
(\delta_1 . \delta_2) (t) = \left\{
  \begin{array}{ll}
    \delta_1 (a+2(t-a)),              & t \in [a, \frac {a+b} 2], \\
    \delta_2 (a+2(t- \frac {a+b} 2)), & t \in [\frac {a+b} 2, b].
  \end{array}
\right.
\end{equation*}
For paths $\delta_3 : [a,b] \rightarrow X_1$ and $\delta_4 : [b,c]
\rightarrow X_1$ satisfying $\delta_3 (b) = \delta_4 (b),$ denote by
$\delta_1 \vee \delta_2 : [a,c] \rightarrow X_1$ the path defined by
\begin{equation*}
(\delta_3 \vee \delta_4) (t) = \left\{
  \begin{array}{ll}
    \delta_3 (t), & t \in [a, b], \\
    \delta_4 (t), & t \in [b, c].
  \end{array}
\right.
\end{equation*}
Also, for a path $\delta : [a, b] \rightarrow X_1$ denote by
$\delta^* : [1-b,1-a] \rightarrow X_2$ the path $\delta^* (t) =
\delta (1-t)^{-1}.$
\end{notation}

Note that $\sigma_0.\gamma_0 \in \Omega_0 \subset \Omega_{\hat{D}_R,
(W_{d^i})_{i \in I^+}}.$ We would show that the difference between
$c_1 \big( F_{V_B} / \sigma_0.\gamma_0 \big)$ and $c_1 \big( F_{V_B}
/ \gamma_0 \big)$ is $12 \chi(\id),$ $24 \chi(\id),$ $60 \chi(\id)$
according to $R = \tetra,$ $\octa,$ $\icosa,$ respectively. Here,
$l_R = 12,$ $24,$ $60$ because the number of edges of
$\complexK_\tetra,$ $\complexK_\octa,$ $\complexK_\icosa$ is $6,$
$12,$ $30$ and $D_R$ is a half of an edge at each case,
respectively. For the calculation, we need to describe the
equivariant clutching maps precisely. Let $\Phi = \cup ~
\varphi_{\hat{e}}$ in $\Omega_{\hat{D}_R, (W_{d^i})_{i \in I^+}}$ be
the extension of $\gamma_0,$ i.e. $\varphi_{\hat{e}^0}(t) = \gamma_0
(t)$ for $t \in [0, 1/2].$ By Condition N1., $\varphi_{c(\hat{e}^0)}
(t) = \gamma_0 (1-t)^{-1}$ for $t \in [1/2, 1].$ Pick an element
$c_0$ of $(G_\chi)_{b(e^0)}$ such that $c_0 b(\bar{e}^0) = b \big(
c(\bar{e}^0) \big).$ By equivariance of $\Phi,$
\begin{align*}
\varphi_{c(\hat{e}^0)}(t) &= \Phi(t)  \\
                      &= c_0 \Phi (c_0^{-1}t) c_0^{-1} \\
                      &= c_0 \varphi_{\hat{e}^0} (c_0^{-1}t) c_0^{-1} \\
                      &= c_0 \gamma_0 (t) c_0^{-1}
\end{align*}
for $t \in [0, 1/2].$ Again by Condition N1.,
\begin{align*}
\varphi_{\hat{e}^0}(t) &= \varphi_{c(\hat{e}^0)}(1-t)^{-1} \\
                      &= c_0 \gamma_0 (1-t)^{-1} c_0^{-1}
\end{align*}
for $t \in [1/2, 1].$ This gives $\varphi_{\hat{e}^0} = \gamma_0
\vee c_0 \gamma_0^* c_0^{-1}$ and $\varphi_{c(\hat{e}^0)} = c_0
\gamma_0 c_0^{-1} \vee \gamma_0^*.$ Let $\Phi^\prime = \cup
\varphi_{\hat{e}}^\prime$ be the extension of $\sigma_0.\gamma_0.$
Then,
\begin{align*}
\varphi_{\hat{e}^0}^\prime &= \sigma_0.\gamma_0 \vee c_0
(\gamma_0^*.\sigma_0^*) c_0^{-1} \hbox{ and} \\
\varphi_{c(\hat{e}^0)}^\prime &= c_0 \sigma_0.\gamma_0 c_0^{-1} \vee
\gamma_0^*.\sigma_0^*
\end{align*}
by using $(\sigma_0.\gamma_0)^* = \gamma_0^*.\sigma_0^*.$ Replacing
only $\varphi_{\hat{e}}$ and $\varphi_{c(\hat{e}^0)}$ of $\Phi$ with
$\varphi_{\hat{e}}^\prime$ and $\varphi_{c(\hat{e}^0)}^\prime,$ we
obtain a new nonequivariant clutching map, say $\Phi_1.$ Here,
$\Phi_1$ becomes an nonequivariant clutching map because $\Phi_1 =
\Phi$ on vertices of $\lineL_R.$ The difference $c_1 \big( F_{V_B} /
\Phi_1 \big) -c_1 \big( F_{V_B} / \Phi \big)$ is equal to $\pm 2
\chi( \id )$ ( say $+2 \chi( \id )$ ) by Lemma \ref{lemma: reduce to
smaller matrix} because $\varphi_{\hat{e}^0}^\prime$ contains two
generators $\sigma_0$ and $c_0 \sigma_0^* c_0^{-1}.$ We would repeat
this argument for other edges. Pick edges $\hat{e}$ and $c(\hat{e})$
such that $\{ \hat{e}, c(\hat{e}) \} \ne \{ \hat{e}^0, c(\hat{e}^0)
\}.$ Then, there exists the unique $g_0 \in G_\chi$ up to $H$ such
that $g_0 \hat{e}^0 = \hat{e}.$ By equivariance,
\begin{align*}
\varphi_{\hat{e}} &= g_0 \cdot \varphi_{\hat{e}^0} \\
                  &= g_0 \cdot (\gamma_0 \vee c_0 \gamma_0^* c_0^{-1}) \\
                  &= g_0 \gamma_0 g_0^{-1} \vee g_0 c_0 \gamma_0^* c_0^{-1}
                   g_0^{-1},
\end{align*}
and
\begin{equation*}
\varphi_{\hat{e}}^\prime = g_0 \sigma_0.\gamma_0 g_0^{-1} \vee g_0
c_0 \gamma_0^*.\sigma_0^* c_0^{-1} g_0^{-1}.
\end{equation*}
Replacing only $\varphi_{\hat{e}}$ and $\varphi_{c(\hat{e}^0)}$ of
$\Phi_1$ with $\varphi_{\hat{e}}^\prime$ and
$\varphi_{c(\hat{e}^0)}^\prime,$ we obtain a new nonequivariant
clutching map, say $\Phi_2,$ is obtained, and the difference $c_1
\big( F_{V_B} / \Phi_2 \big) -c_1 \big( F_{V_B} / \Phi_1 \big)$ is
also equal to $2 \chi( \id )$ because $g_0$ preserves the
orientation. In this way, if we replace all $\varphi_{\hat{e}}$ of
$\Phi$ with $\varphi_{\hat{e}}^\prime$ to obtain $\Phi^\prime,$ then
the difference $c_1 \big( F_{V_B} / \Phi^\prime \big) -c_1 \big(
F_{V_B} / \Phi \big)$ is equal to $l_R \chi( \id ).$ Proposition
\ref{proposition: proposition for isomorphism} and these
calculations prove the proposition. \qed
\end{proof}

\section{Equivariant clutching maps when $\pr(\rho(G_\chi)) = \Z_n,$ $\D_n$}
\label{section: clutching construction cyclic dihedral}

Assume that $\rho(G_\chi)=R$ for some finite $R$ in Table
\ref{table: introduction} such that $\pr(R) = \Z_n,$ $\D_n.$ We
redefine notations introduced for cases of $\pr(R) = \tetra,$
$\octa,$ $\icosa$ to be in accordance with these cases, and mimic
what we have done in Section \ref{section: clutching construction},
\ref{section: relation}, \ref{section: equivariant clutching maps}
with those notations. In our treatment of cases when $R = \Z_n,
\langle a_n, -b \rangle,$ there are some differences with other
cases which are caused by existence of fixed points.

\begin{figure}[ht!]
\begin{center}
\begin{pspicture}(-5,-3.5)(4.5,3.7)\footnotesize

\psline[linewidth=0.5pt](-4,3.5)(-6,0.5)(-3,0.5)(-2,1.5)(-4,3.5)
\psline[linewidth=0.5pt](-3,0.5)(-4,3.5)
\psline[linewidth=0.5pt,linestyle=dotted](-5,1.5)(-2,1.5)
\psline[linewidth=0.5pt,linestyle=dotted](-4,3.5)(-5,1.5)(-6,0.5)

\uput[dl](-5.9,0.6){$\bar{v}_N^0$}
\uput[dr](-3.1,0.6){$\bar{v}_N^1$} \uput[r](-2.1,1.5){$\bar{v}_N^2$}
\uput[u](-4,3.5){$N$}

\uput[d](-4.4,0.6){$\bar{e}_N^0$} \uput[dr](-2.6,1.1){$\bar{e}_N^1$}

\uput[dl](-2,3.5){$\lineK_{R,N}$}

\psline[linewidth=0.5pt](-6,-1.5)(-3,-1.5)(-2,-0.5)(-5,-0.5)(-6,-1.5)
\psline[linewidth=0.5pt](-3,-1.5)(-4,-3.5)(-2,-0.5)
\psline[linewidth=0.5pt](-6,-1.5)(-4,-3.5)
\psline[linewidth=0.5pt,linestyle=dotted](-5,-0.5)(-4,-3.5)

\uput[dl](-5.9,-1.4){$\bar{v}_S^0$}
\uput[dr](-3.1,-1.4){$\bar{v}_S^1$}
\uput[r](-2.1,-0.4){$\bar{v}_S^2$}
\uput[l](-4.9,-0.4){$\bar{v}_S^3$} \uput[d](-4,-3.5){$S$}

\uput[d](-4.4,-1.4){$\bar{e}_S^0$} \uput[l](-2.5,-1){$\bar{e}_S^1$}
\uput[u](-3.5,-0.6){$\bar{e}_S^2$} \uput[l](-5.5,-1){$\bar{e}_S^3$}

\uput[dl](-2,-2.5){$\lineK_{R,S}$}

\psline[linearc=1]{->}(-0.8, -0.5)(0.8, -0.5)

\uput[u](0, -0.5){$\pi$}

\psline[linewidth=0.5pt](4,2.5)(2,-0.5)(4,-2.5)(6,0.5)(4,2.5)
\psline[linewidth=0.5pt](2,-0.5)(5,-0.5)(6,0.5)
\psline[linewidth=0.5pt](4,2.5)(5,-0.5)(4,-2.5)
\psline[linewidth=0.5pt, linestyle=dotted](2,-0.5)(3,0.5)(6,0.5)
\psline[linewidth=0.5pt, linestyle=dotted](4,2.5)(3,0.5)(4,-2.5)

\uput[l](2.1,-0.5){$v^0$} \uput[ul](5,-0.5){$v^1$}
\uput[ur](5.9,0.5){$v^2$} \uput[d](4,-2.5){$S$} \uput[u](4,2.5){$N$}

\uput[d](3.6,-0.4){$e^0$} \uput[l](5.7,0.1){$e^1$}
\end{pspicture}
\end{center}
\caption{\label{figure: simple form K_4} $\pi : \lineK_R \rightarrow
\complexK_R$ when $\complexK_R = \complexK_4$}
\end{figure}

First, we rewrite Section \ref{section: clutching construction}. In
these cases, $\complexK_R = \complexK_{m_R}$ with $m_R = n/2,$ $n,$
$2n$ by Table \ref{table: introduction}. Denote by $\complexK_{R,
S}$ the lower simplicial subcomplex of $\complexK_R$ whose
underlying space $|\complexK_{R, S}|$ is equal to $|\complexK_R|
\cap \{ (x,y,z) \in \R^3 ~|~ z \le 0 \},$ and by $\complexK_{R, N}$
the upper part. Let $\complexL_R$ be the subcomplex $\complexK_{R,
S} \cap \complexK_{R, N}$ of $\complexK_R$ lying on the equator
$z=0.$ Put $B = \{ S, N \} \subset |\complexK_R|$ on which $R$ (and
$G_\chi$) acts. Since each $| \complexK_{R, q} |$ for $q \in B$ has
a simple equivariant structure for $R_q,$ we consider $\complexK_R$
as the union of two pieces $\complexK_{R, q}$ for $q \in B,$ not of
all faces of $\complexK_R.$ Denote by $\lineK_R$ the disjoint union
$\amalg_{q \in B} ~ \complexK_{R, q},$ and denote by $\lineK_{R, q}$
the subcomplex $\complexK_{R, q}$ of $\lineK_R$ so that $|\lineK_R|
= \amalg_{q \in B} ~ |\lineK_{R, q}|.$ We denote simply by $\pi$ and
$|\pi|$ natural quotient maps from $\lineK_R$ and $|\lineK_R|$ to
$\complexK_R$ and $|\complexK_R|,$ respectively. Let $\lineL_R =
\pi^{-1}(\complexL_R),$ and let $\lineL_{R, q}$ be the subcomplex
$\lineL_R \cap \lineK_{R, q}$ of $\lineK_R$ for $q \in B.$ The
subset $|\pi|^{-1} (B)$ in $|\lineK_R|$ is often confused with $B$
in $|\complexK_R|.$ The injection from $\lineL_R$ to $\lineK_R$ and
its underlying space map are denoted by $\imath_{\lineL}$ and
$\imath_{|\lineL|},$ respectively. The $G_\chi$-actions on
$\complexK_R,$ $|\complexK_R|$ induce $G_\chi$-actions on
$\complexL_R,$ $\lineK_R,$ $\lineL_R,$ and their underlying spaces.
We do not need define $\hatL_R.$ Here, we introduce the following
notations:
\begin{align*}
\bar{v}_q^i                  &= \lineL_{R, q} \cap \pi^{-1} ( v^i ),  \\
\bar{e}_q^i                  &= \lineL_{R, q} \cap \pi^{-1} ( e^i ),  \\
\bar{D}_R                    &= |\lineL_{R, S}| \cap |\pi|^{-1} ( D_R ), \\
\bar{\mathbf{D}}_R           &= |\pi|^{-1} ( D_R ), \\
\bar{d}^{i^\prime}         &=  |\lineL_{R, S}| \cap |\pi|^{-1} ( d^{i^\prime} ) \\
\end{align*}
for $q \in B,$ $i \in \Z_{m_R},$ $i^\prime \in I.$ Some notations
are illustrated in Figure \ref{figure: simple form K_4}. Also, let
$c : \lineL_R \rightarrow \lineL_R$ be the simplicial map whose
underlying space map $|c| : |\lineL_R| \rightarrow |\lineL_R|$
satisfies
\begin{equation*}
\bar{x} \ne |c|(\bar{x}) \quad \text{and} \quad |\pi| (\bar{x}) =
|\pi| \big( |c|(\bar{x}) \big)
\end{equation*}
for each $\bar{x} \in |\lineL_R|.$ And, put $\bar{x}_0 = \bar{x}$
and $\bar{x}_1 = |c|(\bar{x})$ for each $\bar{x} \in |\lineL_R|.$
Now, we describe an equivariant vector bundle over $|\complexK_R|$
as equivariant clutching construction by using an equivariant vector
bundle over $|\lineK_R|.$ Pick a $G_\chi$-vector bundle $V_B$ over
$B$ such that $(\res_H^{G_\chi} V_B)|_q$ is $\chi$-isotypical at
each $q$ in $B.$ We denote by $V_q$ the isotropy representations of
$V_B$ at $q$ for each $q \in B.$ Then, we have
\begin{equation}
\label{equation: V_B cyclic dihedral} V_B \cong G_\chi
\times_{(G_\chi)_S} V_S \quad \text{ and } \quad \res_H^{(G_\chi)_S}
V_S \cong \res_H^{(G_\chi)_N} V_N
\end{equation}
when $R \ne \Z_n,$ $\langle a_n, -b \rangle.$ Define $\Vect_{G_\chi}
(|\complexK_R|, \chi)_{V_B}$ and $\Vect_{G_\chi} (|\lineK_R|,
\chi)_{V_B}$ as before. We observe that $\Vect_{G_\chi} (|\lineK_R|,
\chi)_{V_B}$ has the unique element $[F_{V_B}]$ for the bundle
\begin{equation*}
F_{V_B} = \left\{
  \begin{array}{ll}
    \amalg_{q \in B} ~ | \lineK_{R, q} | \times V_q \quad
    & \hbox{if } R = \Z_n, ~ \langle a_n, -b \rangle, \\
    G_\chi \times_{(G_\chi)_S} (| \lineK_{R, S} |
    \times V_S) \quad
    & \hbox{if } R \ne \Z_n, ~ \langle a_n, -b \rangle.
  \end{array}
\right.
\end{equation*}
Henceforward, we use trivializations
\begin{equation}
\label{equation: trivialization for cyclic dihedral}
\begin{array}{ll}
| \lineK_{R, q} | \times V_q & \quad \text{ for } \big(
\res_{(G_\chi)_q}^{G_\chi} F_{V_B} \big)
\big|_{| \lineK_{R, q} |}, \\
|\bar{e}_q^i| \times \res_{(G_\chi)_{b(\bar{e}_q^i)}}^{(G_\chi)_q}
V_q & \quad \text{ for } \big(
\res_{(G_\chi)_{b(\bar{e}_q^i)}}^{G_\chi} F_{V_B} \big)
\big|_{|\bar{e}_q^i|}
\end{array}
\end{equation}
for each $q$ and $i.$ Then, each $E$ in $\Vect_{G_\chi} (
|\complexK_R|, \chi)_{V_B}$ can be constructed by gluing $F_{V_B}
\cong |\pi|^* E$ along $|\lineL_{R, q}|$'s through
\begin{equation*}
|\lineL_{R, q}| \times V_q \rightarrow |\lineL_{R, q^\prime}| \times
V_{q^\prime}, ~ (\bar{x}, u) \mapsto ( |c|(\bar{x}), \varphi_q
(\bar{x}) u )
\end{equation*}
via continuous maps
\begin{equation*}
\varphi_q : |\lineL_{R, q}| \rightarrow \Iso(V_q, V_{q^\prime})
\end{equation*}
for $\bar{x} \in |\lineL_{R, q}|,$ $u \in V_q,$ $B = \{ q, q^\prime
\}$ as we have done in Section \ref{section: clutching
construction}. The union $\Phi = \cup_{q \in B} ~ \varphi_q$ is
called an equivariant clutching map of $E$ with respect to $V_B.$
This construction of $E$ is denoted by $F_{V_B} / \Phi.$ The map
$\Phi$ is defined on $|\lineL_R|,$ and we also regard $\Phi$ as a
map
\begin{equation*}
\imath_{|\lineL|}^* F_{V_B} \rightarrow \imath_{|\lineL|}^* F_{V_B},
\quad(\bar{x}, u) \mapsto \big( |c|(\bar{x}), \Phi (\bar{x}) u \big)
\end{equation*}
by using trivialization (\ref{equation: trivialization for cyclic
dihedral}) for each $q \in B,$ $(\bar{x}, u) \in |\lineL_{R, q}|
\times V_q.$ Also, $\Phi$ should be equivariant. An equivariant
clutching map of some bundle in $\Vect_{G_\chi} (|\complexK_R|,
\chi)_{V_B}$ with respect to $V_B$ is called simply an equivariant
clutching map with respect to $V_B,$ and let $\Omega_{V_B}$ be the
set of all equivariant clutching maps with respect to $V_B.$ And, if
we define $\Omega_{\bar{D}_R, V_B}$ as the set
\begin{equation*}
\{~ \Phi |_{\bar{D}_R} ~ | ~ \Phi \in \Omega_{V_B} ~\},
\end{equation*}
then the restriction map $\Omega_{V_B} \rightarrow
\Omega_{\bar{D}_R, V_B}$ is bijective, and $\pi_0 (\Omega_{V_B})
\cong \pi_0 (\Omega_{\bar{D}_R, V_B}).$ We call an equivariant
clutching map $\Phi$ the \textit{extension} of $\Phi|_{\bar{D}_R}.$
And, denote also by $F_{V_B} / ~ \Phi |_{\bar{D}_R}$ the bundle
$F_{V_B} / \Phi.$ Let $C^0 (|\lineL_R|, V_B)$ be the set of
functions $\Phi$ defined on $|\lineL_R|$ satisfying
$\Phi|_{|\lineL_{R, q}|}(\bar{x}) \in \Iso(V_q, V_{q^\prime})$ for
$\bar{x} \in |\lineL_{R, q}|$ and $B=\{ q, q^\prime \}.$ We can
define the quotient $F_{V_B} / \Phi$ for any $\Phi$ in $C^0
(|\lineL_R|, V_B).$ And, let $C^0 (\bar{D}_R, V_B)$ be the set
\begin{equation*}
\{~ \Phi |_{\bar{D}_R} ~ | ~ \Phi \in C^0 ( |\lineL_R|, V_B ) ~\}.
\end{equation*}
A function $\Phi$ in $C^0 (|\lineL_R|, V_B)$ or a function
$\Phi_{\bar{D}_R}$ in $C^0 (\bar{D}_R, V_B)$ is called a
\textit{preclutching map} with respect to $V_B.$ A preclutching map
$\Phi$ in $C^0 (|\lineL_R|, V_B)$ is an equivariant clutching map
with respect to $V_B$ if and only if it satisfies the following
conditions:
\begin{itemize}
  \item[N1$^\prime.$] $\Phi(|c|(\bar{x})) = \Phi(\bar{x})^{-1}$
  for each $\bar{x} \in |\lineL_R|,$
  \item[E1$^\prime.$] $\Phi(g\bar{x}) = g \Phi(\bar{x}) g^{-1}$
  for each $\bar{x} \in |\lineL_R|, g \in G_\chi.$
\end{itemize}
More precisely, if $\Phi$ satisfies Condition N1$^\prime.$, then the
quotient $F_{V_B} / \Phi$ becomes a nonequivariant vector bundle.
And, if $\Phi$ also satisfies Condition E1$^\prime.$, then $F_{V_B}
/ \Phi$ becomes an equivariant vector bundle, i.e. $\Phi$ is an
equivariant clutching map with respect to $V_B.$

Second, we rewrite Section \ref{section: relation}. We obtain Lemma
\ref{lemma: homotopy gives isomorphism} and \ref{lemma: equivalent
condition for isomorphism} by replacing $p_{|\lineL|}$ with
$\imath_{|\lineL|}.$ Define $\imath_\Omega,$ $p_\Omega,$ $p_{\pi_0}$
as before by replacing $\Omega_{\hat{D}_R, V_B}$ with
$\Omega_{\bar{D}_R, V_B}.$ We decompose $\Omega_{\bar{D}_R, V_B}.$
When $R = \Z_n,$ $\langle a_n, -b \rangle,$ denote a pair $(W_S,
W_N)$ in $A_{G_\chi} (S^2, \chi)$ by $(W_q)_{q \in B}.$ And, pick
the $G_\chi$-vector bundle $V_B$ over $B$ such that $V_B|_q = W_q$
for each $(W_q)_{q \in B} \in A_{G_\chi} (S^2, \chi)$ and $q \in B.$
Define $\Omega_{\bar{D}_R, (W_q)_{q \in B}} = (p_\Omega \circ
p_{\pi_0})^{-1} \big( (W_q)_{q \in B} \big).$ Then, we obtain
$\Omega_{\bar{D}_R, V_B} = \Omega_{\bar{D}_R, (W_q)_{q \in B}}$
because $V_B$ and $(W_q)_{q \in B}$ can be regarded as the same
$G_\chi$-bundle over $B.$ When $R \ne \Z_n,$ $\langle a_n, -b
\rangle,$ put
\begin{equation*}
V_B = G_\chi \times_{(G_\chi)_{d^{-1}}} W_{d^{-1}}, \quad F_{V_B} =
G_\chi \times_{(G_\chi)_{d^{-1}}} ( |\bar{f}^{-1}| \times W_{d^{-1}}
)
\end{equation*}
for each $(W_{d^i})_{i \in I^+} \in A_{G_\chi} (S^2, \chi).$ Then,
we obtain the decomposition of $\Omega_{\bar{D}_R, V_B}$ by
replacing $\Omega_{\hat{D}_R, (W_{d^i})_{i \in I^+}}$ of
(\ref{equation: decomposition of Omega}) with $\Omega_{\bar{D}_R,
(W_{d^i})_{i \in I^+}}.$ Proposition \ref{proposition: proposition
for isomorphism} holds if
\begin{enumerate}
  \item we replace $(W_{d^i})_{i \in I^+},$ $\Omega_{\hat{D}_R, (W_{d^i})_{i \in I^+}}$
  with $(W_q)_{q \in B},$ $\Omega_{\bar{D}_R, (W_q)_{q \in B}}$
  when $R = \Z_n,$ $\langle a_n, -b \rangle,$ respectively.
  \item we replace $\Omega_{\hat{D}_R, (W_{d^i})_{i \in I^+}}$ with
  $\Omega_{\bar{D}_R, (W_{d^i})_{i \in I^+}}$
  when $R \ne \Z_n,$ $\langle a_n, -b \rangle.$
\end{enumerate}

Third, we rewrite Section \ref{section: equivariant clutching maps}.
To begin with, we define some sets of equivariant pointwise
clutching maps. For each $\bar{x}$ in $|\lineL_R|$ and $x =
|\pi|(\bar{x}),$ put $\bar{\mathbf{x}} = |\pi|^{-1} (x) = \{
\bar{x}_j | j \in \Z_2 \},$ and let $\mathcal{A}_{\bar{x}}$ be the
set of equivariant pointwise clutching maps with respect to the
$(G_\chi)_x$-bundle $\big( \res_{(G_\chi)_x}^{G_\chi} F_{V_B} \big)
|_{\bar{\mathbf{x}}}.$ Also for each edge $\bar{e}$ in $\lineL_R,$
$\pi(\bar{e}) = e,$ $\bar{x}$ in $|\bar{e}|,$ put $\bar{\mathbf{x}}
= |\pi|^{-1} (x) = \{ \bar{x}_j | j \in \Z_2 \},$ and let
$\mathcal{A}_{\bar{e}}$ be the set of equivariant pointwise
clutching maps with respect to the $(G_\chi)_{|e|}$-bundle $\big(
\res_{(G_\chi)_{|e|}}^{G_\chi} F_{V_B} \big) |_{\bar{\mathbf{x}}}$
where $\mathcal{A}_{\bar{e}}$'s for different $\bar{x}$'s are
identified as in Section \ref{section: equivariant clutching maps}.
We can define saturations on $\mathcal{A}_{\bar{x}}$ and elements of
it, and also group actions on saturations. Evaluation maps
$\mathcal{A}_{\bar{x}} \rightarrow \mathcal{A}_{\bar{x}}^0$ and
$\mathcal{A}_{\bar{e}} \rightarrow \mathcal{A}_{\bar{e}}^0$ are
bijective for any $\bar{x} \in |\lineL_R|,$ $\bar{e} \in \lineL_R.$
With these, we can rewrite Theorem \ref{theorem: clutching
condition}.

\begin{theorem} \label{theorem: clutching condition cyclic dihedral}
Assume that $\pr(R) = \Z_n, \D_n.$ Then, a preclutching map
$\Phi_{\bar{D}_R}$ in $C^0 (\bar{D}_R, V_B)$ is in
$\Omega_{\bar{D}_R, V_B}$ if and only if there exists the unique
$\psi_{\bar{x}} \in \mathcal{A}_{\bar{x}}$ for each $\bar{x} \in
\bar{D}_R$ satisfying the following conditions:
\begin{enumerate}
  \item[E2$^\prime.$] $\psi_{\bar{x}} (\bar{x}) = \Phi_{\bar{D}_R} (\bar{x})$
  for each $\bar{x} \in \bar{D}_R,$
  \item[E3$^\prime.$] for each $\bar{x}, \bar{x}^\prime \in \bar{D}_R$
  and $x = |\pi|(\bar{x}), x^\prime = |\pi|(\bar{x}^\prime),$
  if $x^\prime = gx$ for some $g \in G_\chi,$ then
  $\bar{\psi}_{x^\prime} = g \cdot \bar{\psi}_x.$
\end{enumerate}
The set $( \psi_{\bar{x}} )_{\bar{x} \in \bar{D}_R}$ is called
\textit{determined} by $\Phi_{\bar{D}_R}.$
\end{theorem}
In cases when $R = \Z_n,$ $\langle a_n, -b \rangle,$ Theorem
\ref{theorem: clutching condition cyclic dihedral} is sufficient to
calculate the homotopy $\pi_0 ( \Omega_{\bar{D}_R, V_B} ),$ but in
other cases we need more. In the remaining of this section, assume
that $R \ne \Z_n,$ $\langle a_n, -b \rangle.$ With exactly the same
definition of $\bar{A}_{G_\chi} (S^2, V_B),$ we can rewrite
Corollary \ref{corollary: Omega}.

\begin{corollary} \label{corollary: Omega cyclic dihedral}
The set $\Omega_{\bar{D}_R, V_B}$ is equal to the set
\begin{align*}
\Big\{ & \Phi_{\bar{D}_R} \in C^0 (\bar{D}_R, V_B) \Big| ~
\Phi_{\bar{D}_R} (\bar{x}) \in \mathcal{A}_{\bar{e}}^0 \text{ for
each } \bar{x} \text{ and any } \bar{e} \text{ such that } \bar{x}
\in |\bar{e}|,
\text{ and } \\
&  \Phi_{\bar{D}_R} (\bar{d}^0) = \psi_{\bar{d}^0} (\bar{d}^0), ~
\Phi_{\bar{D}_R} (\bar{d}^1) = \psi_{\bar{d}^1} (\bar{d}^1) \text{
for some } (\psi_{\bar{d}^i})_{i \in I} \in \bar{A}_{G_\chi} (S^2,
V_B) ~ \Big\}.
\end{align*}
\end{corollary}

By using this, we would show nonemptiness of $\Omega_{\bar{D}_R,
(W_{d^i})_{i \in I^+}}.$ For this, we need a lemma.

\begin{lemma}
 \label{lemma: existence of barA cyclic dihedral}
For each $(W_{d^i})_{i \in I^+} \in A_{G_\chi} (S^2, \chi),$ if we
put
\begin{equation*}
V_B = G_\chi \times_{(G_\chi)_{d^{-1}}} W_{d^{-1}}, \quad F_{V_B} =
G_\chi \times_{(G_\chi)_{d^{-1}}} ( |\lineK_{R, S}| \times
W_{d^{-1}} ),
\end{equation*}
then each $\mathcal{A}_{\bar{x}}$ for $\bar{x} \in [\bar{d}^0,
\bar{d}^1]$ is nonempty. And, we can pick an element
$(\psi_{\bar{d}^i})_{i \in I}$ in $\bar{A}_{G_\chi} (S^2, V_B)$
which determines $(W_{d^i})_{i \in I^+}.$
\end{lemma}

\begin{proof}
For each $i \in I,$ put $\bar{\mathbf{x}} = |\pi|^{-1}(d^i) = \{
\bar{x}_j | \bar{x}_j = \bar{d}_j^i \text{ for } j \in \Z_2 \}.$ If
we put $F_i = \big( \res_{(G_\chi)_{d^i}}^{G_\chi} F_{V_B} \big)
|_{\bar{\mathbf{x}}}$ and $N_2 = (G_\chi)_{d^i},$ then $N_1 =
(G_\chi)_{\bar{d}^i}.$ By Table \ref{table: all about K_R} and
(\ref{equation: V_B cyclic dihedral}), at least one of $F_i$'s
satisfies Condition F2. except two cases $R = \Z_n \times Z$ with
odd $n,$ $\langle -a_n \rangle$ with even $n/2.$ Then, we obtain
nonemptiness as in Lemma \ref{lemma: existence of barA} in these
cases. In the remaining two cases, $(G_\chi)_{\bar{x}} = H$ for each
$\bar{x} \in [\bar{d}^0, \bar{d}^1]$ by Table \ref{table: all about
K_R}, and $F_i$'s satisfy Condition F1. because $\res_H^{(G_\chi)_S}
V_S \cong \res_H^{(G_\chi)_N} V_N$ by (\ref{equation: V_B cyclic
dihedral}). So, we obtain nonemptiness by Lemma \ref{lemma:
pointwise clutching for m=2 nontransitive}. The second statement is
proved as in Lemma \ref{lemma: existence of barA}. \qed
\end{proof}

By using these, we obtain nonemptiness of $\Omega_{\bar{D}_R,
(W_{d^i})_{i \in I^+}}.$

\begin{proposition}
 \label{proposition: nonempty omega cyclic dihedral}
For each $(W_{d^i})_{i \in I^+} \in A_{G_\chi} (S^2, \chi),$ the set
$\Omega_{\bar{D}_R, (W_{d^i})_{i \in I^+}}$ is nonempty.
\end{proposition}

\section{Proof for cases when $\pr(\rho(G_\chi)) = \Z_n,$ $\D_n$}
\label{section: cyclic dihedral cases}

In this section, we calculate homotopy of equivariant clutching maps
in cases when $\pr(R) = \Z_n,$ $\D_n.$ We parameterize each edge
$|\bar{e}_S^i|$ in $|\lineK_{R, S}|$ for $0 \le i \le m_R-1$ by the
interval $[i, i+1]$ linearly to satisfy $\bar{v}_S^i \mapsto i$ when
$m_R \ge 3.$ The vertex $\bar{v}_S^0$ is parameterized by 0 or $m_R$
according to context. First, we deal with cases when $R = \Z_n,$
$\langle a_n, -b \rangle.$ Similarly to Lemma \ref{lemma: elementary
lemma on isotropy representation}, we easily obtain the following
lemma:
\begin{lemma}
 \label{lemma: restricted isotropy representation}
Assume that $R = \Z_n,$ $\langle a_n, -b \rangle.$ Then,
\begin{equation*}
\res_{(G_\chi)_x}^{G_\chi} E_S \cong \res_{(G_\chi)_x}^{G_\chi} E_N
\end{equation*}
for each $E$ in $\Vect_{G_\chi} (|\complexK_R|, \chi)$ and $x \in
|\complexL_R|.$
\end{lemma}

\begin{proof}
In these cases, the great circle passing through $N, S, x$ is fixed
by $(G_\chi)_x$ because $G_\chi$ fixes $B.$ From this, we obtain a
proof by Lemma \ref{lemma: elementary lemma on isotropy
representation}. \qed
\end{proof}

This lemma says that $p_\vect$ is well-defined in cases when $R =
\Z_n,$ $\langle a_n, -b \rangle.$ In other cases, $p_\vect$ is
well-defined by Lemma \ref{lemma: elementary lemma on isotropy
representation} and Lemma \ref{lemma: intersection of isotropy}. For
a path $\gamma$ defined on $[0,1],$ define a path $\gamma_{-i}$ on
$[i, i+1]$ as $\gamma_{-i} (t) = \gamma (t-i)$ for $i \in \Z.$

\begin{proposition}
\label{proposition: cyclic case}
 Assume that $R = \Z_n.$
Then, Theorem \ref{main: by isotropy and chern} holds for the case.
\end{proposition}

\begin{proof}
For simplicity, we prove the proposition only when $m_R = n \ge 3.$
Other cases are similarly proved. For each $(W_q)_{q \in B}$ in
$A_{G_\chi} ( S^2, \chi ),$ pick the $G_\chi$-bundle $V_B$ over $B$
such that $V_B|_q = W_q$ for $q \in B.$ Proof is similar to
Proposition \ref{proposition: platonic orientable case}.

In this case, $(G_\chi)_x = H$ for each point $x$ in $|\complexL_R|$
by Table \ref{table: all about K_R} so that $\mathcal{A}_{\bar{x}} =
\mathcal{A}_{\bar{e}}$ and $\mathcal{A}_{\bar{x}}^0 = \Iso_H ( V_S,
V_N )$ by Lemma \ref{lemma: pointwise clutching for m=2
nontransitive} for each $\bar{x} \in |\lineL_{R, S}|$ and $\bar{e}
\in \lineL_{R, S}$ which are nonempty by Lemma \ref{lemma: existence
of barA cyclic dihedral}. Let $g_0$ be an element in $G_\chi$ such
that $\rho(g_0) = a_n,$ i.e. $g_0 d^0 = d^1.$ Similarly to Corollary
\ref{corollary: Omega cyclic dihedral}, $\Omega_{\bar{D}_R, (W_q)_{q
\in B}} = \Omega_{\bar{D}_R, V_B}$ is equal to
\begin{align*}
\Big\{ \Phi_{\bar{D}_R} \in C^0 ( \bar{D}_R, V_B ) ~ | ~
\Phi_{\bar{D}_R} ( \bar{x} ) \in \Iso_H (V_S, V_N) \text{ for each }
\bar{x} \in \bar{D}_R, \\
\text{ and } \Phi_{\bar{D}_R} (\bar{d}^1) = g_0 \Phi_{\bar{D}_R}
(\bar{d}^0) g_0^{-1} \Big\}
\end{align*}
by Theorem \ref{theorem: clutching condition cyclic dihedral}. Pick
a point $*$ in $\Iso_H (V_S, V_N),$ and put $*^\prime = g_0
* g_0^{-1}.$ Since $\Iso_H (V_S, V_N)$ is
path-connected, we may assume that each element $\Phi_{\bar{D}_R}$
in $\Omega_{\bar{D}_R, V_B}$ has $*,
*^\prime$ at $\bar{d}^0, \bar{d}^1,$ respectively. Then, it is easy that $\pi_0
( \Omega_{\bar{D}_R, V_B}) \cong \Z$ by Lemma \ref{lemma: relative
homotopy}.

Let $\gamma_0 : \bar{D}_R = [0,1] \rightarrow \Iso_H (V_S, V_N)$ be
an arbitrary element of $\Omega_{\bar{D}_R, V_B}$ such that
$\gamma_0 (0) = *,$ $\gamma_0 (1) = *^\prime,$ and let $\sigma_0 :
[0,1] \rightarrow \Iso_H (V_S, V_N)$ be a loop such that $\sigma_0
(0) = \sigma_0 (1) = *$ and $[\sigma_0]$ is a generator of $\pi_1
\big( \Iso_H (V_S, V_N), * \big).$ So, $\sigma_0.\gamma_0$ is
contained in $\Omega_{\bar{D}_R, V_B}.$ We would show that the
difference $c_1 \big( F_{V_B} / \sigma_0.\gamma_0 \big) - c_1 \big(
F_{V_B} / \gamma_0 \big)$ is equal to $n \chi(id)$ up to sign where
$l_R = n.$ For this, we need to describe the equivariant clutching
maps precisely. Let $\Phi = \cup_{ q \in B } ~ \varphi_q$ be the
extension of $\gamma_0,$ i.e. $\varphi_S (t) = \gamma_0 (t)$ for $t
\in [0, 1].$ By equivariance of $\Phi,$
\begin{align*}
\varphi_S(t) &= g_0^i \varphi_S(t-i) g_0^{-i} \\
                  &= g_0^i \gamma_0(t-i) g_0^{-i} \\
                  &= g_0^i (\gamma_0)_{-i}(t) g_0^{-i}
\end{align*}
for $0 \le i \le n-1$ and $t \in [i, i+1].$ That is,
\begin{equation*}
\varphi_S = \gamma_0 \vee \big( g_0 (\gamma_0)_{-1} g_0^{-1} \big)
\vee \cdots \vee \big( g_0^{n-1} (\gamma_0)_{-n+1} g_0^{-n+1} \big).
\end{equation*}
And, let $\Phi^\prime = \cup_{ q \in B } ~ \varphi_q^\prime$ be the
extension of $\sigma_0.\gamma_0.$ Then,
\begin{equation}
\label{equation: phi prime} \varphi_S^\prime = (\sigma_0.\gamma_0)
\vee \big( g_0 (\sigma_0.\gamma_0)_{-1} g_0^{-1} \big) \vee \cdots
\vee \big( g_0^{n-1} (\sigma_0.\gamma_0)_{-n+1} g_0^{-n+1} \big).
\end{equation}
Here, $g_0^i (\sigma_0.\gamma_0)_{-i} g_0^{-i}$ for $0 \le i \le
n-1$ is equal to
\begin{equation*}
g_0^i (\sigma_0)_{-i} ~ g_0^{-i} ~.~ g_0^i (\gamma_0)_{-i} ~
g_0^{-i},
\end{equation*}
and we can move $g_0^i (\sigma_0)_{-i} ~ g_0^{-i}$'s in
(\ref{equation: phi prime}) (nonequivariantly) homotopically to $*.$
For example,
\begin{align*}
& [ (\sigma_0.\gamma_0) \vee \big( g_0 (\sigma_0)_{-1} ~ g_0^{-1}
~.~ g_0 (\gamma_0)_{-1} ~ g_0^{-1}
\big) ] \\
= & [ (\sigma_0.\gamma_0. g_0 \sigma_0 g_0^{-1} ) \vee \big( g_0
(\gamma_0)_{-1} ~  g_0^{-1} \big) ] \\
= & [ \big( \sigma_0. \bar{\gamma}_0^{-1} ( g_0 \sigma_0 g_0^{-1} )
. \gamma_0 \big) \vee \big( g_0 (\gamma_0)_{-1} ~ g_0^{-1} \big) ] \\
= & [ ( \sigma_0. \sigma_0 . \gamma_0 ) \vee \big( g_0
(\gamma_0)_{-1} ~ g_0^{-1} \big) ].
\end{align*}
Since $\varphi_S^\prime$ contains $n$ $\sigma_0$'s, difference
between Chern classes is $n \chi(id)$ up to sign by Lemma
\ref{lemma: reduce to smaller matrix}. And, we can conclude that
$c_1$ on $\Omega_{\bar{D}_R, V_B}$ is injective and its image is
equal to the set $\{ \chi( \id ) ( nk + k_0 ) ~ | ~ k \in \Z \}$
where $k_0$ is dependent on the pair. Therefore, we obtain a proof
by Proposition \ref{proposition: proposition for isomorphism}.(1).
\qed
\end{proof}

\begin{proposition}
\label{proposition: <a_n, -b> case}
 Assume that
$R = \langle a_n, -b \rangle.$ Then, Theorem \ref{main: only by
isotropy} holds for the case.
\end{proposition}

\begin{proof}
For simplicity, we prove the proposition only when $m_R \ge 3.$
Other cases are similarly proved. For each $(W_q)_{q \in B}$ in
$A_{G_\chi} ( S^2, \chi ),$ pick the $G_\chi$-bundle $V_B$ over $B$
such that $V_B|_q = W_q$ for $q \in B.$

In this case, $(G_\chi)_x = H$ and $\mathcal{A}_{\bar{x}}^0 = \Iso_H
( V_S, V_N )$ for each interior $\bar{x} \in \bar{D}_R$ and its
image $x = |\pi|(\bar{x})$ by Table \ref{table: all about K_R} and
Lemma \ref{lemma: pointwise clutching for m=2 nontransitive}. Here,
$\Iso_H ( V_S, V_N )$ is nonempty by Lemma \ref{lemma: existence of
barA cyclic dihedral}. Similarly, $\mathcal{A}_{\bar{x}}^0 =
\Iso_{(G_\chi)_{x}} ( V_S, V_N )$ is nonempty for $\bar{x} =
\bar{d}^0, \bar{d}^1$ and its image $x = |\pi|(\bar{x}).$ Similarly
to Corollary \ref{corollary: Omega cyclic dihedral},
$\Omega_{\bar{D}_R, (W_q)_{q \in B}} = \Omega_{\bar{D}_R, V_B}$ is
equal to
\begin{align*}
\Big\{ \Phi_{\bar{D}_R} \in C^0 ( \bar{D}_R, V_B ) ~ | ~
\Phi_{\bar{D}_R} ( \bar{x} ) \in \Iso_H (V_S, V_N) \text{ for each }
\bar{x} \in \bar{D}_R, \\
\text{ and } \Phi_{\bar{D}_R} (\bar{x}) \in \mathcal{A}_{\bar{x}}^0
\text{ for } \bar{x} = \bar{d}^0, \bar{d}^1 \Big\}
\end{align*}
by Theorem \ref{theorem: clutching condition cyclic dihedral}. Then,
we obtain a proof by Lemma \ref{lemma: relative homotopy} because
the inclusion from $\Iso_{(G_\chi)_x} ( V_S, V_N )$ to $\Iso_H (
V_S, V_N )$ for $x = d^0, d^1$ induces surjection in the level of
fundamental groups by Lemma \ref{lemma: reduce to smaller matrix
cyclic extension}. \qed
\end{proof}

\begin{proposition}
\label{proposition: D_n x Z with even n case} Assume that $R$ is
equal to one of the following:
\begin{center}
\begin{tabular}{lll}
$\D_n \times Z,$ even $n,$ \qquad    & \qquad $\langle -a_n, -b \rangle,$ odd $n/2,$ \qquad    & \qquad $\langle -a_n, b \rangle,$ odd $n/2,$  \\
$\Z_n \times Z,$ even $n,$ \qquad    & \qquad $\langle -a_n \rangle,$ odd $n/2.$     \qquad    &    \\
\end{tabular}
\end{center}
Then, Theorem \ref{main: only by isotropy} holds for these cases.
\end{proposition}

\begin{proof}
Proof is similar to Proposition \ref{proposition: platonic
nonorientable case}. \qed
\end{proof}

\begin{proposition}
\label{proposition: D_n case} Assume that $R = \D_n.$ Then, Theorem
\ref{main: by isotropy and chern} holds for the case where $l_R =
2n.$
\end{proposition}

\begin{proof}
Proof is similar to Proposition \ref{proposition: platonic
orientable case}. \qed
\end{proof}

\begin{proposition}
\label{proposition: D_n x Z with odd n case} Assume that $R$ is
equal to one of the following:
\begin{center}
\begin{tabular}{lll}
$\D_n \times Z,$ odd $n,$  \qquad    & \qquad $\langle -a_n, b
\rangle,$ even $n/2,$ \qquad    & \qquad $\langle -a_n, -b \rangle,$
even $n/2.$
\end{tabular}
\end{center}
Then, Theorem \ref{main: only by isotropy} holds for these cases.
\end{proposition}

\begin{proof}
For each $( W_{d^i} )_{i \in I^+}$ in $A_{G_\chi} ( S^2, \chi ),$
put $V_B = G_\chi \times_{(G_\chi)_{d^{-1}}} W_{d^{-1}}.$ In these
cases, $\mathcal{A}_{\bar{x}}$ is nonempty for each $\bar{x} \in
\bar{D}_R$ by Lemma \ref{lemma: existence of barA cyclic dihedral}.
Let us investigate $\mathcal{A}_{\bar{x}}$ more precisely. First,
$(G_\chi)_x = H$ and $\mathcal{A}_{\bar{x}}^0 = \Iso_H ( V_S, V_N )$
for each interior $\bar{x} \in \bar{D}_R$ and its image $x =
|\pi|(\bar{x})$ by Table \ref{table: all about K_R} and Lemma
\ref{lemma: pointwise clutching for m=2 nontransitive}. Also,
$(G_\chi)_x / H \cong \Z_2$ for $x = d^0$ or $d^1$ (say it $d^0$),
and $\rho \big( (G_\chi)_x \big)$ has an element of the form $-a_n^i
b$ fixing $S$ for some $i$ by Table \ref{table: all about K_R}. From
this, $(G_\chi)_{d^0}$ fixes the great circle containing $S, N, d^0$
so that $V_S$ and $V_N$ are $(G_\chi)_{d^0}$-isomorphic as in Lemma
\ref{lemma: restricted isotropy representation}. And,
$\mathcal{A}_{\bar{d}^0}^0$ is homeomorphic to
$\Iso_{(G_\chi)_{d^0}} ( V_S, V_N )$ by Lemma \ref{lemma: pointwise
clutching for m=2 nontransitive}. Pick an element $\psi_{\bar{d}^1}$
in $\mathcal{A}_{\bar{d}^1}$ which determines $W_{d^1}.$ Note that
$d^0$ and $d^1$ are not in a $G_\chi$-orbit because one is a vertex
and the other is not a vertex in $D_R$ by Table \ref{table:
introduction}. Similarly to (\ref{equation: Omega}) of Proposition
\ref{proposition: nonempty omega}, $\Omega_{\bar{D}_R, ( W_{d^i}
)_{i \in I^+}}$ is equal to
\begin{align*}
\Big\{ \Phi_{\bar{D}_R} \in C^0 ( \bar{D}_R, V_B ) ~ | ~
\Phi_{\bar{D}_R} ( \bar{x} ) \in \Iso_H (V_S, V_N) \text{ for each }
\bar{x} \in \bar{D}_R, \\
\Phi_{\bar{D}_R} (\bar{d}^0) \in \Iso_{(G_\chi)_{d^0}} ( V_S, V_N ),
\text{ and } \Phi_{\bar{D}_R} (\bar{d}^1) \in
(\mathcal{A}_{\bar{d}^1})_{\psi_{\bar{d}^1}}^0 \Big\}
\end{align*}
by Theorem \ref{theorem: clutching condition cyclic dihedral} which
is nonempty by Proposition \ref{proposition: nonempty omega cyclic
dihedral}. Here, $(\mathcal{A}_{\bar{d}^1})_{\psi_{\bar{d}^1}}$ is
simply connected by Proposition \ref{proposition: psi for cyclic}
because $(G_\chi)_{d^1}/H \cong \Z_2$ and $\bar{d}^1,$
$|c|(\bar{d}^1)$ are in a $G_\chi$-orbit by Table \ref{table: all
about K_R}. And, the inclusion from $\Iso_{(G_\chi)_{d^0}} ( V_S,
V_N )$ to $\Iso_H ( V_S, V_N )$ induces surjection in the level of
fundamental groups by Lemma \ref{lemma: reduce to smaller matrix
cyclic extension}. Therefore, we obtain a proof by Lemma \ref{lemma:
relative homotopy}. \qed
\end{proof}

Now, we deal with remaining two cases $R = \Z_n \times Z$ with odd
$n$ or $\langle -a_n \rangle$ with even $n/2.$ For a path $\gamma$
defined on $[0,1],$ define a path $\gamma^{tr}$ on $[0,1]$ as
$\gamma^{tr} (t) = \gamma (1-t)$ for $t \in [0,1].$

\begin{proposition}
 \label{proposition: Z_n x Z with odd n case}
Assume that $R = \Z_n \times Z$ with odd $n$ or $\langle -a_n
\rangle$ with even $n/2.$ For each triple $(W_{d^i})_{i \in I^+} \in
A_{G_\chi} (S^2, \chi),$ we have
\begin{equation*}
\pi_0 ( \Omega_{\bar{D}_R, (W_{d^i})_{i \in I^+}} ) \cong \Z.
\end{equation*}
And, $c_1 \big( F_{V_B} / \Phi \big)$ is constant for $\Phi$ in
$\Omega_{\bar{D}_R, (W_{d^i})_{i \in I^+}}$ when $V_B = G_\chi
\times_{(G_\chi)_{d^{-1}}} W_{d^{-1}}.$
\end{proposition}

\begin{proof}
In these cases, $\mathcal{A}_{\bar{x}}$ is nonempty for each
$\bar{x} \in \bar{D}_R$ by Lemma \ref{lemma: existence of barA
cyclic dihedral}. Since $(G_\chi)_x=H$ for each $\bar{x} \in
\bar{D}_R$ and its image $x = |\pi|(\bar{x})$ by Table \ref{table:
all about K_R}, the triple satisfies
\begin{equation*}
W_{d^{-1}} = V_S \text{ and } W_{d^i} \cong \res_H^{(G_\chi)_S} V_S
\end{equation*}
for $i \in I$ by Definition \ref{definition: A_R}, i.e. each triple
is determined by the third entry $W_{d^{-1}}.$ So,
$\Omega_{\bar{D}_R, (W_{d^i})_{i \in I^+}}$ $= \Omega_{\bar{D}_R,
V_B}$ because $V_B$ is also determined by $W_{d^{-1}}.$ And,
$\mathcal{A}_{\bar{x}}^0 = \mathcal{A}_{\bar{e}}^0 = \Iso_H ( V_S,
V_N )$ for $\bar{x} \in |\lineL_{R, S}|,$ $\bar{e} \in \lineL_{R,
S}$ by Lemma \ref{lemma: pointwise clutching for m=2 nontransitive}
so that $\pi_0 ( \mathcal{A}_{\bar{x}}^0 ) \cong \Z$ by Schur's
Lemma because $V_S$ is $H$-isotypical.

Let $g_0 \in {G_\chi}$ be an element such that
\begin{enumerate}
  \item $\rho(g_0) = -a_n^{n+1/2}$ if $\rho(G_\chi) = \Z_n \times Z$ with
  odd $n,$
  \item $\rho(g_0) = -a_n^{n/2 +1}$ if $\rho(G_\chi) = \langle -a_n \rangle$ with
  even $n/2.$
\end{enumerate}
In both cases, $g_0 \bar{v}_S^i = \bar{v}_N^{i+1}$ for each $i \in
\Z_{m_R}$ in $\lineL_R.$ Similarly to Corollary \ref{corollary:
Omega cyclic dihedral}, $\Omega_{\bar{D}_R, (W_{d^i})_{i \in I^+}} =
\Omega_{\bar{D}_R, V_B}$ is equal to
\begin{align*}
\Big\{ \Phi_{\bar{D}_R} \in C^0 ( \bar{D}_R, V_B ) ~ | ~
\Phi_{\bar{D}_R} ( \bar{x} ) \in \Iso_H (V_S, V_N) \text{ for each }
\bar{x} \in \bar{D}_R, \\
\text{ and } \Phi_{\bar{D}_R} (\bar{d}^1) = g_0
\Phi_{\bar{D}_R}^{-1} (\bar{d}^0) g_0^{-1} \Big\}
\end{align*}
by Theorem \ref{theorem: clutching condition cyclic dihedral}. Pick
an element $*$ in $\mathcal{A}_{\bar{d}^0}^0 = \Iso_H ( V_S, V_N ),$
and put $*^\prime = g_0
*^{-1} g_0^{-1}$ in $\mathcal{A}_{\bar{d}^1}^0 = \Iso_H ( V_S, V_N ).$
Since $\Iso_H (V_S, V_N)$ is path-connected, we may assume that each
element $\Phi_{\bar{D}_R}$ in $\Omega_{\bar{D}_R, V_B}$ satisfies
\begin{equation*}
\Phi_{\bar{D}_R} (\bar{d}^0) =
* \text{ and } \Phi_{\bar{D}_R} (\bar{d}^1) = *^\prime.
\end{equation*}
So, $\pi_0 ( \Omega_{\bar{D}_R, V_B} ) \cong \Z$ by Lemma
\ref{lemma: relative homotopy} because $\pi_1 \big( \Iso_H ( V_S,
V_N ) \big) \cong \Z.$

Next, we calculate the first Chern class. Our calculation is done in
a similar way with Proposition \ref{proposition: cyclic case}. Pick
an arbitrary element $\Phi_{\bar{D}_R}$ in $\Omega_{\bar{D}_R, V_B}$
such that $\Phi_{\bar{D}_R} (\bar{d}^0) =
*$ and $\Phi_{\bar{D}_R} (\bar{d}^1) = *^\prime.$ Put $\gamma(t) =
\Phi_{\bar{D}_R}(t)$ on $t \in [0,1].$ Take a loop $\sigma_0 : [0,
1] \rightarrow \Iso_H ( V_S, V_N )$ such that $\sigma_0 (0) =
\sigma_0 (1) = *$ and $[ \sigma_0 ]$ is a generator of $\pi_1 (
\Iso_H ( V_S, V_N ), * ).$ To prove our result, we only have to show
that $c_1 \big( F_{V_B} / \sigma_0.\gamma \big) = c_1 \big( F_{V_B}
/ \gamma \big)$ where $\sigma_0.\gamma$ is in $\Omega_{\bar{D}_R,
V_B}$ by Corollary \ref{corollary: Omega cyclic dihedral}. Let $\Phi
= \cup_{q \in B} ~ \varphi_q$ be the extension of $\gamma,$ i.e.
$\varphi_S (t) = \gamma(t)$ on $t \in [0,1].$ Then, equivariance of
$\Phi$ shows that
\begin{align*}
\varphi_S (t) &= g_0 \varphi_N (g_0^{-1} t) g_0^{-1} \\
                   &= g_0 \varphi_S (g_0^{-1} t)^{-1} g_0^{-1} \\
                   &= g_0 \gamma(t-1)^{-1} g_0^{-1}
\end{align*}
for $t \in [1,2].$ That is, $\varphi_S = \gamma \wedge g_0
\gamma_{-1}^{-1} g_0^{-1}$ on $[0,2].$ By using this, if
$\Phi^\prime = \cup_{q \in B} ~ \varphi_q^\prime$ is the extension
of $\sigma_0 . \gamma,$ then we have
\begin{align*}
\varphi_S^\prime &= \sigma_0 . \gamma \wedge g_0 \big(
(\sigma_0)_{-1}^{-1} ~ . ~ \gamma_{-1}^{-1} \big) g_0^{-1} \\
&= \sigma_0 . \gamma \wedge \Big ( g_0 (\sigma_0)_{-1}^{-1} g_0^{-1}
~ . ~ g_0 \gamma_{-1}^{-1} g_0^{-1} \Big)
\end{align*}
on $[0,2].$ Here, two generators $[\sigma_0]$ in $\pi_1 \big( \Iso_H
( V_S, V_N ), * \big)$ and $\big[ g_0 (\sigma_0)_{-1}^{-1} g_0^{-1}
\big]$ in $\pi_1 \big( \Iso_H ( V_S, V_N ), *^\prime \big)$ cancel
each other in $\pi_1 ( \Iso_H ( V_S, V_N ) ).$ This is in contrast
with the calculation of Proposition \ref{proposition: cyclic case}.
Since the whole $\varphi_S$ and $\varphi_S^\prime$ are equivariantly
determined by $\gamma$ and $\sigma_0.\gamma$ as in the proof of
Proposition \ref{proposition: platonic orientable case}, we obtain
$c_1 \big( F_{V_B} / \sigma_0.\gamma \big) = c_1 \big( F_{V_B} /
\gamma \big).$ \qed
\end{proof}

The proposition says that Proposition \ref{proposition: proposition
for isomorphism} does not applies to these two cases. So, we need to
apply Lemma \ref{lemma: equivalent condition for isomorphism} to
these cases. For this, we prove two technical lemmas.

\begin{lemma}
\label{lemma: bundle isomorphism}
 Assume that $R = \Z_n \times Z$ with odd $n$ or
$\langle -a_n \rangle$ with even $n/2.$ For each $(W_{d^i})_{i \in
I^+} \in A_{G_\chi} (S^2, \chi),$ put $V_B = G_\chi
\times_{(G_\chi)_{d^{-1}}} W_{d^{-1}}.$ Let $\eta_S : |\lineK_{R,
S}| \times V_S \rightarrow |\lineK_{R, S}| \times V_S$ be a
$(G_\chi)_S$-isomorphism. For an element $g$ in $G_\chi$ such that
$g S = N,$ let $\eta_N : |\lineK_{R, N}| \times V_N \rightarrow
|\lineK_{R, N}| \times V_N$ be defined by $\eta_N (x) = g \eta_S (
g^{-1} x ) g^{-1}.$ Then,
\begin{enumerate}
  \item $\eta = \cup_{q \in B} ~ \eta_q$ is the unique ${G_\chi}$-isomorphism of
  $F_{V_B}$ extending $\eta_S.$
  \item For a map $\Phi = \bigcup_{ q \in B } ~ \varphi_q$
        in $\Omega_{V_B},$
        the map $\Phi^\prime = \bigcup_{ q \in B } ~ \varphi_q^\prime$
        with $\varphi_q^\prime = \eta_{q^\prime} \varphi_q
        \eta_q^{-1}$ is also an equivariant clutching map satisfying
        the following commutative diagram
\end{enumerate}

\begin{equation*}
\SelectTips{cm}{} \xymatrix{ |\lineL_{R, q}| \times V_q
\ar[r]^-{\eta_q |_{|\lineL_{R, q}|}} \ar[d]^-{\varphi_q} &
|\lineL_{R, q}| \times V_q
\ar[d]^-{\varphi_q^\prime}\\
|\lineL_{R, q^\prime}| \times V_{q^\prime} \ar[r]^-{\eta_{q^\prime}
|_{|\lineL_{R, q^\prime}|}} & |\lineL_{R, q^\prime}| \times
V_{q^\prime} }
\end{equation*}
when $B = \{ q, q^\prime \}.$
\end{lemma}

\begin{proof}
Well-definedness is an only issue here, and normality of $(G_\chi)_S
= (G_\chi)_N$ in ${G_\chi}$ is used for this. \qed
\end{proof}

\begin{remark}
~
\begin{enumerate}
  \item In this lemma, $\eta$ gives a $G_\chi$-isomorphism
  between $F_{V_B} / \Phi$ and $F_{V_B} / \Phi^\prime.$
  \item Let $\eta_{S, t}$ for $t \in [0,1]$ be a homotopy
of $(G_\chi)_S$-isomorphisms of $|\lineK_{R, S}| \times V_S.$ Let
$\Phi_t^\prime = \bigcup_{q \in B} ~\varphi_{q, t}^\prime$ for each
$t \in [0,1]$ be the equivariant clutching map determined by
$\eta_{S, t}$ and $\Phi$ in this lemma. Then, $F_{V_B} /
\Phi_0^\prime$ and $F_{V_B} / \Phi_1^\prime$ are $G_\chi$-isomorphic
by Lemma \ref{lemma: homotopy gives isomorphism}. \qed
\end{enumerate}
\end{remark}

\begin{lemma}
\label{lemma: bundle isomorhism standard form}
 Assume that $R = \Z_n \times Z$ with odd $n$ or
$\langle -a_n \rangle$ with even $n/2.$ For each $(W_{d^i})_{i \in
I^+} \in A_{G_\chi} (S^2, \chi),$ put $V_B = G_\chi
\times_{(G_\chi)_{d^{-1}}} W_{d^{-1}}.$ Then, each
${G_\chi}$-isomorphism $\eta = \cup_{q \in B} ~ \eta_q$ of $F_{V_B}$
is equivariantly homotopic to a ${G_\chi}$-isomorphism $\eta^\prime
= \cup_{q \in B} ~ \eta_q^\prime$ of $F_{V_B}$ such that
\begin{equation*}
\eta_S^\prime |_{|\bar{e}_S^0| \cup |\bar{e}_S^1|} ~ = ~ \gamma
\wedge (\gamma^{tr})_{-1}
\end{equation*}
for some loop $\gamma: \bar{D}_R = [0,1] \rightarrow \Iso_H ( V_S )$
satisfying $\gamma (\bar{v}_S^0) = \gamma (\bar{v}_S^1) = \id$ where
$(\gamma^{tr})_{-1}$ is defined on $[1,2].$
\end{lemma}

\begin{proof}
Note that $R_S = \Z_{m_R /2}$ for both cases, and pick an element
$g_1 \in (G_\chi)_S$ such that $g_1 \bar{v}_S^0 = \bar{v}_S^2.$
Consider $| \lineK_{R, S} |$ as the quotient $|\lineL_{R, S}| \times
[0,1] / |\lineL_{R, S}| \times 0,$ and parameterize points of it by
$(\bar{x},t)$ with $\bar{x} \in |\lineL_{R, S}|$ and $t \in [0,1].$
Let $\gamma_i : [1/2, 1] \rightarrow \Iso_H (V_S)$ be paths such
that
\begin{equation*}
\gamma_i(1/2) = \eta_S (\bar{v}_S^i), \quad \gamma_i (1) = \id,
\quad \gamma_2 (t) = g_1 \gamma_0 (t) g_1^{-1}
\end{equation*}
for each $t,$ $i \in \{ 0, 2 \}.$ Then, define
\begin{equation*}
\eta_S^\prime : |\lineL_{R, S}| \times [0,1/2] ~ \bigcup ~ \Big(
\cup_{i = 0,2} ~ ( \bar{v}_S^i \times [1/2, 1] ) \Big)
\longrightarrow \Iso_H (V_S)
\end{equation*}
to satisfy
\begin{equation*}
  \begin{array}{ll}
    \eta_S^\prime (\bar{x},t) = \eta_S (\bar{x}, 2t)      & \text{ for } t \in [0, 1/2], \\
    \eta_S^\prime (\bar{v}_S^i, t) = \gamma_i (t)   & \text{ for } t \in [1/2, 1], i \in \{ 0,2 \}.
  \end{array}
\end{equation*}
We can extend $\eta_S^\prime$ to $|\lineL_{R, S}| \times [0,1/2]
\bigcup \big( |\bar{e}_S^0| ~ \cup ~ |\bar{e}_S^1| \big) \times
[1/2, 1].$ Then, $\eta_S^\prime$ is equivariantly extended to all
$|\lineK_{R, S}|$ so that $\eta_S^\prime$ is a
$(G_\chi)_S$-isomorphism of $| \lineK_{R, S} | \times V_S.$ Since
$\eta_S^\prime$ is defined on $| \lineK_{R, S} |,$ $\eta_S^\prime
|_{|\lineL_{R, S}|}$ is a (nonequivariantly) homotopically trivial
map when we consider $\eta_S^\prime |_{|\lineL_{R, S}|}$ as a map
from $S^1$ to $\Iso_H (V_S).$ Let $\alpha: |\bar{e}_S^0| \cup
|\bar{e}_S^1| \rightarrow \Iso_H (V_S)$ be the restriction
$\eta_S^\prime |_{|\bar{e}_S^0| \cup |\bar{e}_S^1|}.$ Note that
$\alpha ( \bar{v}_S^0 ) = \alpha ( \bar{v}_S^2 ) = \id.$ Since
$\eta_S^\prime |_{|\lineL_{R, S}|}$ is the ${G_\chi}$-orbit of
$\alpha,$ $\alpha$ should be also (nonequivariantly) homotopically
trivial. Therefore, we may assume that $\alpha = \gamma \wedge
(\gamma^{tr})_{-1}$ for some loop $\gamma: \bar{D}_R \rightarrow
\Iso_H ( V_S )$ satisfying $\gamma (\bar{v}_S^0) = \gamma
(\bar{v}_S^1) = \id.$ \qed
\end{proof}

Now, we can prove Theorem \ref{main: not by isotropy and chern
class}.

\begin{proof}[Proof of Theorem \ref{main: not by isotropy and chern
class}] For each $(W_{d^i})_{i \in I^+},$ put $V_B = G_\chi
\times_{(G_\chi)_{d^{-1}}} W_{d^{-1}}.$ Since $\Omega_{\bar{D}_R,
V_B}$ $=$ $\Omega_{\bar{D}_R, (W_{d^i})_{i \in I^+}}$ as we have
seen in the proof of Proposition \ref{proposition: Z_n x Z with odd
n case}, we have $p_\vect^{-1} \big( (W_{d^i})_{i \in I^+} \big)$
$=$ $\Vect_{G_\chi} (S^2, \chi)_{V_B}.$ So, it suffices to show that
the set $\Vect_{G_\chi} (S^2, \chi)_{V_B}$ has exactly two elements
for each $V_B$ to prove the first statement.

Let $\Phi = \cup_{ q \in B } ~ \varphi_q$ be a map in
$\Omega_{V_B}.$ Put $\Phi (\bar{d}^0) = *$ and $\Phi (\bar{d}^1) =
*^\prime.$ Let $g_0 \in {G_\chi}$ be the element in the proof of
Proposition \ref{proposition: Z_n x Z with odd n case} such that
$g_0 \bar{v}_S^i = \bar{v}_N^{i+1}$ for each $i.$ We would determine
which classes of $\Omega_{\bar{D}_R, V_B}$ give isomorphic
equivariant vector bundles. Take a loop $\sigma_0 : \bar{D}_R =
[0,1] \rightarrow \Iso_H (V_S)$ such that $\sigma_0 (0) = \sigma_0
(1) = \id$ and $[\sigma_0]$ is a generator of $\pi_1 \big( \Iso_H
(V_S), \id \big).$ For each ${G_\chi}$-isomorphism $\eta = \cup_{q
\in B} ~ \eta_q$ of $F_{V_B},$ we may assume that $\eta_S |_{
|\bar{e}_S^0| \cup |\bar{e}_S^1| } = \sigma_0^j \wedge
(\sigma_0^{tr})_{-1}^j$ for some $j \in \Z$ by Lemma \ref{lemma:
bundle isomorhism standard form}. Put
\begin{equation*}
\varphi_S^\prime (\bar{x}) = \eta_N (|c|(\bar{x})) \varphi_S
(\bar{x}) \eta_S(\bar{x})^{-1}
\end{equation*}
for all $\bar{x} \in \bar{D}_R.$ By equivariance,
\begin{equation*}
\eta_N (|c|(\bar{x})) = g_0^{-1} \eta_S \big( g_0 |c|(\bar{x}) \big)
g_0 = g_0^{-1} \eta_S (\bar{x}+1) g_0.
\end{equation*}
Since $\eta_S |_{ |\bar{e}_S^0| \cup |\bar{e}_S^1| } = \sigma_0^j
\wedge (\sigma_0^{tr})_{-1}^j,$
\begin{equation*}
\varphi_S^\prime (\bar{x}) = g_0^{-1} ~ \sigma_0^{tr} (\bar{x})^j ~
g_0 ~ \varphi_S (\bar{x}) ~ \sigma_0 (\bar{x})^{-j}
\end{equation*}
for $\bar{x} \in \bar{D}_R.$ And, this path is nonequivariantly
homotopic to $\varphi_S (\bar{x}) \sigma_0 (\bar{x})^{-2j}$ in
$\big[ [0,1], 0, 1; \Iso_H (V_S, V_N), *, *^\prime \big].$ Since
each class in $\Omega_{\bar{D}_R, V_B}$ is nonequivariantly
homotopic to one of $\varphi_S (\bar{x}) \sigma_0 (\bar{x})^i$ for
$i \in \Z$ by Proposition \ref{proposition: Z_n x Z with odd n
case}, this says that the equivariant vector bundle determined by an
equivariant clutching map on $\bar{D}_R$ with respect to $V_B$
depends only on the parity of $i$ by Lemma \ref{lemma: equivalent
condition for isomorphism}. Therefore, $\Vect_{G_\chi} (S^2,
\chi)_{V_B}$ have two different ${G_\chi}$-bundles. By the arguments
on parity, we similarly obtain the second statement. \qed
\end{proof}

\section{Proof for cases when $\pr( \rho(G_\chi) ) = \SO(2),$ $\orthogonal(2)$}
\label{section: nonzero-dimensional case}

Assume that $\rho(G_\chi)=R$ for some one-dimensional $R$ in Table
\ref{table: introduction}. In these cases, $S^2$ can not have
equivariant simplicial complex structure. So, we introduce new
notations. Consider the disjoint union $\bar{S}^2 = S_S^2 \amalg
S_N^2$ of the lower and upper hemispheres $S_S^2$ and $S_N^2$ of
$S^2,$ and denote by $\bar{S}_S^2$ and $\bar{S}_N^2$ hemispheres
$S_S^2$ and $S_N^2$ in $\bar{S}^2,$ respectively. Denote by $\pi$
the natural quotient maps from $\bar{S}^2$ to $S^2.$ Denote by
$\bar{S}_S^1$ and $\bar{S}_N^1$ boundaries of $\bar{S}_S^2$ and
$\bar{S}_N^2$ in $\bar{S}^2,$ respectively. And, denote by
$\bar{S}^1$ the disjoint union $\bar{S}_S^1 \amalg \bar{S}_N^1$
which is the preimage of the equator through $\pi.$ Put $B = \{ S, N
\} \subset S^2$ on which $R$ (and $G_\chi$) acts. The subset
$\pi^{-1} (B)$ in $\bar{S}^2$ is often confused with $B$ in $S^2.$
Denote by $\bar{v}_q^0$ be the point $\pi^{-1}(v^0) \cap
\bar{S}_q^1$ for $q \in B$ where $v^0 = (1, 0, 0).$ Note that if $R
=$ $\SO(2),$ $\langle \SO(2), -b \rangle,$ then $R$ fixes $B,$ and
otherwise $R$ acts transitively on $B.$ We can redefine notations
$V_B,$ $F_{V_B},$ $\Phi,$ $F_{V_B}/\Phi,$ $\imath_\Omega,$
$p_\Omega,$ $p_{\pi_0},$ $\Omega_{V_B},$ $\cdots$ of Section
\ref{section: clutching construction cyclic dihedral} by replacing
$|\complexK_R|,$ $|\lineK_R|,$ $|\lineK_{R, q}|,$ $|\lineL_R|,$
$|\lineL_{R, q}|$ with $S^2,$ $\bar{S}^2,$ $\bar{S}_q^2,$
$\bar{S}^1,$ $\bar{S}_q^1,$ respectively. Here, notations for cases
when $R =$ $\SO(2),$ $\langle \SO(2), -b \rangle$ and $R \ne$
$\SO(2),$ $\langle \SO(2), -b \rangle$ are redefined in the same way
with cases when $R = \Z_n,$ $\langle a_n, -b \rangle$ and $R \ne
\Z_n,$ $\langle a_n, -b \rangle$ of Section \ref{section: clutching
construction cyclic dihedral}, respectively.

Put $\bar{\mathbf{x}} = \{ \bar{x}_j ~ | ~ j \in \Z_2 \}$ with
$\bar{x}_0 = \bar{v}_S^0$ and $\bar{x}_1 = \bar{v}_N^0.$ Let
$\mathcal{A}_{\bar{v}_S^0}$ be the set of equivariant pointwise
clutching maps with respect to the $(G_\chi)_{v^0}$-bundle $\big(
\res_{(G_\chi)_{v^0}}^{G_\chi} F_{V_B} \big) |_{ \bar{\mathbf{x}}
}.$

\begin{proposition}
 \label{proposition: equivariant homotopy nonzero-dimensional no fixed point}
Assume that $R$ is equal to one of the following:
\begin{center}
\begin{tabular}{lll}
$\orthogonal(2) \times Z,$  \quad    & \quad $\orthogonal(2),$ \quad
& \quad $\langle \SO(2), ~ -a_2 \rangle.$
\end{tabular}
\end{center}
Then, Theorem \ref{main: only by isotropy} holds for these cases.
\end{proposition}

\begin{proof}
Put $V_B = G_\chi \times_{(G_\chi)_{d^{-1}}} W_{d^{-1}}$ for each $(
W_{d^i} )_{i \in I^+}$ in $A_{G_\chi} ( S^2, \chi ).$ By Proposition
\ref{proposition: proposition for isomorphism}.(2), we only have to
show that $\pi_0 ( \Omega_{(W_{d^i})_{i \in I^+}} )$ consists of
exactly one element for each $(W_{d^i})_{i \in I^+}.$ First, we show
nonemptiness of $\mathcal{A}_{\bar{v}_S^0}.$ For those $R$'s, the
$(G_\chi)_{v^0}$-bundle $F = \big( \res_{(G_\chi)_{v^0}}^{G_\chi}
F_{V_B} \big) |_{ \bar{\mathbf{x}} }$ satisfies Condition F2. by
Table \ref{table: all about K_R}. Since $d^0 = d^1 = v^0,$ we have
$W_{d^0} \cong W_{d^1}$ which is an $(G_\chi)_{v^0}$ extension of
$F_{\bar{x}_0}=$ $\res_{ (G_\chi)_S \cap (G_\chi)_{v^0}
}^{(G_\chi)_S} W_{d^{-1}}$ by Definition \ref{definition: A_R}. This
means that $\mathcal{A}_{\bar{v}_S^0}$ is nonempty by Theorem
\ref{theorem: bijectivity with extensions} as in the proof of Lemma
\ref{lemma: existence of barA}. Pick an element $\psi_{\bar{v}_S^0}$
in $\mathcal{A}_{\bar{v}_S^0}$ which determines $W_{d^0}$

Consider the evaluation map at $\bar{v}_S^0$
\begin{equation*}
\Omega_{(W_{d^i})_{i \in I^+}} \rightarrow
(\mathcal{A}_{\bar{v}_S^0})_{\psi_{\bar{v}_S^0}}^0, \quad \Phi
\mapsto \Phi(\bar{v}_S^0).
\end{equation*}
By definition of $\Omega_{(W_{d^i})_{i \in I^+}},$ the isotropy
representation $(F_{V_B} / \Phi)_{d^0}$ for each $\Phi$ is
isomorphic to $W_{d^0}$ so that $\Phi(\bar{v}_S^0)$ is contained in
$(\mathcal{A}_{\bar{v}_S^0})_{\psi_{\bar{v}_S^0}}^0,$ i.e. the
evaluation map is well-defined. We show that it is a one-to-one
correspondence. For those $R$'s, ${G_\chi}$ acts transitively on
$\bar{S}^1.$ So, each map $\Phi$ in $\Omega_{V_B}$ is determined by
$\Phi (\bar{v}_S^0)$ through equivariance, and the evaluation is
injective. For each $\psi$ in
$(\mathcal{A}_{\bar{v}_S^0})_{\psi_{\bar{v}_S^0}},$ we construct a
map $\Phi$ to satisfy
\begin{equation*}
\Phi ( g \bar{v}_S^0 ) = g \psi( \bar{v}_S^0 ) g^{-1}.
\end{equation*}
for each $g \in {G_\chi},$ especially $\Phi |_{ \bar{\mathbf{x}} } =
\psi.$ We can show that this is a well-defined equivariant clutching
map with respect to $V_B$ so that the evaluation is surjective. And,
it induces the bijection from $\pi_0 ( \Omega_{(W_{d^i})_{i \in
I^+}} )$ to a one point set $\pi_0 \big(
(\mathcal{A}_{\bar{v}_S^0})_{\psi_{\bar{v}_S^0}} \big).$ Therefore,
we obtain a proof. \qed
\end{proof}

\begin{proposition}
 \label{proposition: equivariant homotopy nonzero-dimensional fixed point}
Assume that $R$ is equal to one of $\SO(2),$ $\langle \SO(2), -b
\rangle.$ Then, Theorem \ref{main: only by isotropy} holds for these
cases.
\end{proposition}

\begin{proof}
Pick a $G_\chi$-vector bundle $V_B$ over $B$ such that $V_B|_q =
W_q$ for each pair $(W_q)_{q \in B} \in A_{G_\chi} (S^2, \chi)$ and
$q \in B.$ By Proposition \ref{proposition: proposition for
isomorphism}.(2), we only have to show that $\pi_0 (
\Omega_{(W_q)_{q \in B}} )$ consists of exactly one element for each
$(W_q)_{q \in B}.$ For those $R$'s, the $(G_\chi)_{v^0}$-bundle $F =
\big( \res_{(G_\chi)_{v^0}}^{G_\chi} F_{V_B} \big) |_{
\bar{\mathbf{x}} }$ satisfies Condition F1. by Table \ref{table: all
about K_R} and Definition \ref{definition: A_R}. This means that
$\mathcal{A}_{\bar{v}_S^0}^0$ is isomorphic to nonempty $\Iso_H (
V_S, V_N )$ by Proposition \ref{lemma: pointwise clutching for m=2
nontransitive}. Consider the evaluation map at $\bar{v}_S^0$
\begin{equation*}
\Omega_{(W_q)_{q \in B}} \rightarrow \mathcal{A}_{\bar{v}_S^0}^0, ~
\Phi \mapsto \Phi(\bar{v}_S^0).
\end{equation*}
In these cases, ${G_\chi}$ acts transitively only on $\bar{S}_S^1,$
but each map $\Phi$ in $\Omega_{V_B}$ is determined by $\Phi
(\bar{v}_S^0)$ through equivariance and inverse. It can be shown
that the evaluation is bijective as in Proposition \ref{proposition:
equivariant homotopy nonzero-dimensional no fixed point}. And, it
induces the bijective map from $\pi_0 ( \Omega_{(W_q)_{q \in B}} )$
to $\pi_0 \big( \mathcal{A}_{\bar{v}_S^0}^0 \big)$ which is
one-point set. Therefore, we obtain a proof. \qed
\end{proof}

\section{Equivariant line bundles over effective $G_\chi/H$-actions}
\label{section: reduction to effective}

In this section, we prove Theorem \ref{main: reduction to line
bundle} and calculate Chern classes. Since $H$ is the kernel of the
$G_\chi$-action on $S^2,$ $S^2$ delivers the $G_\chi/H$-action.
Since $\rho(G_\chi)=R$ for some $R$ of Table \ref{table:
introduction} by assumption, we may assume that $G_\chi/H$ is equal
to $R$ and the $G_\chi/H$-action is equal to the $R$-action on
$S^2.$

\begin{proof}[Proof of Theorem \ref{main: reduction to line bundle}]
We prove this only for the case when $R = \Z_n$ because other cases
are proved similarly. Let $U$ be the $H$-representation with the
character $\chi.$ Let $\bar{U}$ be a $(G_\chi)_q$-extension of $U$
for $q \in B$ whose existence is guaranteed by Theorem \ref{theorem:
extension by cyclic}. Pick a bundle $E$ in $\Vect_{G_\chi} (S^2,
\chi),$ and put $(W_q)_{q \in B} = p_\vect (E).$ Then, $W_q$'s are
direct sums of $\bar{U} \otimes \Omega(l)$'s by Corollary
\ref{corollary: expressed by extensions} because $(G_\chi)_q/H \cong
\Z_n$ for $q \in B.$ Pick arbitrary direct summands $\bar{U} \otimes
\Omega(l_0)$ and $\bar{U} \otimes \Omega(l_1)$ of $W_S$ and $W_N,$
respectively. Define $(W_q^\prime)_{q \in B}$ by
\begin{equation*}
W_S^\prime = \bar{U} \otimes \Omega(l_0) \quad \text{and} \quad
W_N^\prime = \bar{U} \otimes \Omega(l_1).
\end{equation*}
By definition of $A_{G_\chi} (S^2, \chi),$ the pair $(W_q^\prime)_{q
\in B}$ is contained in $A_{G_\chi} (S^2, \chi).$ Since
$p_\vect^{-1} \Big( (W_q^\prime)_{q \in B} \Big)$ is nonempty by
Theorem \ref{main: by isotropy and chern}, there exists a bundle $L$
with rank $\chi(\id)$ in $\Vect_{G_\chi} (S^2, \chi).$ Existence of
$L$ proves the isomorphism by \cite[Lemma 2.2]{CKMS}. By similar
arguments, we can show that $A_R (S^2, \id)$ is generated by all the
elements with one-dimensional entries. By using this and Theorem
\ref{main: by isotropy and chern}, \ref{main: not by isotropy and
chern class}, we can show that $\Vect_R (S^2)$ is generated by line
bundles.

Now, we calculate the number of elements in $A_R (S^2, \id)$ with
one-dimensional entries. By Definition \ref{definition: A_R}, each
pair $(W_q)_{q \in B}$ in $\Rep(R)^2$ is in $A_R (S^2, \id),$ i.e.
there is no relation between $W_S$ and $W_N.$ Since the number of
one-dimensional representations in $\Rep(R)$ is equal to $n = |R_S|
= |R_N|,$ we obtain a proof. \qed
\end{proof}

\begin{remark}
We explain for the reason why we prove the isomorphism of Theorem
\ref{main: reduction to line bundle} only for $R$'s appearing in
Theorem \ref{main: by isotropy and chern}, \ref{main: not by
isotropy and chern class}. In the proof of Theorem \ref{main:
reduction to line bundle}, existence of $(G_\chi)_q$-extensions of
$U$ for all $q$ or $(G_\chi)_{d^i}$-extensions of $U$ for all $i$ is
critical according to $R.$ But, such existence is not guaranteed if
$R_q$ or $R_{d^i}$ is isomorphic to $\D_{2m}$ for some $m$ by
\cite[Corollary 3.5.(2)]{CMS}, and almost all $R$'s appearing in
Theorem \ref{main: only by isotropy} satisfy that $R_q$ or $R_{d^i}$
is isomorphic to $\D_{2m}$ for some $q$ or $i$ according to $R.$ So,
we can not obtain the isomorphism for such $R$'s. The
inextensibility does not happen in dealing with equivariant vector
bundles over circle. \qed
\end{remark}

In the below, we use the notation $\mathbf{W}$ to denote an element
in $A_R (S^2, \id)$ with one-dimensional entries. We would calculate
Chern classes of line bundles in $\Vect_R (S^2).$ Especially, we
would calculate $k_0$ of Theorem \ref{main: by isotropy and chern}
which is dependent on $\mathbf{W}.$ Denote it by $k_0(\mathbf{W})$
to stress its dependency. By Theorem \ref{main: by isotropy and
chern}, $k_0(\mathbf{W})$ is determined up to $l_R \cdot \Z,$ i.e.
$k_0(\mathbf{W})$ lives in $\Z_{l_R}.$ More precisely, Theorem
\ref{main: by isotropy and chern} says that $k_0(\mathbf{W})$ is
congruent modulo $l_R$ to $c_1 (L)$ for any line bundle $L \in
p_\vect^{-1} (\mathbf{W}).$ So, we will calculate $c_1 (L)$
$(\modulo$ $l_R)$ for one bundle $L$ in $p_\vect^{-1} (\mathbf{W}).$
In doing so, $c_1 (L)$ is expressed by using $\mathbf{W} = (L_q)_{q
\in B}$ or $\mathbf{W} = (L_{d^i})_{i \in I^+}$ according to $R.$
When $n \in \N$ is understood, put $\xi_0 = \exp ( 2\pi \sqrt{-1} /
n ),$ and let $\Omega(l)$ for $l \in \Z_n$ be the one-dimensional
$\Z_n$-representation satisfying $a_n \cdot v = \xi_0^l v$ for each
$v \in \C.$ Then, we have the following well-known result:

\begin{lemma} \label{lemma: line Z_n-bundle}
Assume that $R=\Z_n.$ For any line bundle $L$ in $\Vect_R (S^2),$ if
$L_q \cong \Omega(l_q)$ for $q \in B,$ $l_q \in \Z_n,$ then $c_1 (L)
\equiv l_N-l_S ~ (\modulo ~ n).$
\end{lemma}

We obtain similar results for cases when $R = \D_n,$ $\Z_n \times Z$
with odd $n,$ or $\langle -a_n \rangle$ with even $n/2.$

\begin{lemma} \label{lemma: line D_n-bundle}
Assume that $R=\D_n.$ For any line bundle $L$ in $\Vect_R (S^2),$ if
$L_q \cong \Omega(l_q)$ for $q \in B,$ $l_q \in \Z_n,$ then $l_N
\equiv -l_S ~ (\modulo ~ n),$ and $c_1 (L)$ is congruent modulo $2n$
to
\begin{equation*}
\left\{
  \begin{array}{ll}
    -2l_S        &  \text{if } L_{d^0} \cong L_{d^1},   \\
    -2l_S + n    &  \text{if } L_{d^0} \ncong L_{d^1}.
  \end{array}
\right.
\end{equation*}
\end{lemma}

\begin{proof}
The first statement easily follows from the relation $b a_n b^{-1} =
a_n^{-1}.$ To prove the second statement, we would construct line
bundles $L^\prime$ in $\Vect_R (S^2)$ such that $L_q^\prime \cong
\Omega(l_q).$ Pick the $R$-bundle $V_B$ over $B$ such that
\begin{equation*}
a_n \cdot v = \xi_0^{l_q} v \quad \text{and} \quad b \cdot v = v
\end{equation*}
for $q \in B$ and $v \in V_q = \C.$ And, define $F_{V_B}$ as
$\amalg_{q \in B} ~ |\lineK_{R, q}| \times V_q$ such that
\begin{equation*}
g \cdot (\bar{x}, v) = (g \cdot \bar{x}, g \cdot v)
\end{equation*}
for $g \in R,$ $q \in B,$ $\bar{x} \in |\lineK_{R, q}|,$ $v \in
V_q.$ We calculate $\mathcal{A}_{\bar{x}}.$ Let $\varphi_0$ be the
element $\id$ in $\Iso(V_S, V_N) = \Iso(\C).$ Then, we can show the
following:
\begin{enumerate}
  \item $\mathcal{A}_{\bar{v}_S^0}^0 = \{ \varphi_0, ~ -\varphi_0 \}.$
  And, $\varphi_0,$ $-\varphi_0$ determine $\langle b \rangle$-representations
  $\Omega(0),$ $\Omega(1),$ respectively.
  \item $\mathcal{A}_{b(\bar{e}_S^0)}^0
  = \{ \xi_0^{-l_S} \varphi_0, ~ -\xi_0^{-l_S} \varphi_0 \}.$
  And, $\xi_0^{-l_S} \varphi_0,$ $-\xi_0^{-l_S} \varphi_0$
  determine $\langle a_n b \rangle$-representations
  $\Omega(0),$ $\Omega(1),$ respectively.
  \item $\mathcal{A}_{\bar{x}}^0 = \Iso(V_S, V_N) = \Iso(\C)$
  for each interior $\bar{x}$ of $[\bar{v}_S^0, b(\bar{e}_S^0)].$
\end{enumerate}
By using the parametrization on $|\lineL_{R, S}|,$ we can define two
equivariant clutching maps $\Phi$ and $\Phi^\prime$ with respect to
$V_B$ which satisfy
\begin{equation*}
\Phi(t) = \exp \Big( -\frac {4 \pi t ~ l_S \sqrt{-1} ~} {n} ~ \Big)
\varphi_0 \quad \text{and} \quad \Phi^\prime (t) = \exp \Big( -\frac
{4 \pi t ~ \big( l_S + \frac n 2 \big) \sqrt{-1} ~ } {n} ~ \Big)
\varphi_0
\end{equation*}
for $t \in [0, n].$ And, we can show the following:
\begin{enumerate}
  \item $L_{d^{-1}}^\prime \cong \Omega(l_S),$ $L_{d^0}^\prime \cong \Omega(0),$
  $L_{d^1}^\prime \cong \Omega(0),$ and $c_1(L^\prime) \equiv -2l_S ~ (\modulo ~ 2n)$
  for $L^\prime = F_{V_B}/\Phi,$
  \item $L_{d^{-1}}^\prime \cong \Omega(l_S),$ $L_{d^0}^\prime \cong \Omega(1),$
  $L_{d^1}^\prime \cong \Omega(1),$ and $c_1(L^\prime) \equiv -2l_S ~ (\modulo ~ 2n)$
  for $L^\prime = F_{V_B}/-\Phi,$
  \item $L_{d^{-1}}^\prime \cong \Omega(l_S),$ $L_{d^0}^\prime \cong \Omega(0),$
  $L_{d^1}^\prime \cong \Omega(1),$ and $c_1(L^\prime) \equiv -2l_S+n ~ (\modulo ~ 2n)$
  for $L^\prime = F_{V_B}/\Phi^\prime,$
  \item $L_{d^{-1}}^\prime \cong \Omega(l_S),$ $L_{d^0}^\prime \cong \Omega(1),$
  $L_{d^1}^\prime \cong \Omega(0),$ and $c_1(L^\prime) \equiv -2l_S+n ~ (\modulo ~ 2n)$
  for $L^\prime = F_{V_B}/-\Phi^\prime.$
\end{enumerate}
Images of these four line bundles through $p_\vect$ are four
$\mathbf{W}$'s in $A_R (S^2, \id)$ whose $W_{d^{-1}}$-entry is
isomorphic to $\Omega(l_S).$ Therefore, we obtain a proof. \qed
\end{proof}

\begin{remark}
We explain for how to calculate $k_0(\mathbf{W})$ for cases when $R
= \tetra,$ $\octa,$ $\icosa.$ Let $\mathbf{W} = p_\vect(L)$ for some
line bundle $L$ in $\Vect_R (S^2).$ Then, it suffices to calculate
$c_1(L) ~ (\modulo ~ l_R).$ Note $l_R = |R|$ in these cases. For a
$2$-Sylow subgroup $P$ of $R,$ observe that the restricted
$P$-action on $S^2$ is conjugate to $\D_m$ for some $m.$ Then, we
can calculate $c_1(L) ~ (\modulo ~ 2m)$ by applying Lemma
\ref{lemma: line D_n-bundle} to $\res_P^R L$ where $|\D_m| = 2m.$
For other prime number $p$ dividing $|R|$ and a $p$-Sylow subgroup
$P$ of $R,$ observe that the restricted $P$-action on $S^2$ is
conjugate to $\Z_p,$ and we can calculate $c_1(L) ~ (\modulo ~ p)$
by applying Lemma \ref{lemma: line Z_n-bundle} to $\res_P^R L$ where
$|\Z_p| = p.$ So, we can calculate $c_1(L) ~ (\modulo ~ l_R)$ by
Chinese Remainder Theorem because $l_R = |R|.$ \qed
\end{remark}

\begin{lemma} \label{lemma: line Z_nxZ-bundle}
Assume that $R = \Z_n \times Z$ with odd $n.$ For any line bundle
$L$ in $\Vect_R (S^2),$ if $L_q \cong \Omega(l_q)$ for $q \in B,$
$l_q \in \Z_n,$ then $l_N \equiv l_S ~ (\modulo ~ n),$ and $c_1 (L)$
is trivial.
\end{lemma}

\begin{proof}
The first statement easily follows because $a_n$ and $-\id$ commute
in $R.$ To prove the second statement, we would construct a line
bundle $L^\prime$ in $\Vect_R (S^2)$ such that $L_q^\prime \cong
\Omega(l_q).$ Pick the $R$-bundle $V_B$ over $B$ such that
\begin{equation*}
a_n \cdot v = \xi_0^{l_q} v \quad \text{and} \quad -\id \cdot v = v
\end{equation*}
for $q \in B$ and $v \in V_q = \C.$ And, define $F_{V_B}$ as
$\amalg_{q \in B} ~ |\lineK_{R, q}| \times V_q$ such that
\begin{equation*}
g \cdot (\bar{x}, v) = (g \cdot \bar{x}, g \cdot v)
\end{equation*}
for $g \in R,$ $q \in B,$ $\bar{x} \in |\lineK_{R, q}|,$ $v \in
V_q.$ Then, we can define the equivariant clutching map $\Phi$ with
respect to $V_B$ which satisfies $\Phi(\bar{x}) = \id \in \Iso(V_S,
V_N)=\Iso(\C)$ for each $\bar{x} \in \lineL_{R, S}.$ For $L^\prime =
F_{V_B}/\Phi,$ the Chern class $c_1(L^\prime)$ is trivial. Since two
equivariant vector bundles in $\Vect_R (S^2)$ with the same isotropy
representation at each $d^i$ have the same Chern class by Theorem
\ref{main: not by isotropy and chern class}, we obtain a proof. \qed
\end{proof}

\begin{lemma} \label{lemma: line <-a_n>-bundle}
Assume that $R = \langle -a_n \rangle$ with even $n/2.$ For any line
bundle $L$ in $\Vect_R (S^2),$ if $L_q \cong \Omega(l_q)$ for $q \in
B,$ $l_q \in \Z_{n/2},$ then $l_N \equiv l_S ~ (\modulo ~ n/2),$ and
$c_1 (L)$ is trivial.
\end{lemma}

\begin{proof}
Similarly to Lemma \ref{lemma: line D_n-bundle} and Lemma
\ref{lemma: line Z_nxZ-bundle}, the first statement can be proved by
using the fact that $R$ is a cyclic group. To prove the second
statement, we would construct a line bundle $L^\prime$ in $\Vect_R
(S^2)$ such that $L_q^\prime \cong \Omega(l_q).$ Put $c_0 = \exp
\big( \frac {2 \pi l_S \sqrt{-1} ~} {n} \big)$ so that $c_0^2 =
\xi_0^{l_S}$ where $\xi_0 = \exp \big( \frac {2 \pi \sqrt{-1} ~}
{n/2} \big).$ Pick the $R$-bundle $V_B$ over $B$ such that
\begin{equation*}
-a_n \cdot v = c_0 v
\end{equation*}
for $q \in B$ and $v \in V_q = \C.$ And, define $F_{V_B}$ as
$\amalg_{q \in B} ~ |\lineK_{R, q}| \times V_q$ such that
\begin{equation*}
g \cdot (\bar{x}, v) = (g \cdot \bar{x}, g \cdot v)
\end{equation*}
for $g \in R,$ $q \in B,$ $\bar{x} \in |\lineK_{R, q}|,$ $v \in
V_q.$ Then, the remaining is the same with Lemma \ref{lemma: line
Z_nxZ-bundle}. \qed
\end{proof}

\section{Appendix: representation extension}
\label{section: representation extension}

Let $N_0$ and $N_2$ be compact Lie groups such that $N_0 \lhd N_2$
and $N_2 / N_0 \cong \Z_m.$ Let $a_0$ be a fixed generator of
$N_2/N_0,$ and let $\Omega(l)$ be the representation defined by
\begin{equation*}
N_2 /N_0 \times \C \rightarrow \C, \quad (a_0, z) \mapsto \exp( 2
\pi l \sqrt{-1} / m ) z
\end{equation*}
for $l \in \Z_m.$ We also consider $\Omega(l)$ to be an
$N_2$-representation via the projection $N_2 \rightarrow N_2/N_0.$
Then, we obtain the following result from \cite{CMS}:

\begin{theorem} \label{theorem: extension by cyclic}
For $U \in \Irr(N_0),$ if the character of $U$ is fixed by $N_2,$
then there exists an $N_2$-extension of $U.$ If $\bar{U}$ is an
$N_2$-extension of $U,$ then the number of mutually nonisomorphic
$N_2$-extensions of $U$ is $m$ and they are $\bar{U} \otimes
\Omega(l)$ for $l \in \Z_m.$
\end{theorem}

\begin{proof}
By \cite[Theorem 3.2.]{CMS}, $U$ has $m$ mutually nonisomorphic
$N_2$-extensions. Call one of them $\bar{U}.$ By \cite[Proposition
3.1.]{CMS} and its proof, each extension of $U$ is expressed as
$\bar{U} \otimes \Omega(l)$ for some $l \in \Z_m.$ \qed
\end{proof}

\begin{corollary}
 \label{corollary: expressed by extensions}
Let $U$ be an irreducible $N_0$-representation whose character is
fixed by $N_2,$ and $\bar{U}$ be an $N_2$-extension of $U.$ If $W$
be an $N_2$-representation such that $\res_{N_0}^{N_2} W$ is
$U$-isotypical, $W$ is a direct sum of $\bar{U} \otimes
\Omega(l)$'s.
\end{corollary}

\begin{proof}
First, we prove that the induced representation $\ind_{N_0}^{N_2} U$
is isomorphic to the direct sum $\oplus_{l \in \Z_m} (\bar{U}
\otimes \Omega(l)).$ By Frobenius reciprocity, $\Hom_{N_2} ( \bar{U}
\otimes \Omega(l), \ind_{N_0}^{N_2} U ) \cong \Hom_{N_0} (
\res_{N_0}^{N_2} \bar{U} \otimes \Omega(l), U)$ is one-dimensional,
and this means that each $\bar{U} \otimes \Omega(l)$ for $l \in
\Z_m$ is a subrepresentation of $\ind_{N_0}^{N_2} U$ by Schur's
Lemma. So, $\ind_{N_0}^{N_2} U$ is isomorphic to the direct sum
$\oplus_{l \in \Z_m} (\bar{U} \otimes \Omega(l)).$

We may assume that $W$ is irreducible. We only have to show that $W$
is one of $\bar{U} \otimes \Omega(l)$'s. Since $\res_{N_0}^{N_2} W$
is $U$-isotypical, $\res_{N_0}^{N_2} W \cong lU$ for some integer
$l.$ By Frobenius reciprocity, $\Hom_{N_2} ( W, \ind_{N_0}^{N_2} U )
\cong \Hom_{N_0} ( \res_{N_0}^{N_2} W, U).$ Since $\res_{N_0}^{N_2}
W$ is isomorphic to $lU,$ we obtain that $\Hom_{N_0} (
\res_{N_0}^{N_2} W, U)$ is $l$-dimensional by Schur's Lemma. But,
since $\bar{U} \otimes \Omega(l)$'s are all different and $W$ is
irreducible, Schur's Lemma says that $\Hom_{N_2} ( W,
\ind_{N_0}^{N_2} U )$ is at most one-dimensional, i.e. $l \le 1.$
Therefore, $l$ is equal to 1, and this gives a proof. \qed
\end{proof}

\end{document}